\documentclass[11pt,a4paper]{amsart}

\usepackage{graphicx}      % include this line if your document contains figures
\usepackage{amsmath}
\usepackage{amssymb}
\usepackage{xcolor}
\usepackage{amsthm}
\usepackage{appendix}
\usepackage{epstopdf}
\usepackage{gensymb}
\usepackage{placeins}
\usepackage{graphicx}
\usepackage{mathtools}
\usepackage{subcaption}

\usepackage{cleveref}
\usepackage{enumerate}
\usepackage{graphicx}
\usepackage{pgfplots,tikz}
\usetikzlibrary{shapes,arrows,arrows.meta,decorations.markings,decorations.text}
\usepackage{mathtools}
\usepackage{algorithm,algorithmic}
\tikzstyle{arrow} = [draw, -latex']
\usepackage[centering,marginparwidth=2.5cm]{geometry}
\usepackage{float}
\usepackage{multirow}
\usepackage{amsopn}

\tikzset{>=latex}
\Crefname{ALC@unique}{Line}{Lines}

\newcommand{\vnorm}[1]{\left\lVert#1\right\rVert}
\newcommand{\norm}[1]{\vnorm{#1}}

\newtheorem{theorem}{Theorem}
\newtheorem{assumption}[theorem]{Assumption}
\theoremstyle{definition}
\newtheorem{definition}[theorem]{Definition}
\newtheorem{remark}[theorem]{Remark}

\crefname{figure}{Figure}{Figures}
\crefname{theorem}{Theorem}{Theorems}
\crefname{remark}{Remark}{Remarks}
\crefname{definition}{Definition}{Definitions}
\crefname{assumption}{Assumption}{Assumption}

\begin{document}
\title[Efficient MPC for PDEs with Goal Oriented Error Estimation]{Efficient Model Predictive Control for Parabolic PDEs with Goal Oriented Error Estimation}
\author[Lars Grüne, Manuel Schaller, and Anton Schiela]{Lars Grüne$^{1}$, Manuel Schaller$^{1,2}$, and Anton Schiela$^{1}$}
\thanks{}

\thanks{$^{1}$Universität Bayreuth, Institute of Mathematics, Germany}
\thanks{$^{2}$Technische Universit\"at Ilmemau, Institute of Mathematics, Germany (e-mail: manuel.schaller@tu-ilmenau.de).}

	\thanks{{\bf Acknowledgments:}	This work was supported by the German Research Foundation (DFG) under grant numbers GR 1569/17-1 and SCHI 1379/5-1.}

\begin{abstract}% Abstract of not more than 250 words.
We show how a posteriori goal oriented error estimation can be used to efficiently solve the subproblems occurring in a Model Predictive Control (MPC) algorithm. In MPC, only an initial part of a computed solution is implemented as a feedback, which motivates grid refinement particularly tailored to this context. To this end, we present a truncated cost functional as objective for goal oriented adaptivity and prove under stabilizability assumptions that error indicators decay exponentially outside the support of this quantity. This leads to very efficient time and space discretizations for MPC, which we will illustrate by means of various numerical examples.

\end{abstract}

\maketitle
\section{Introduction}
\noindent In this work, we present a posteriori goal oriented grid adaptivity as a method to efficiently solve the subproblems arising in a Model Predictive Control (MPC) algorithm. 
MPC is a feedback controller, in which the solution of an optimal control problem (OCP) on an infinite or indefinite time horizon is approximated by a series of optimal control problems on a finite horizon $T>0$. In every feedback loop, an initial part of the optimal control up to time $\tau>0$, where often $\tau \ll T$, is implemented as feedback. This procedure is depicted in Algorithm~\ref{alg::mpcabstract}.
\begin{algorithm}[H]
	\caption{Standard MPC Algorithm}
	\label{alg::mpcabstract}
	\begin{algorithmic}[1]
		\STATE{Given: Prediction horizon $0<T$, implementation horizon $0<\tau\leq T$, initial state $x_0$}
		\STATE{$k=0$}
		\WHILE{controller active}
		\STATE{Solve OCP on $[k\tau,T+k\tau]$ with initial datum $x_k$, save optimal control in $u$}
		\STATE{Implement $u_{\big|[k\tau,(k+1)\tau]}$ as feedback, measure/estimate resulting state and save in $x_{k+1}$}
		\STATE{$k = k+1$}
		\ENDWHILE
	\end{algorithmic}
\end{algorithm}
MPC is widely used in many applications, such as automotive engineering \cite{Hrovat2012}, electrical engineering \cite{Vazquez2014a}, agriculture \cite{Ding2018}, and chemical engineering \cite{Camacho2007,Qin2003}.
It can be rigorously shown that this procedure yields an approximation of the optimal control for the original problem, if, e.g., a turnpike property holds \cite{Gruene2016}. In a nutshell, the turnpike property states that solutions to OCPs stay close to an optimal steady state, the so-called turnpike, for the majority of the time. For an in-depth introduction to and analysis of MPC methods, the interested reader is referred to the monographs \cite{Gruene2016b,Rawlings2017}.

As only an initial part of a computed optimal control is used as a feedback, the computation only needs to be accurate on this initial part. To this end, goal oriented a posteriori error estimation, cf., e.g., \cite{Bangerth2003} for an overview of this subject, yields a technique for specialized grid refinement tailored to an MPC context. In a nutshell, the aim of goal oriented error estimation techniques is to refine the time and/or space grid to reduce the error in an arbitrary functional $I(x,u)$, the so called quantity of interest (QOI), in order to guarantee that 
\begin{align*}
I(x,u) - I(\tilde{x},\tilde{u}) < \text{tol},
\end{align*}
where $(x,u)$ is the optimal solution and $(\tilde{x},\tilde{u})$ a numerical approximation on a time and/or space grid. In the particular case of MPC, this methodology can be used to minimize the error of the MPC feedback and its influence on the state, meaning that $I(x,u)$ is a functional incorporating only $x_{\mid_{[0,\tau]}}$ and $u_{\mid_{[0,\tau]}}$. To this end, we present a truncated version of the cost functional as an objective for refinement that is specialized for MPC. 

%We briefly comment on adaptive time discretization if the solutions satisfy steady state turnpike behavior.
%If the optimal control problem is autonomous and satisfies a stabilizability and detectability condition. An important property of turnpike behavior is that the approaching and leaving arcs' lengths are independent of the time horizon. As in between these transient arcs the solution stays close to an equilibrium of the dynamics, any adaptive time discretization scheme will predominantly refine the time grid at the beginning and the end of the time interval to resolve dynamic parts. Further, the independence of leaving and approaching arc of the size of the interval suggests that in case of time adaptivity, the resulting computational cost will be almost independent of the length of the interval. This naturally suggests that in case of a steady state turnpike, an adaptive time discretization can be very efficient. Moreover, the turnpike property was also exploited in \cite{Trelat2015} to construct an efficient shooting algorithm. These considerations are however not applicable when moving to non-autonomous problems that do not possess a corresponding static optimization problem and hence no optimal equilibrium. One then has to rely on classical a posteriori grid refinement techniques to adaptively refine the time and space grid. 

The underlying feature that allows for efficient goal oriented methods is an exponential stability of the optimally controlled system. This property was analyzed for linear-quadratic infinite dimensional systems in \cite{Gruene2018c,Gruene2019} and for nonlinear parabolic problems in \cite{Gruene2021} and is very closely connected to the turnpike property. We briefly recall existing numerical approaches exploiting this stability in the literature. Turnpike behavior was used in \cite{Trelat2015} to construct an efficient shooting algorithm. In \cite{Gruene2018c,Shin2020}, the exponential stability of the optimal control problem was leveraged to construct a priori discretizations which are specialized for MPC. In \cite{Na2020}, a Schwarz decomposition method is presented, which exploits the exponential decay of perturbations to show improved convergence properties. However, in all these works, a priori discretizations were chosen, requiring an a priori knowledge of the rate of decay. In practical applications, this information is not at hand and one has to rely on a posteriori methods for grid refinement. 

Besides grid adaptivity, an alternative technique to reduce computational effort is proper orthogonal decomposition (POD) \cite{Graessle2018,Graessle2019a,Kunisch2001}. A combination of MPC with POD methods was presented recently in \cite{Gruene2019b,Mechelli2017}. However, to the best of our knowledge, goal oriented techniques, i.e., reduction of the discretization error in an arbitrary functional, have not been considered in the literature.

The main objective of this work is twofold. First, we prove under stabilizability assumptions that the continuous-time and discrete-time error indicators for a QOI localized in time decay exponentially outside the support of the QOI. Second, we illustrate the resulting performance gain in an MPC scheme and compare the closed-loop cost functional value to the standard case of refinement to reduce the error in the cost functional. In that context we will see that for autonomous, non-autonomous, linear and nonlinear problems, a truncated QOI yields a significant increase of the MPC controllers performance. The presented approach is shown to be efficient in the sense that for a fixed number of total degrees of freedom, the closed-loop cost of the MPC trajectory will be significantly lower when using a truncated cost functional for refinement as opposed to using the full cost functional.
\section{Setting and preliminaries}
\label{sec:intro}
In this section we define the parabolic optimal control problem and the corresponding optimality conditions. We further present the spatial and temporal discretization scheme and recall the basics of goal oriented error estimation for parabolic optimization problems.
\subsection{Optimal control problem and optimality conditions}
Suppose that $(V,\|\cdot\|_V)$ is a separable Banach space, $(H,\langle\cdot,\cdot\rangle)$ is a Hilbert space with norm $\|\cdot\|$, and $V\hookrightarrow H \cong H^*\hookrightarrow V^*$ forms a Gelfand triple, i.e., the embeddings are continuous and dense. A common choice would be, e.g, $V=H^1_0(\Omega)$ and $H=L_2(\Omega)$, where we denote by $\Omega\subset\mathbb{R}^n,\, n\in \{2,3\}$ the spatial domain with locally Lipschitz boundary $\partial \Omega$.  Further, we fix $T>0$ and denote
\begin{align*}
W([0,T]) := \{v\colon[0,T]\to V\,|\,v\in L_2(0,T;V), \dot{v}\in L_2(0,T;V^*)\},
\end{align*}
where by $\dot{v}$ we mean the weak time derivative, which we will sometimes also denote as $\frac{d}{dt}x$.
A well-known property is that $W([0,T])\hookrightarrow C(0,T;H)$, cf.\ \cite[Proposition 23.23]{Zeidler19902a}. For an in-depth treatment of parabolic equations in this setting, the reader is referred to the standard literature \cite{Wloka1987,Zeidler19902a}. The control space will be denoted by $U$, which we assume to be a Hilbert space with scalar product $\langle \cdot,\cdot\rangle_U$ and corresponding norm $\|\cdot\|_U$. We consider the optimal control problem
\begin{align}
\label{def:numerics:ocp}
\begin{split}
\min_{(x,u)}\, J(x,u)&:=\int_{0}^{T}\bar{J}(t,x(t),u(t))\,dt\\
\text{s.t. }\dot{x}(t) &= \bar{A}(x(t)) + \bar{B}u(t)+f(t) \quad  \text{f.a.e. } t\in [0,T],\\
\qquad x(0) &= x_0,
\end{split}
\end{align}
where $x_0\in H$, $f\in L_2(0,T;V^*)$ and for a.e.\ $t\in [0,T]$, we assume that $\bar{J}(t,\cdot,\cdot)$ is twice continuously Fr\'echet differentiable functional on $V\times U$. Analogously, we assume $J(\cdot)$ to be a twice continuously Fr\'echet differentiable functional on $L_2(0,T;V)\times L_2(0,T;U)$. The operator $\bar{B}:U\to V^*$ is assumed to be linear and continuous and $\bar{A} \colon V \to V^*$ is a twice continuously Fr\'echet differentiable operator. We will assume that the optimal control problem has a solution in $W([0,T])$, which, besides the classical lower semi-continuity and coercivity properties of the objective functional can be assured via existence of a continuous control to state map. For clarity of presentation, we assume that $J_{xu}=J_{ux}=0$, which is, e.g., the case for standard tracking type functionals.

%\MS{In the following we will tacitly identify $H$ and $U$ with their corresponding dual spaces via the Riesz isomorphism. The derivative of a nonlinear function, e.g., $\bar{A}(x)$ will be denoted by $\bar{A}'(x)$ and the partial derivative of a function depending on several arguments by subscript, e.g., $J_x(x,u)$ specifies the partial derivative of $J(x,u)$ with respect to its first argument.}

%In the linear case, i.e., if ${A}(x)={A}x$ this follows if $-A$ satisfies a G\aa rding inequality, cf.\ \cite[Theorem 3.4]{Schiela2013}:
%\begin{align*}
%\exists \omega \in \mathbb{R}, \alpha > 0 : \quad \alpha \|x\|_{L_2(0,T;V)}^2 \le -\langle Ax,x\rangle_{L_2(0,T;V^*) \times L_2(0,T;V)}+\omega \|x\|_{L_2(0,T;H)}^2.
%\end{align*}
%In the semilinear and quasilinear case, monotonicity properties of the nonlinearities play a key role in that matter, cf.\ \cite[Chapter 5]{Troeltzsch2010} and \cite{Raymond1999}, where global existence of solutions was ensured by sufficiently regular data $(Bu+f,x_0)$, such that the solution is bounded, i.e., $x\in L_\infty(0,T;L_\infty(\Omega))$. For a in-depth analysis of optimal control problems governed by quasilinear parabolic equations, the interested reader is referred to \cite{Bonifacius2018,Casas1995,Ladyzhenskaia1968,Meinlschmidt2017,Papageorgiou1992}.. 
We will now derive optimality conditions via the Lagrange formalism, i.e., we define for $(\lambda,\lambda_0)\in L_2(0,T;V)\times H$ the Lagrange function
\begin{align}
\label{defn:numerics:contlag}
%\begin{split}
L(x,u,(\lambda,\lambda_0))&:= J(x,u) + \langle \dot{x}-A(x)-Bu-f,\lambda\rangle_{L_2(0,T;V^*)\times L_2(0,T;V)} + \langle x(0)-x_0,\lambda_0\rangle,
%\\&-  \langle Bu+f,\lambda\rangle_{L_2(0,T;V^*) \times L_2(0,T;V)}.
%\end{split}
\end{align}
where
\begin{align*}
\langle A(x),\lambda\rangle_{L_2(0,T;V^*) \times L_2(0,T;V)} &:= \int_0^T \langle \bar{A}(x), \lambda\rangle_{V^* \times V}\,dt,\\
\langle Bu,\lambda\rangle_{L_2(0,T;V^*) \times L_2(0,T;V)} &:= \int_0^T \langle \bar{B}u, \lambda\rangle_{V^* \times V}\,dt.
\end{align*}
The corresponding optimality conditions read
\begin{align*}
L'(x,u,\lambda,\lambda_0)=
\begin{pmatrix}
J_x(x,u) + \left(\frac{d}{dt} - A'(x)\right)^*\lambda + \langle \lambda(0),\lambda_0\rangle\\
J_u(x,u)- B^*\lambda\\
x'-A(x) - Bu-f\\
x(0)-x_0
\end{pmatrix}
=0.
\end{align*}
As $J_x(x,u)\in L_2(0,T;V^*)$, we obtain $\lambda \in W([0,T])$ and $\lambda(0)=\lambda_0$, cf.\ \cite[Proposition 3.8]{Schiela2013}. Hence we may remove $\lambda_0$ from the list of arguments of $L$ and we may write equivalently using integration by parts in the first equation,
	\begin{align}
	\label{eq:numerics:optcond}
	L'(x,u,\lambda)=
	\begin{pmatrix}
	J_x(x,u) -\dot{\lambda}- A'(x)^*\lambda\\
	\lambda(T)\\
	J_u(x,u)- B^*\lambda\\
	x'-A(x) - Bu-f\\
	x(0)-x_0
	\end{pmatrix}
	=0.
	\end{align}

We note that one could straightforwardly incorporate a terminal state penalization ${J}_T(x(T))$ into the cost functional, which would result in a nonzero terminal condition  $\lambda(T)$ for the backward dynamics of the adjoint state $\lambda$, cf.\ the first two lines of \eqref{eq:numerics:optcond}. We will omit this terminal cost for ease of presentation.
Correspondingly, the second derivative of the Lagrange function is given by
\begin{align}
\label{eq:numerics:L''}
L''(x,u,\lambda) = 
\begin{pmatrix}
J_{xx}(x,u)-{A}''(x)^*\lambda& 0 & -\frac{d}{dt}-{A}'(x)^* \\
0&0&E_T\\
0 & J_{uu}(x,u) & -B^*\\
\frac{d}{dt} - {A}'(x)& -B & 0\\
E_0&0&0
\end{pmatrix},
\end{align}
where for $t\in [0,T]$, $E_t\colon C(0,T;H)\to H$ is the time evaluation operator defined by $E_tx=x(t)$.
\subsection{Discretization and goal oriented error estimation}
For the discretization of the infinite-dimensional problem we use a discontinuous Galerkin approach of order zero in time denoted by dG(0) and a continuous Galerkin approach of order one in space denoted by cG(1), cf.\ \cite{Meidner2008a,Meidner2007}. In the literature this combined approach is often referred to as dG(0)cG(1)-discretization. 
We will briefly recall the definition of this discretization and the corresponding a posteriori error goal oriented estimation. In the following we will abbreviate
\begin{align*}
\mathcal{W}:=W([0,T]),\qquad \mathcal{U}=L_2(0,T;U), \qquad \langle v,w\rangle_{I} := \int_I \langle v(t),w(t)\rangle_{V^*\times V}\,dt.
\end{align*}
%As is common in the underlying literature for discontinuous Galerkin methods, cf.\ \cite[Section 2.2]{Meidner2008a}, whenever we write $\langle v,w\rangle_H$ for $v\in V^*$, $w\in V$, we mean it in the sense of the continuous extension of $\langle \cdot,\cdot\rangle_H$ to $V^*\times V$ by the dense embeddings of the Gelfand triple, cf.\ \cite[p.\ 142]{Gajewski1974}. 
% Denoting by $i : V\to H$ the canonical embedding of $V$ into $H$, we get that 
%\begin{align*}
%\langle i^*(h),v\rangle_{V^* \times V} = \langle h,i(v) \rangle.
%\end{align*}
%for any $h\in H$ and $v\in V$. Thus, by the density of the embeddings in the Gelfand triple $V\hookrightarrow H\hookrightarrow V^*$, there is a continuous extension of $\langle \cdot,\cdot \rangle$ to $V^*\times V$.
\subsubsection*{Time discretization}
We split up the interval $[0,T] = \{0\}\cup I_1 \cup I_2\cup \cdots \cup I_M$ into subintervals $I_m = (t_{m-1},t_m]$ of corresponding size $k_m := t_m - t_{m-1}$ for $m\in \{1,\ldots,M\}$ and set $I_0 := \{0\}$, where $0=t_0<t_1\cdots<t_M = T$. We define the discrete-time spaces of piecewise constant in time ansatz functions by 
\begin{align*}
	\mathcal{W}_k &:= \{v_k \in L_2(0,T;H)\,|\, {v_k}_{\big|I_m} \in \mathcal{P}^0(I_m,V),\,\, m=1,\ldots,M,\, v_{k}(0)\in H\},\\
	\mathcal{U}_k &:= \{u_k \in L_2(0,T;U)\,|\,  {u_k}_{\big|I_m} \in \mathcal{P}^0(I_m,U),\,\, m=1,\ldots,M\},
	\end{align*}
	where $ \mathcal{P}^0(I_m,V)$ (respectively $\mathcal{P}^0(I_m,U)$) denotes the space of polynomials of degree zero (i.e., constant functions) defined on $I_m$ with values in $V$ (respectively $U$).
By continuity of elements in $\mathcal{W}=W([0,T])\hookrightarrow C(0,T;H)$, this forms a non-conforming ansatz space as elements of $\mathcal{W}_k$ are not necessarily continuous. However, despite the nonconformity, the important feature of Galerkin orthogonality of the difference of continuous and discrete solution to the test space is preserved, cf. \cite[Remark 5.2]{Meidner2008a}. To capture the possible discontinuities, we denote the right and left sided limits and the jump at time grid point $t_m$ for $v_k\in \mathcal{W}_k$ via 
\begin{align*}
v_{k,m}^+ := \lim_{t\to 0^+}v_k(t_m+t),\quad v_{k,m}^- := \lim_{t\to 0^+}v_k(t_m-t),\quad [v]_{k,m} := v_{k,m}^+ - v_{k,m}^-.
\end{align*}
Due to the nonconformity of the ansatz space, the Lagrange function defined in \eqref{defn:numerics:contlag} is not defined on $\mathcal{W}_k$. Thus, we define the discrete-time Lagrange function $L^k:\mathcal{W}_k\times \mathcal{U}_k \times \mathcal{W}_k \to \mathbb{R}$ by
\begin{align}
\label{defn:numerics:semidisclag}
\begin{split}
L^k(x_k,u_k,\lambda_k) := \sum_{m=1}^M \int_{t_{m-1}}^{t_m} \bar{J}(s,x_k,u_k)\,ds + \sum_{m=1}^M \big(\langle \dot{x}_k,\lambda_k\rangle_{I_m} &- \langle \bar{A}(x_k) - \bar{B}u_k-f,\lambda_k\rangle_{I_m}\big) \\+ \sum_{m=1}^{M}\langle [x_k]_{m-1},\lambda_{k,m-1}^+& \rangle +  \langle x_{k,0}^- -x_0,\lambda_{k,0}^-\rangle,
\end{split}
\end{align} where the jump terms $[x_k]_{m-1}$ capture possible discontinuities of the state. This Lagrange function is also well-defined for state and adjoint state belonging to the continuous function space $\mathcal{W}$ and on this space it coincides with the continuous Lagrangian defined in \eqref{defn:numerics:contlag}. For piecewise constant functions belonging to the space $\mathcal{W}_k$, the time derivative vanishes, whereas for functions continuous in time belonging to $\mathcal{W}$, the jump terms vanish. 

The discrete-time version for the state equation of \eqref{eq:numerics:optcond} reads
\begin{align}
\begin{split}
\label{eq:semidiscretestate}
\langle {L}^k_\lambda(x_k,u_k,\lambda_k),\varphi_k\rangle_{\mathcal{W}_k^*\times \mathcal{W}_k}=\sum_{m=1}^{M} \big(\langle \dot{x}_k,\varphi_k\rangle_{I_m} &- \langle \bar{A}(x_k) - \bar{B}u_k-f,\varphi_k \rangle_{I_m}\big)\\
+\sum_{m=1}^{M}\langle [x_k]_{m-1},\varphi_{k,m-1}^+& \rangle +  \langle x_{k,0}^- -x_0,\varphi_{k,0}^-\rangle = 0.
\end{split}
\end{align}
for  $\varphi_k\in \mathcal{W}_k$. Analogously, the discrete-time counterpart to the third equation of \eqref{eq:numerics:optcond}, is given by
\begin{align}
\label{eq:semidiscretegradient}
\langle {L}^k_u(x_k,u_k,\lambda_k),\varphi_k\rangle_{\mathcal{U}_k^*\times \mathcal{U}_k}=\sum_{m=1}^M \langle \bar{J}_u(\cdot,x_k,u_k) - \bar{B}^*\lambda_k,\varphi_k\rangle_{I_m} = 0
\end{align}
for $\varphi_k \in \mathcal{U}_k$.
Using integration by parts on each subinterval in the state equation \eqref{eq:semidiscretestate}, one can derive the adjoint equation as discrete-time counterpart to the first equation of \eqref{eq:numerics:optcond}, that is,
\begin{align}
\label{eq:semidiscreteadjoint}
\begin{split}
\langle{L}^k_x(x_k,u_k,\lambda_k),\varphi_k\rangle_{\mathcal{W}_k^*\times \mathcal{W}_k} =\sum_{m=1}^M \langle& \bar{J}_x(\cdot,x_k,u_k),\varphi_k\rangle_{I_m} +\sum_{m=1}^M \big(\langle -\dot{\lambda}_k - \bar{A}'(x_k)^*\lambda_k,\varphi_k\rangle_{I_m} \\&- ( [\lambda_k]_{m-1},\varphi^-_{k,m-1})\big) + \langle \lambda_{k,M}^-,\varphi_{k,M}^-\rangle=0
\end{split}
\end{align}
for all $\varphi_k\in \mathcal{W}_k$. The resulting time-stepping scheme is equivalent to an implicit Euler method if the temporal integrals are approximated via the box rule, cf.\ \cite[Section 3.4.1]{Meidner2008a} and thus inherits its A-stability.
\subsubsection*{Space discretization}
For spatial discretization we use piecewise affine linear continuous finite elements as presented in the standard literature \cite{braess1997,Ciarlet2002,Hackbusch1986}. To this end, we assign a regular triangulation $\mathcal{K}_h^m$ of the spatial domain $\Omega$ and corresponding conforming finite element spaces $V_h^m\subset V$ and ${U}^m_h\subset {U}$  to each interval $I_m$ and obtain the fully discrete spaces
\begin{align*}
\mathcal{W}_{kh} &:= \{v_{kh}\in L_2(0,T,H)\,|\, v_{kh\big|I_m}\in \mathcal{P}^0(I_m,V^m_h),\, m=1,\ldots,M,\, v_{kh}(0)\in V_h^0\},\\
\mathcal{U}_{kh} &:= \{u_{kh}\in L_2(0,T,{U})\,|\, u_{kh\big|I_m}\in \mathcal{P}^0(I_m,U^m_h),\, m=1,\ldots,M\}.
\end{align*}
Due to conformity of these spaces with respect to the discrete-time spaces, i.e., $\mathcal{W}_{kh}\subset \mathcal{W}_k$ and $\mathcal{U}_{kh}\subset \mathcal{U}_k$, the discrete-time Lagrangian \eqref{defn:numerics:semidisclag} is well-defined on $\mathcal{W}_{kh}\times \mathcal{U}_{kh}\times \mathcal{W}_{kh}$.

In order to allow full flexibility for the spatial adaptivity, it is possible that the triangulation $\mathcal{K}_h^m$ on the interval $I_m$ is different from the triangulation $\mathcal{K}_h^{m+1}$ on the interval $I_{m+1}$. In terms of numerical realization this leads to difficulties of efficiently evaluating the scalar product of basis elements of different time steps as needed for the assembly of the Euler step equations \eqref{eq:semidiscretestate} and \eqref{eq:semidiscreteadjoint}. In this work, we make use of the remedy presented in \cite{Schmich2008}, where the authors suggest the evaluation of scalar products on a common triangulation. 

We finally note that instead of an a priori control discretization as utilized here, a variational control discretization in the spirit of \cite{Hinze2005} can be included in the underlying a posteriori goal oriented error estimation methodology, cf.\ \cite[Section 3]{Meidner2007}. This variational discretization approach can lead to advantages regarding a priori estimates, in particular in the presence of control constraints. In this work we do not include control constraints, as in this case the necessary stability results for the analysis in \Cref{sec:secondary} and \Cref{sec:secondary_discrete} are not yet available. 
	If the control space $U$ is finite dimensional, i.e., e.g., $(Bu)(t)=\sum_{k=1}^K u_k(t)\chi_k$ with appropriate shape functions $\chi_k$, then the control space is fully discrete after time discretization. A discretization in space is not necessary in this case.
\subsubsection*{Goal oriented error estimation}
We will concisely introduce the concept of goal oriented error estimation for optimal control of parabolic PDEs based on the works \cite{Meidner2008a,Meidner2007,Meidner2008}. A comprehensive introduction to adaptive finite element methods for ODEs and PDEs with applications is given in the monograph \cite{Bangerth2003}. The main idea of goal oriented error estimation is to estimate and reduce the discretization error with respect to an arbitrary functional $I(x,u)$, called the quantity of interest (QOI). 

We follow the literature \cite{Meidner2008a,Meidner2007} and denote by $(x,u,\lambda)\in (\mathcal{W}\times \mathcal{U} \times \mathcal{W})$ a continuous solution of the extremal equations \eqref{eq:numerics:optcond}, by $(x_k,u_k,\lambda_k)\in (\mathcal{W}_k\times \mathcal{U}_k\times \mathcal{W}_k)$ and by $(x_{kh},u_{kh},\lambda_{kh})\in (\mathcal{W}_{kh}\times \mathcal{U}_{kh} \times \mathcal{W}_{kh})$ time and fully discrete solutions of the system described by \eqref{eq:semidiscretestate}, \eqref{eq:semidiscretegradient}, and \eqref{eq:semidiscreteadjoint}. The aim of goal oriented a posteriori error estimation is to derive error estimators $\eta_k$ and $\eta_h$ such that
\begin{align*}
I(x,u) - I(x_{kh},u_{kh})\approx \eta_k + \eta_h,
\end{align*}
where $\eta_k$ approximates the time discretization error and $\eta_h$ approximates the space discretization error. A detailed derivation of the estimators is performed in \cite[Chapter 6]{Meidner2008a} and \cite{Meidner2007}. We briefly recall the main steps for the convenience of the reader and for later use. For more details, the interested reader is referred to the references above. Note that as we chose classical a priori control discretization, the discretization error of the control variable is included in the terms $\eta_k$ and $\eta_h$. If one chose a different control discretization, one could denote this term separately, cf.\ \cite{Meidner2007}.

In order to obtain computable error estimators, besides the solution triple $\xi:=(x,u,\lambda)$, a second triple of variables $\chi:=(v,q,z)$ has to be considered, to which we will refer as secondary variables. If $I(x,u)$ does not involve point evaluations in time, these secondary variables solve on the continuous level the linear system
\begin{align}
\label{eq:secondaryvars_cont}
L''(\xi)\chi=({L}^k)''(\xi){\chi} &= -\begin{pmatrix}
I_x(x,u)\\
0\\
I_u(x,u)\\
0\\
0
\end{pmatrix}
&&\text{ in } \mathcal{W}^*\times H\times \mathcal{U}^*\times \mathcal{W}^*\times H ,
%\end{align}
\intertext{on the discrete-time level the system}
%\begin{align}
\label{eq:secondaryvars_timedisc}
({L}^k)''(\xi_{k}){\chi_{k}} &= -\begin{pmatrix}
I_x(x_{k},u_{k})\\
0\\
I_u(x_{k},u_{k})\\
0\\
0
\end{pmatrix}
&&\text{ in } \mathcal{W}_{k}^*\times H\times \mathcal{U}_{k}^*\times \mathcal{W}_{k}^*\times H ,
%\end{align}
\intertext{and on the fully discrete level the system}
%\begin{align}
\label{eq:secondaryvars_spacetimedisc}
({L}^k)''(\xi_{kh}){\chi_{kh}} &= -\begin{pmatrix}
I_x(x_{kh},u_{kh})\\
0\\
I_u(x_{kh},u_{kh})\\
0\\
0
\end{pmatrix}
&&\text{ in } \mathcal{W}_{kh}^*\times V_h^0\times  \mathcal{U}_{kh}^*\times  \mathcal{W}_{kh}^* \times V_h^M.
\end{align}
These equations are similar to the defining equation of a Lagrange-Newton step, where the derivative of the Lagrangian on the right-hand side is replaced by the derivative of the QOI.

With the continuous triples $\xi = (x,u,\lambda)$ and $\chi = (v,q,z)$ and the corresponding discrete counterparts at hand we define the residual of the first order optimality condition via
\begin{align*}
\rho^\lambda(x,u,\lambda)\varphi &:= \langle {L}^k_x(x,u,\lambda),\varphi\rangle_{\mathcal{W}_k^*\times \mathcal{W}_k},\\
\rho^u(x,u,\lambda)\varphi &:= \langle {L}^k_u(x,u,\lambda),\varphi\rangle_{\mathcal{U}_k^*\times \mathcal{U}_k},\\
\rho^x(x,u,\lambda)\varphi &:= \langle {L}^k_\lambda(x,u,\lambda),\varphi\rangle_{\mathcal{W}_k^*\times \mathcal{W}_k},
\end{align*}
and a residual involving the secondary variables $\chi = (v,q,z)$ via 
\begin{align*}
\rho^z(\xi,v,q,z)\varphi &:=  {L}^k_{\lambda x}(\xi)(z,\varphi)  +  {L}^k_{ux}(\xi)(q,\varphi)+ {L}^k_{xx}(\xi)(v,\varphi)+I_x(x,u)\varphi,\\
\rho^q(\xi,v,q,z)\varphi &:= {L}^k_{uu}(\xi)(q,\varphi) + {L}^k_{xu}(\xi)(v,\varphi) + {L}^k_{\lambda u}(\xi)(z,\varphi) + I_u(x,u)\varphi,\\
\rho^v(\xi,v,q)\varphi &:=  {L}^k_{x\lambda }(\xi)(v,\varphi) + {L}^k_{u\lambda}(\xi)(q,\varphi).
\end{align*}
With these residuals one obtains the time error representation formula
	\begin{align}
	\label{eq:timeerr}
	\begin{split}
	&I(x,u) - I(x_k,u_k)=\\  & \frac{1}{2} \big(\rho^\lambda(x_k,u_k,\lambda_k)(v-\underline{v}_k) + \rho^u(x_k,u_k,\lambda_k)(q-\underline{q}_k) + \rho^x(x_k,u_k)(z-\underline{z}_k) \\
	&+ \rho^z(\xi_k,v_k,q_k,z_k)(x-\underline{x}_k) +\rho^q(\xi_k,v_k,q_k,z_k)(u-\underline{u}_k)+ \rho^v(\xi_k,v_k,q_k)(\lambda-\underline{\lambda}_k)\big)\\
	&+ \mathcal{R}_k
	\end{split}
	%\end{align}
	\intertext{for $(\underline{v}_k,\underline{q}_k,\underline{z}_k),(\underline{x}_k,\underline{u}_k,\underline{\lambda}_k) \in \mathcal{W}_k\times \mathcal{U}_k \times \mathcal{W}_k$ arbitrary and the space error representation formula}
	%\begin{align}
	\label{eq:spaceerr}
	\begin{split}
	&I(x_k,u_k) - I(x_{kh},u_{kh})  =\\
	&\qquad \qquad\frac{1}{2} \big(\rho^\lambda(x_{kh},u_{kh},\lambda_{kh})(v_k-\underline{v}_{kh}) + \rho^u(x_{kh},u_{kh},\lambda_{kh})(q_k-\underline{q}_{kh})\\ 
	&\qquad\qquad+ \rho^x(x_{kh},u_{kh})(z_k-\underline{z}_{kh}) 
	+ \rho^z(\xi_{kh},v_{kh},q_{kh},z_{kh})(x_k-\underline{x}_{kh})\\ 
	&\qquad\qquad+\rho^q(\xi_{kh},v_{kh},q_{kh},z_{kh})(u_k-\underline{u}_{kh})+ \rho^v(\xi_{kh},v_{kh},q_{kh})(\lambda_k-\underline{\lambda}_{kh})\big)\\
	&\qquad\qquad+ \mathcal{R}_h
	\end{split}
	\end{align}
	for $(\underline{v}_{kh},\underline{q}_{kh},\underline{z}_{kh}),(\underline{x}_{kh},\underline{u}_{kh},\underline{\lambda}_{kh}) \in \mathcal{W}_{kh}\times \mathcal{U}_{kh} \times \mathcal{W}_{kh}$. The remainder terms $\mathcal{R}_k$ and $\mathcal{R}_h$ depend on the third derivative of the (exterior) Lagrangian and are cubic in the error in primal and dual variables, i.e., $\mathcal{R}_k$ is cubic in $x-x_k$, $u-u_k$ etc.\ and $\mathcal{R}_h$ is cubic in $x_k-x_{kh}$, $u_k-u_{kh}$ and so forth. The arbitrary choice of the test functions %$(\bar{v}_k,\bar{q}_k,\bar{z}_k),(\bar{y}_k,\bar{u}_k,\bar{p}_k),(\bar{v}_{kh},\bar{q}_{kh},\bar{z}_{kh})$ and $(\bar{y}_{kh},\bar{u}_{kh},\bar{p}_{kh})$ 
originates from Galerkin orthogonality, cf.\ \cite[Proposition 4.1, Theorem 4.3]{Meidner2007}. The terms $v-\underline{v}_k$, $q-\underline{q}_k$, $z-\underline{z}_k$, $x-\underline{x}_k$, $u-\underline{u}_k$ and $\lambda-\underline{\lambda}_k$ are often called weights and need to be approximated to obtain computable error estimates as the solutions in the infinite-dimensional spaces, i.e., variables with no subscript or subscript $k$, are not at hand. Such an approximation can be performed by, e.g., interpolation or a higher order method. In this work we resort to the latter and adapt the method presented in \cite{Weiser2013} utilizing hierarchical error estimation. Having approximated the weights, we denote by $\eta_k$ and $\eta_h$ the approximations of \cref{eq:timeerr} and \cref{eq:spaceerr}, respectively. For simplicity, we localize the error indicators via the cell-wise contributions. For more involved localization methods, the reader is referred to \cite{Becker2001, Braak2003,Meidner2008} and \cite[Section 4.3]{Richter2015}.

\section{Exponential decay of continuous-time error indicators}
\label{sec:secondary}
Having introduced the concept of goal oriented error estimation, we will present a quantity of interest particularly well-suited for the adaptive solution of the optimal control problems in a Model Predictive Controller.
In every iteration of the MPC loop, the control on $[0,\tau]$ is used as feedback. Hence, we suggest using a truncation of the cost functional as a quantity of interest, namely
\begin{align}
\label{eq:mpcqoi}
I^{\tau}(x,u) := \int\limits_0^\tau \bar{J}(t,x,u)\,dt.
\end{align}
This specialized quantity of interest in goal oriented error estimation yields time and space grids, such that the error of the MPC feedback is small. In this part, we will rigorously prove that the error indicators $\eta_k$ and $\eta_h$ approximating the errors \eqref{eq:timeerr} and \eqref{eq:spaceerr} for the QOI \eqref{eq:mpcqoi} decay exponentially outside the interval $[0,\tau]$. First we observe that by the linear dependence, the error indicators inherit the behavior of the secondary variables. Thus, it suffices to analyze the behavior of the continuous version of these variables, i.e., $\chi=(v,q,z)$ defined in \eqref{eq:secondaryvars_cont} or the discrete-time version $\chi_k=(v_k,q_k,z_k)$ defined in \eqref{eq:secondaryvars_timedisc}. In these defining equations, we observe that the right-hand side depends on the derivatives of the QOI. In case of a QOI as defined in \eqref{eq:mpcqoi} these functionals only integrate over a small part of the time horizon if $\tau \ll T$. In the following sensitivity analysis, we will show that $\chi=(v,q,z)$ defined in \eqref{eq:secondaryvars_cont} respectively the discrete-time secondary variables $\xi_k=(v_k,q_k,z_k)$ defined in \eqref{eq:secondaryvars_timedisc} inherit the locality of the QOI in the sense that they are large on $[0,\tau]$ and small on $[\tau,T]$.

To derive sensitivity estimates, we make the following assumptions, where $(x,u)$ is an optimal solution to \cref{def:numerics:ocp}.
\begin{assumption}\hfill
	\label{as:main}
	\begin{itemize}
		\item There is a Banach space $(Y,\|\cdot\|_Y)$ and an operator $C\in L(L_2(0,T;V),L_2(0,T;Y))$ such that $L_{xx}(x,u)=C^*C$.
		\item There is an operator $R\in L(L_2(0,T;U),L_2(0,T;U))$ satisfying the coercivity condition $ \|Ru\|_{L_2(0,T;U)}\geq \alpha \|u\|_{L_2(0,T;U)}$ for $\alpha > 0$ such that $J_{uu}(x,u)=R^*R$.
		\item $A'(x)\in L(L_2(0,T;V), L_2(0,T;V^*))$.
		\item $(A'(x),C)$ is exponentially detectable in the sense that there is a feedback operator $K_C\in L(L_2(0,T;Y),L_2(0,T;V^*))$ and $\alpha > 0$ such that \begin{align*}
\hspace{1cm}-\langle (A'(x)+K_C C)v,v\rangle_{L_2(0,T;V^*) \times L_2(0,T;V)} \geq \alpha \|v\|_{L_2(0,T;V)}^2 \quad \text{for all}\quad v\in L_2(0,T;V)
		\end{align*}
		\item $(A'(x),B)$ is exponentially stabilizable in the sense that there is a feedback operator $K_B\in L(L_2(0,T;V),L_2(0,T;U))$ and $\alpha > 0$ such that \begin{align*}
			\hspace{1cm}-\langle (A'(x)+BK_B)v,v\rangle_{L_2(0,T;V^*) \times L_2(0,T;V)} \geq \alpha \|v\|_{L_2(0,T;V)}^2 \quad \text{for all} \quad v \in L_2(0,T;V).
		\end{align*}
		\item $\|A'(x)\|_{L(L_2(0,T;V),L_2(0,T;V^*))}, \|B\|_{L(L_2(0,T;U),L_2(0,T;V^*))},\|C\|_{L(L_2(0,T;V),L_2(0,T;Y))}$, and \\$\|R\|_{L(L_2(0,T;U))}$ can be bounded independently of $T$.
	\end{itemize}
\end{assumption}
The stabilizability notion of the last two assumptions was introduced in \cite[Definition 3.6]{Gruene2018c}. It allows for a straightforward derivation of stability estimates in $W([0,T])$ and is easy to verify via, e.g., generalized Poincar\'e or Friedrichs inequalities, cf.\ \cite[Example 3.7 and Example 3.8]{Gruene2018c}. In particular, a consequence of the above assumptions is the unique solvability of the linearized optimality conditions at the optimal solution, cf.\ \cite[Proof of Corollary 3.16]{Gruene2018c}.
\begin{theorem}
	\label{thm:numerics:scaling_cont}
	Let \cref{as:main} hold. Consider the QOI $I^\tau(x,u)$ defined in \eqref{eq:mpcqoi}. Let $\chi =(v,q,z)\in W([0,T])\times L_2(0,T;{U})\times W([0,T])$ solve \eqref{eq:secondaryvars_cont}, i.e.,
	\begin{align*}
	{L}''(x,u,\lambda)\chi = \begin{pmatrix}
	{L}_{xx}(x,u)&0&-\frac{d}{dt}-A'(x)^*\\
	0&0&E_T\\
	0&J_{uu}(x,u)&-B^*\\
	\frac{d}{dt}-A'(x)&-B&0\\
	E_0&0&0
	\end{pmatrix}
	\begin{pmatrix}
	v\\q\\z
	\end{pmatrix}
	= -\begin{pmatrix}
	I^\tau_x(x,u)\\
	0\\
	I^\tau_u(x,u)\\
	0\\
	0
	\end{pmatrix}.
	\end{align*}
	Then, defining
	\begin{align*}
	M:=
	\begin{pmatrix}
	C^*C&-\frac{d}{dt}-A'(x)^*\\
	0&E_T\\
	\frac{d}{dt}-A'(x)&-BJ_{uu}(x,u)^{-1}B^*\\
	E_0&0
	\end{pmatrix}
	\end{align*}
	the solution operator norm $\|M^{-1}\|_{L((L_2(0,T;V^*)\times H)^2, W([0,T])^2)}$ can be bounded independently of $T$. Further, for all $\mu>0$ satisfying 
	\begin{align*}
	\mu < \frac{1}{\|M^{-1}\|_{L((L_2(0,T;V^*)\times H)^2, W([0,T])^2)}}
	\end{align*}
	there is a constant $c(\tau)> 0$ independent of $T$ such that
	\begin{align}
	\label{eq:numerics:upperbound}
	\left\|e^{\mu t} \begin{pmatrix}
	v\\q\\z
	\end{pmatrix}\right\|_{W([0,T])\times L_2(0,T;{U}) \times W([0,T])}\leq c(\tau)\big(\|{J}_x(x,u)\|_{L_2(0,\tau;V^*)}+ \|{J}_u(x,u)\|_{L_2(0,\tau;U)}\big).
	\end{align}
\end{theorem}
\begin{proof}
	We first rewrite the system after elimination of the secondary control via $q=\!J_{uu}^{-1}(x,u)(B^*z- I^\tau_u(x,u))$ as
	\begin{align*}
	\begin{pmatrix}
	C^*C&-\frac{d}{dt}-A'(x)^*\\
	0&E_T\\
	\frac{d}{dt}-A'(x)&-BJ_{uu}(x,u)^{-1}B^*\\
	E_0&0
	\end{pmatrix}
	\begin{pmatrix}
	v\\z
	\end{pmatrix}
	= -\begin{pmatrix}
	I^\tau_x(x,u)\\
	0\\
	BJ_{uu}(x,u)^{-1}I^\tau_u(x,u)\\
	0.
	\end{pmatrix}.
	\end{align*}
	The bound on $M^{-1}$ follows by \cite[Corollary 3.16]{Gruene2018c}. To obtain the bound on the scaled variables, we set $s(t)=e^{\mu t}$ for $\mu>0$. Setting $\varepsilon:=-\left(I^\tau_x(x,u),
	0,
	BJ_{uu}(x,u)^{-1}I^\tau_u(x,u),
	0\right)\in (L_2(0,T;V^*)\times H)^2$, a straightforward computation and application of the product rule yields
	\begin{align*}
	M\begin{pmatrix}
	v\\z
	\end{pmatrix} &= \varepsilon\\
	(M-\mu P)\left(s\begin{pmatrix}
	v\\z
	\end{pmatrix}\right)&=s\varepsilon\\
	(I-\mu M^{-1}P)\left(s\begin{pmatrix}
	v\\z
	\end{pmatrix}\right)&=M^{-1}s\varepsilon
	\end{align*}
	where $P:=\begin{psmallmatrix}
	0&-I\\
	0&0\\
	I&0\\
	0&0
	\end{psmallmatrix}$. Thus, choosing $\mu < \frac{1}{\|M^{-1}\|_{L((L_2(0,T;V^*)\times H)^2, W([0,T])^2)}}$ and setting \\$\beta = \mu \|M^{-1}\|_{L((L_2(0,T;V^*)\times H)^2, W([0,T])^2)}<1$, a standard Neumann series argument, cf.\ \cite[Theorem 2.14]{Kress1989} yields, 
	\begin{align*}
	\left\|s\begin{pmatrix}
	v\\z
	\end{pmatrix}\right\|_{W([0,T])^2} \leq \frac{\|M^{-1}\|_{L((L_2(0,T;V^*)\times H)^2, W([0,T])^2)}}{1-\beta}\|s\varepsilon\|_{(L_2(0,T;V^*)\times H)^2}.
	\end{align*}
	Further estimating the right-hand side we obtain
	\begin{align*}
	&\norm{\int_0^\tau e^{\mu t} \bar{J}_x(t,x,u) \cdot \,dt}_{L_2(0,T;V^*)} + \norm{BJ_{uu}(x,u)^{-1}\int_0^\tau e^{\mu t}\bar{J}_u(t,x,u) \cdot \,dt}_{L_2(0,T;U)}
	\\&\leq e^{\mu \tau}\left(1+\|BJ_{uu}^{-1}\|_{L(L_2(0,T;U),L_2(0,T;V^*))}\right)\left(\|{J}_x(x,u)\|_{L_2(0,\tau;V^*)} + \|{J}_u(x,u)\|_{L_2(0,\tau;U)}\right),
	\end{align*}
	which concludes the proof.
\end{proof}
We will derive a similar estimate in \Cref{sec:secondary_discrete} for the discrete-time secondary variables $v_k$, $q_k$ and $z_k$ and the fully discrete secondary variables $v_{kh}$, $q_{kh}$ and $z_{kh}$.
Further, we will show in \cref{rem:dependenceonT} that the term on the right-hand side, i.e., $\|{J}_x(x,u)\|_{L_2(0,\tau;V^*)} + \|{J}_u(x,u)\|_{L_2(0,\tau;U)}$ is bounded independently of the time horizon $T$ if a turnpike property holds. We give a short interpretation of the estimate \eqref{eq:numerics:upperbound}: As the scaling $e^{\mu t}$ grows exponentially in time, the variables $(v,q,z)$ have to decay exponentially in time such that the product is bounded (almost) independently of the end time $T$. Thus, the secondary variables $(v,q,z)$ inherit the locality of the QOI $I^\tau(x,u)$ in the sense that they are also localized on $[0,\tau]$. Due to the linear dependence, this also carries over to the error indicators approximating \eqref{eq:timeerr} and \eqref{eq:spaceerr}. 
\begin{remark}
	\label{rem:dependenceonT}
	We will briefly give sufficient conditions under which the upper bound in \eqref{eq:numerics:upperbound} can be shown to be bounded independently of $T$ in the case of a linear quadratic problem. It turns out that when a turnpike property holds, the initial part of the optimal solution is only affected by the horizon negligibly, if the horizon is large. 
	Consider a time horizon  $T>0$  and the linear quadratic optimal control problem
	\begin{align*}
	\min_{(x,u)} \frac{1}{2} \int_0^T \|C(x(t)-x_\text{d})\|_Y^2 + \|R(u(t)-u_\text{d})\|_U^2 \,dt \qquad s.t.\ \dot{x}=Ax+ Bu, \quad x(0)=x_0.
	\end{align*} 
	Suppose that the involved operators satisfy the stabilizability assumptions of \cref{as:main}. Then it follows by \cite[Corollary 5.3]{Gruene2018c} that the state and control satisfy the turnpike estimate
	\begin{align}
	\label{eq:tpestimate}
	\norm{\begin{pmatrix}
		x(t)-\bar{x},
		u(t)-\bar{u}
		\end{pmatrix}}_{H\times U}\leq c(e^{-\mu t} + e^{-\mu (T-t)})\left(\|\bar{\lambda}\|+\|\bar{x}-x_0\|\right)
	\end{align}
	for a.e.\ $t\in [0,T]$, where $(\bar{x},\bar{u})$ denotes the optimal solution of the corresponding steady state problem, $\bar{\lambda}$ is the corresponding adjoint state and $c\geq 0$ is independent of $T$. Hence, in particular we have
	\begin{align}
	\label{eq:pwestimate}
	\norm{\begin{pmatrix}
		x(t),
		u(t)
		\end{pmatrix}}_{H\times U} \leq c
	\end{align}
	for a.e.\ $t\in [0,T]$, i.e., boundedness independently of $T$. Thus, computing 
	\begin{align*}
	\|J_x(x,u)\|_{L_2(0,\tau;V^*)}&= \sup_{\|v\|_{L_2(0,T;V)}=1} \int_0^\tau \langle C^*C(x(t)-x_d),v(t)\rangle \,dt \\&\leq \sup_{\|v\|_{L_2(0,T;V)}=1} \|C\|^2_{L(X,Y)} \int_0^\tau \|x(t)-x_d\|^2\,dt  \underbrace{\int_0^\tau \|v(t)\|^2\,dt }_{\leq \|v\|^2_{L_2(0,T;V)=1}}
	\end{align*}
	and analogously for the control we conclude
	\begin{align*}
	\int_0^\tau \|x(t)-x_\text{d}\|+\|u(t)-u_\text{d}\|_U\,dt \leq \tau c_1 \|(x,u)\|_{C(0,T;X)\times L_\infty(0,\tau;U)} + c_2
	\end{align*}
	with $c_1,c_2\geq 0 $ independent of $T$. Hence, together with \eqref{eq:pwestimate},
	\begin{align*}
	\|{J}_x(x,u)\|_{L_2(0,\tau;V^*)} + \|{J}_u(x,u)\|_{L_2(0,\tau;U)} = \int_0^\tau \langle C^*C(x(t)-x_d),\cdot\rangle  + \langle R^*R(u(t)-u_d),\cdot\rangle \,dt \leq c
	\end{align*}
	with $c\geq 0$ independent of $T$.
	Finally we note that the steady state turnpike assumed of \eqref{eq:tpestimate} can be replaced by a dynamic turnpike concept and the proof remains valid. In particular, for time-varying problems in discrete time a similar property was proven in \cite[Theorem 3]{Gruene}.
\end{remark}
\section{Exponential decay of discrete-time error indicators}
\label{sec:secondary_discrete}
The result of \cref{thm:numerics:scaling_cont} does not immediately carry over to the discrete-time secondary variables as defined in \eqref{eq:secondaryvars_timedisc} due to the nonconformity of the discrete-time ansatz space. Thus, we will give a separate proof of this matter in the following. To this end, we introduce a suitable function space for scaled functions of $\mathcal{W}_k$ which are not necessarily piecewise constant in time. This will be important as we will deal with scaled piecewise constant functions in the discrete-time counterpart of \cref{thm:numerics:scaling_cont}.
\begin{definition}
	\label{defn:A}
	We define the space	of functions that are weakly differentiable on every subinterval via
	\begin{align*}
	W^M([0,T]) := \{v\in L_2(0,T;V)\,|\,v_{\big| I_m} \in W([t_{m-1},t_m]),\, m=1,\ldots, M,\,v(0)\in H\}
	\end{align*}
	and endow it with the natural norm
	\begin{align*}
	\|v\|_{W^M([0,T])} = \sum_{m=1}^{M} \left(\|v\|_{W([t_{m-1},t_{m}])} + \|v_{m-1}^- - v_{m-1}^+\|\right)+ \|v(0)\|.
	\end{align*}
	Additionally, we define linear operators ${\Lambda}^k,\Lambda^{k,-}\colon W^M([0,T])\to W^M([0,T])^*$ via the relations
	\begin{align*}
	\langle\Lambda^kv,\varphi&\rangle_{W^M([0,T])^* \times W^M([0,T])} := \\
	&\sum_{m=1}^{M} \left(\langle \dot{v},\varphi\rangle_{I_m}+ \langle[v]_{m-1},\varphi_{m-1}^+ \rangle\right) -  \langle {A}'(x)v,\varphi\rangle_{L_2(0,T;V^*) \times L_2(0,T;V)} + \langle v_{0}^-,\varphi_{0}^-\rangle\\
	\langle \Lambda^{k,-}v,&\varphi\rangle_{W^M([0,T])^* \times W^M([0,T])} := \\
	&-\sum_{m=1}^{M} \left(\langle \dot{v},\varphi\rangle_{I_m}+\langle [v]_{m-1},\varphi_{m-1}^- \rangle\right) -  \langle A'(x)^*v,\varphi\rangle_{L_2(0,T;V^*) \times L_2(0,T;V)}  +  \langle v_{M}^-,\varphi_{M}^-\rangle.
	\end{align*}
\end{definition}
Note that $\|\cdot\|_{W^M([0,T])}$ is equivalent to $\|\cdot\|_{W([0,T])}$ on $W([0,T])$. Further it is clear that $\mathcal{W}_k\hookrightarrow W^M([0,T])$ and that $W^M([0,T])$ with the norm defined above is a Banach space. Testing of the initial respectively terminal condition is included in the operators $\Lambda^k$ respectively $\Lambda^{k,-}$ due to the terms $\langle v_0^-,\varphi_0^-\rangle$ and $\langle v_M^-,\varphi_M^-\rangle$.
We first employ a $T$-independent invertibility result for the discrete-time operator occurring in \eqref{eq:secondaryvars_timedisc}.
To this end, we note that $L_{xx}=L^k_{xx}$ and $L_{uu}=L^k_{uu}$, i.e., the second derivatives with respect to the state and control of the continuous and discrete-time Lagrange function coincide. This is because the time derivative and the jump terms enter the Lagrange function in a linear way, i.e., they vanish in the second derivative.
\begin{theorem}
	\label{thm:operatornorm}
	If \cref{as:main} holds, then the inverse of the operator
	\begin{align*}
	M^k:=\begin{pmatrix}
	L_{xx}(x,u)&\Lambda^{k,-}\\
	\Lambda^k&-BJ_{uu}(x,u)^{-1}B^*
	\end{pmatrix}
	\end{align*}
	can be bounded by
	\begin{align*}
	\|(M^k)^{-1}\|_{L((L_2(0,T;V^*)\times H)^2, W^M([0,T])^2)} \leq c,
	\end{align*}
	where $c\geq 0$ is a constant independent of $T$.
\end{theorem}
\begin{proof}
	As $L_{xx}(x,u)=C^*C$ by \cref{as:main}, we consider the system
	\begin{align}
	\label{eq:system}
	\begin{pmatrix}
	C^*C&\Lambda^{k,-}\\
	\Lambda^k&-BJ_{uu}(x,u)^{-1}B^*\\
	\end{pmatrix}
	\begin{pmatrix}
	v\\ z
	\end{pmatrix}
	=\begin{pmatrix}
	(l_1,z_T)\\(l_2,v_0)
	\end{pmatrix}\in (L_2(0,T;V^*)\times H)^2
	\end{align}
	for $l_1,l_2\in L_2(0,T;V^*)$ and $z_T,v_0 \in H$.
	First, we test the state equation, i.e., the second equation of \eqref{eq:system} with $(v,v_0)$ and obtain
	\begin{align*}
	\sum_{m=1}^M \left(\langle \dot{v},v\rangle_{I_m} + \langle [v]_{m-1},v^+_{m-1}\rangle\right) &+  \|v_0^-\|^2  - \langle A'(x)v+BJ_{uu}(x,u)^{-1}B^*z,v\rangle_{L_2(0,T;V^*) \times L_2(0,T;V)} \\&=\langle l_2,v\rangle_{L_2(0,T;V^*) \times L_2(0,T;V)} + \|v_0\|^2
	\end{align*}
	and compute for the first three terms with the formula $\langle \dot{v},v\rangle_{L_2(0,T;V^*)\times L_2(0,T;V)}=\frac12(\|v(T)\|^2-\|v(0)\|^2)$ for any $v\in W([0,T])$, cf.\ \cite[Proposition 23.23]{Zeidler19902a} and the definition of the jump terms $[v]_{m}:=v_m^+-v_m^-$ applied on every subinterval that
	\begin{align*}
	\sum_{m=1}^M &\left(\langle \dot{v},v\rangle_{I_m} + \langle [v]_{m-1},v^+_{m-1}\rangle\right) +  \|v_0^-\|^2 \\ &= \sum_{m=1}^M\left(\frac{1}{2}\|v_m^-\|^2 - \frac{1}{2}\|v_{m-1}^+\|^2 +\|v_{m-1}^+\|^2-\langle v_{m-1}^-,v_{m-1}^+\rangle\right) + \|v_0^-\|^2\\
	&= \sum_{m=1}^{M}\left(\frac{1}{2}\|v_m^-\|^2 - \langle v_{m-1}^-,v_{m-1}^+\rangle + \frac{1}{2}\|v_{m-1}^+\|^2\right) + \|v_0^-\|^2\\
	&=\sum_{m=1}^M \left(\frac{1}{2}\|v_{m-1}^-\|^2 - \langle v_{m-1}^-,v_{m-1}^+\rangle + \frac{1}{2}\|v_{m-1}^+\|^2\right) +  \frac{1}{2}\left(\|v_M^-\|^2 + \|v_0^-\|^2\right)\\
	&=\sum_{m=1}^M \frac{1}{2}\|v_{m-1}^- -v_{m-1}^+\|^2 +  \frac{1}{2}\left(\|v_M^-\|^2 + \|v_0^-\|^2\right).
	\end{align*}
	Estimating 
		\begin{align*}
		\langle BJ_{uu}(x,u)^{-1}B^*z,v\rangle_{L_2(0,T;V^*)}  \leq \|BJ_{uu}(x,u)^{-1}\|_{L(L_2(0,T;U),L_2(0,T;V^*))}\|B^*z\|_{L_2(0,T;U)}\|v\|_{L_2(0,T;U)}
		\end{align*} and adding the stabilizing feedback $K_C$ from \cref{as:main} we obtain
	\begin{align*}
	\sum_{m=1}^M &\frac{1}{2}\|v_{m-1}^- -v_{m-1}^+\|^2 +  \frac{1}{2}\left(\|v_M^-\|^2 + \|v_0^-\|^2\right) - \langle (A'(x)+K_CC)v,v\rangle_{L_2(0,T;V^*)\times L_2(0,T;V)} \\&\leq c\left(\|C v\|_{L_2(0,T;Y)}+ \|B^*z\|_{L_2(0,T;U)} + \|l_2\|_{L_2(0,T;V^*)}\right)\|v\|_{L_2(0,T;V)} + \|v_0\|^2.
	\end{align*}
	Hence, by $L_2(0,T;V)$-ellipticity of $-(A'(x)+K_CC)$, we get
	\begin{align}
	\label{eq:vest}
	\begin{split}
	\sum_{m=1}^M &\frac{1}{2}\|v_{m-1}^- -v_{m-1}^+\|^2 +  \frac{1}{2}\left(\|v_M^-\|^2 + \|v_0^-\|^2\right) + \|v\|^2_{L_2(0,T;V)} \\&\leq c\left(\|C v\|_{L_2(0,T;Y)}^2 + \|B^*z\|_{L_2(0,T;U)}^2 + \|l_2\|_{L_2(0,T;V^*)}^2 + \|v_0\|^2\right).
	\end{split}
	\end{align}
	Analogously, we test the adjoint equation with $z$ and compute
	\begin{align*}
	-\sum_{m=1}^M&\left( \langle \dot{z},z\rangle_{I_m} + \langle [z]_{m-1},z_{m-1}^-\rangle \right) + \|z_M^-\|^2\\
	&=-\sum_{m=1}^M \left(\frac{1}{2}\|z_m^-\|^2 - \frac{1}{2}\|z_{m-1}^+\|^2 +\langle z_{m-1}^+,z_{m-1}^-\rangle - \|z_{m-1}^-\|^2\right) + \|z_M^-\|^2\\
	&= \|z_0^-\|^2 + \sum_{m=1}^{M}\left(\frac{1}{2}\|z_m^-\|^2  - \langle z_{m-1}^+,z_{m-1}^-\rangle + \frac{1}{2}\|z_{m-1}^+\|^2 \right)\\
	&= \frac{1}{2}\left(\|z_0^-\|^2 + \|z_M^-\|^2\right) + \sum_{m=1}^M \frac{1}{2}\|z_{m-1}^- - z_{m-1}^+\|^2
	\end{align*}
	and thus, analogously to the state using stabilizability of $(A'(x),B)$ in the sense of \cref{as:main} we get for the adjoint that
	\begin{align}
	\label{eq:zest}
	\begin{split}
	\sum_{m=1}^M \frac{1}{2}\|z_{m-1}^- -z_{m-1}^+\|^2 &+  \frac{1}{2}\left(\|z_M^-\|^2 + \|z_0^-\|^2\right) + \|z\|^2_{L_2(0,T;V)} \\&\leq c(\|Cv\|_{L_2(0,T;Y)}^2 + \|B^*z\|_{L_2(0,T;U)}^2 + \|l_1\|_{L_2(0,T;V^*)}^2+\|z_T\|^2).
	\end{split}
	\end{align}
	It remains to estimate the term $\|Cv\|_{L_2(0,T;Y)}^2 + \|B^*z\|_{L_2(0,T;U)}^2$. To this end, we test the first equation of \eqref{eq:system} with $v$, the second equation of \eqref{eq:system} with $z$, subtract the latter from the former, use $J_{uu}=R^*R$ and invertibility of $R$ and obtain
	\begin{align}
	\nonumber
	&\|Cv\|_{L_2(0,T;Y)}^2 + \|B^*z\|_{L_2(0,T;U)}^2\\\label{eq:CvBzest}&\qquad \leq |\langle\Lambda^{k,-}z,v\rangle_{W^M([0,T])^*\times W^M([0,T]) }-\langle\Lambda^{k}v,z\rangle_{W^M([0,T])^*\times W^M([0,T]) }| \\\nonumber &+ (\|(l_1,z_T)\|_{L_2(0,T;V^*)\times H}+ \|(l_2,v_0)\|_{L_2(0,T;V^*)\times H})\left(\|v\|_{L_2(0,T;V)}+\|v_0^-\|+\|z\|_{L_2(0,T;V)}+\|z_M^-\|\right)
	\end{align}
	We proceed to show that $\langle\Lambda^{k,-}z,v\rangle_{W^M([0,T])^*\times W^M([0,T]) } = \langle \Lambda^kv,z\rangle_{W^M([0,T])^*\times W^M([0,T]) }$.
	\begin{align*}
	&\langle\Lambda^{k,-}z,v\rangle_{W^M([0,T])^*\times W^M([0,T])}\\
	&=-\sum_{m=1}^{M} \left(\langle z',v\rangle_{I_m}+\langle [z]_{m-1},v_{m-1}^- \rangle\right) +  \langle A'(x)^*z,v\rangle_{L_2(0,T;V^*)\times L_2(0,T;V)}  +  \langle z_{M}^-,v_{M}^-\rangle\\
	&=\sum_{m=1}^{M} \left(\langle z,v'\rangle_{I_m} -\langle z_m^-,v_m^-\rangle +\langle z_{m-1}^+,v_{m-1}^+\rangle -\langle z_{m-1}^+-z_{m-1}^-,v_{m-1}^- \rangle\right) \\&\qquad\qquad\qquad\qquad \qquad +  \langle A'(x)v,z\rangle_{L_2(0,T;V^*)\times L_2(0,T;V)}  +  \langle z_{M}^-,v_{M}^-\rangle\\
	&= \sum_{m=1}^{M} \left(\langle z,v'\rangle_{I_m}+ \langle z_{m-1}^+,v_{m-1}^+-v_{m-1}^-\rangle\right) + \langle A'(x)v,z\rangle_{L_2(0,T;V^*)\times L_2(0,T;V)} + \langle z_0^-,v_{0}^-\rangle\\ &= \langle \Lambda^kv,z\rangle_{W^M([0,T])^*\times W^M([0,T])}.
	\end{align*}
	The interested reader is referred to a similar result in \cite[Proposition 3.6]{Schiela2013} in a continuous-time setting.	Thus, together with \eqref{eq:vest}, \eqref{eq:zest}, and \eqref{eq:CvBzest} we obtain with $c\geq 0$ independent of $T$ that
	\begin{align}
	\label{eq:productspaceest}
	\begin{split}
	\sum_{m=1}^M &\frac{1}{2}\|v_{m-1}^- -v_{m-1}^+\|^2 +  \frac{1}{2}\left(\|v_M^-\|^2 + \|v_0^-\|^2\right) + \sum_{m=1}^M \frac{1}{2}\|z_{m-1}^- -z_{m-1}^+\|^2 \\&+  \frac{1}{2}\left(\|z_M^-\|^2 + \|z_0^-\|^2\right)+\|(v,z)\|^2_{L_2(0,T;V)^2} \leq c\|(l_1,z_T,l_2,v_0)\|^2_{(L_2(0,T;V^*)\times H)^2}.
	\end{split}
	\end{align}
	To obtain an estimate on the derivatives, we test the state equation with a test function $\varphi_m\in C^\infty([t_{m-1},t_m];V)$ such that $\varphi(t_{m-1})=\varphi(t_m)=0$ and obtain
	\begin{align*}
	\sum_{m=1}^{M} \langle v',\varphi \rangle_{I_m}  = \langle BJ_{uu}^{-1}B^*z+l_2+A'(x)v,\varphi \rangle_{L_2(0,T;V^*)\times L_2(0,T;V)}.
	\end{align*}
	By density of $C_0^\infty([t_{m-1},t_{m}];V)$ in $L_2(t_{m-1},t_{m};V)$, cf.\ \cite[Lemma 2.1]{Schiela2013}, we conclude the estimate
	\begin{align*}
	&\|\dot{v}\|_{L_2(t_{m-1},t_{m};V^*)} \leq \big(\|A'(x)\|_{L(L_2(t_{m-1},t_{m};V),L_2(t_{m-1},t_{m};
		V^*))}\\&+\|BJ_{uu}^{-1}B^*\|_{L(L_2(t_{m-1},t_{m};V),L_2(t_{m-1},t_{m};V^*))} \big)\|(v,z)\|_{L_2(t_{m-1},t_{m};V)^2} + \|l_2\|_{L_2(t_{m-1},t_{m};V^*)},
	\end{align*}
	which together with \eqref{eq:productspaceest} and proceeding analogously for the adjoint yields the result.
\end{proof}
We now obtain an analogous result to \cref{thm:numerics:scaling_cont} for the discrete-time system. 
\begin{theorem}
	\label{thm:scaling}	
	Let the assumptions of \cref{thm:operatornorm} hold and consider the QOI $I^\tau(x,u)$ defined in \eqref{eq:mpcqoi}. Let $(v_k,q_k,z_k)\in W^M([0,T])\times L_2(0,T;{U})\times W^M([0,T])$ solve \eqref{eq:secondaryvars_timedisc}, i.e.,
	\begin{align}
	\begin{pmatrix}
	{L}_{xx}(x,u)&0&\Lambda^{k,-}\\
	0&J_{uu}(x,u)&-B^*\\
	\Lambda^k&-B&0
	\end{pmatrix}
	\begin{pmatrix}
	v_k\\q_k\\z_k
	\end{pmatrix}
	= -\begin{pmatrix}
	I^\tau_x(x,u)\\
	I^\tau_u(x,u)\\
	0
	\end{pmatrix}.
	\end{align}
	Then for all $\mu>0$ satisfying 
	\begin{align*}
	\mu < \frac{1}{\|(M^k)^{-1}\|_{L(L_2(0,T;V^*)\times H)^2, W^M([0,T])^2)}}
	\end{align*}
	there is a constant $c(\tau)> 0$ independent of $T$ such that
	\begin{align}
	\label{eq:auxiliaryestimate}
	\norm{e^{\mu t} \begin{pmatrix}
		v_k\\q_k\\z_k
		\end{pmatrix}}_{W^M([0,T])\times L_2(0,T;{U}) \times W^M([0,T])} \leq c(\tau)\left(\|J_x(x,u)\|_{L_2(0,\tau;V^*)} + \|J_u(x,u)\|_{L_2(0,\tau;U)}\right).
	\end{align}
\end{theorem}
\begin{proof}
	We first rewrite the system by eliminating the control via
	$q_k = J_{uu}^{-1}(x,u)(B^*z_k - I_u(x,u))$  as
	\begin{align*}
	\underbrace{\begin{pmatrix}
		C^*C&\Lambda^{k,-}\\
		\Lambda^k&-BJ_{uu}(x,u)^{-1}B^*\\
		\end{pmatrix}}_{=M^k}
	\begin{pmatrix}
	v_k\\z_k
	\end{pmatrix}
	= -\begin{pmatrix}
	I^\tau_x(x,u)\\
	BJ_{uu}(x,u)^{-1}I^\tau_u(x,u)
	\end{pmatrix}.
	\end{align*}
	We further choose $\mu < \frac{1}{\|(M^k)^{-1}\|_{L((L_2(0,T;V^*)\times H)^2 , W^M([0,T])^2})}$ independently of $T$, cf.\ \cref{thm:operatornorm}, introduce scaled variables $\tilde v_k = e^{\mu t}v_k$ and $\tilde z_k = e^{\mu t}z_k$ and compute that \begin{align*}
	&\langle\Lambda^kv_k,\varphi\rangle_{W^M([0,T])^* \times W^M([0,T])} =\langle \Lambda^k(e^{-\mu t}\tilde{v}_k),\varphi\rangle_{W^M([0,T])^* \times W^M([0,T])}\\ 
	&= \sum_{m=1}^{M} \langle \tfrac{d}{dt}(e^{-\mu t}\tilde{v}_k),\varphi\rangle_{I_m} +
	\langle [e^{-\mu t}\tilde{v}_k]_{m-1},\varphi_{m-1}^+ \rangle -\langle {A}'(x)e^{-\mu t}\tilde{v}_k,\varphi \rangle_{L_2(0,T;V^*)\times L_2(0,T;V)}  \\&\qquad \qquad \qquad \qquad  +\langle (e^{-\mu t}\tilde{v}_k)^-_{0},\varphi_{0}^-\rangle\\
	&= \sum_{m=1}^{M} \langle \dot{\tilde{v}}_k-\mu \tilde{v}_k, e^{-\mu t}\varphi\rangle_{I_m} +\langle [\tilde{v}_k]_{m-1},e^{-\mu t}\varphi_{m-1}^+ \rangle -\langle A'(x)\tilde{v}_k,e^{-\mu t}\varphi \rangle_{L_2(0,T;V^*)\times L_2(0,T;V)} \\&\qquad \qquad \qquad \qquad +  \langle (\tilde{v}_k)^-_{0},(e^{-\mu t}\varphi)_{0}^-\rangle\\
	&= \langle (\Lambda^k-\mu I)\tilde{v}_k,\tilde{\varphi}\rangle_{W^M([0,T])^* \times W^M([0,T])}
	\end{align*}
	where $\tilde{\varphi} = e^{-\mu t}\varphi$. Proceeding analogously for the adjoint equation we get in the scaled variables
	\begin{align*}
	(M^k+\mu P)
	\begin{pmatrix}
	\tilde{v}_k\\\tilde{z}_k
	\end{pmatrix}
	= -\begin{pmatrix}
	\int_0^\tau e^{\mu t}\bar{J}_x(t,x,u) \cdot \,dt\\
	BJ_{uu}(x,u)^{-1}\int_0^\tau e^{\mu t}\bar{J}_u(t,x,u) \cdot \,dt
	\end{pmatrix}
	\end{align*}
	where $P = \begin{pmatrix}
	0&I\\-I&0
	\end{pmatrix}$ and hence $\|P\|_{L(W^M([0,T])^2, (L_2(0,T;V^*)\times H)^2)}<1$.
	%	 The scaled right-hand side can be estimated via \begin{align*}
	%		\|\int_0^\tau &e^{\mu t}\bar{J}_x(t,x,u)\,dt \|_{L_2(0,T;V^*)} + \|BJ_{uu}(x,u)^{-1}\int_0^\tau e^{\mu t}{J}_u(x,u) \,dt\|_{L_2(0,T;V^*)}\\&\leq c(\tau)\left(\|J_x(x,u)\|_{L_2(0,\tau,V)^*} + \|J_u(x,u)\|_{L_2(0,\tau;U)^*}\right).
	%	\end{align*}
	We multiply the equation by $(M^k)^{-1}$ and employ a Neumann-series argument as $\mu < \frac{1}{\|(M^k)^{-1}\|_{L((L_2(0,T;V^*)\times H), W^M([0,T])^2)}}$ and obtain
	\begin{align*}\norm{
		\begin{pmatrix}
		v_k\\z_k
		\end{pmatrix}}_{W^M([0,T])^2} 
	&\leq \|(I+\mu (M^k)^{-1} P)^{-1}\|_{L(W^M([0,T])^2, W^M([0,T])^2)} \\& \hspace{-0.9cm} \|(M^k)^{-1}\|_{L((L_2(0,T;V^*)\times H)^2, W^M([0,T])^2)}\norm{\begin{pmatrix}
		\int_0^\tau e^{\mu t}{J}_x(x,u) \cdot \,dt\\
		BJ_{uu}(x,u)^{-1}\int_0^\tau e^{\mu t}{J}_u(x,u) \cdot\,dt
		\end{pmatrix}}_{L_2(0,T;V^*)^2}
	\\&\leq c(\tau)\left(\|J_x(x,u)\|_{L_2(0,\tau;V^*)} + \|J_u(x,u)\|_{L_2(0,\tau;U)}\right)
	\end{align*}
	with a constant $c> 0$ independent of $T$. For the control, we compute 
	\begin{align*}
	\|q_k\|_{L_2(0,T;{U})} &= \|J_{uu}^{-1}(x,u)B^*z_k + J_{uu}(x,u)^{-1}I_u(x,u)\|_{L_2(0,T;{U})} \\& \leq c(\tau)\left(\|J_x(x,u)\|_{L_2(0,\tau;V^*)} + \|J_u(x,u)\|_{L_2(0,\tau;U)}\right),
	\end{align*}
	which concludes the proof.
\end{proof}

We will briefly illustrate the exponential decay of the secondary variables proven in \cref{thm:numerics:scaling_cont,thm:scaling} for a linear quadratic problem. More precisely, we consider the cost functional
$\int_0^T \bar{J}(t,x,u)\,dt := \frac{1}{2}\int_0^T\|x(t)-x_\text{d}(t)\|_{L_2(\Omega)}^2+ \alpha\|u(t)\|^2_{L_2(\Omega)} \,dt$
with $\alpha > 0$ and dynamics governed by a linear heat equation with distributed control that will be specified in \Cref{subsec:numerics:linquad}.
The plots in \cref{fig:sec_long} show the exponential decay of the linearly interpolated discrete-time secondary variables for the QOI $I^\tau(x,u)$ defined in \eqref{eq:mpcqoi}. In \cref{fig:sec_long}, we observe that for all values of the Tikhonov parameter $\alpha$, the state and the control decay exponentially after the time $\tau=0.5$. The ledges in the plot are introduced by the tolerance of the linear solver used for solution of the linear system \eqref{eq:secondaryvars_spacetimedisc}. The smaller we choose $\alpha$, the faster the secondary variables decay in time. This is because $\|M^{-1}\|_{L((L_2(0,T;V^*)\times H)^2, W([0,T])^2)}$ is proportional to $\alpha$, cf.\ \cite[Section 3]{Gruene2018c}. Thus, decreasing $\alpha$ allows for a larger choice of the scaling parameter $\mu>0$ in \cref{thm:numerics:scaling_cont} due to the bound
\begin{align*}
\mu < \frac{1}{\|M^{-1}\|_{L((L_2(0,T;V^*)\times H)^2, W([0,T])^2)}}.
\end{align*}
This straightforwardly carries over to the discrete-time setting considered in \cref{thm:scaling,thm:operatornorm}.
%Further, e.g., in \cite[Section 5]{Gruene2018c}, the decay parameter $\mu$ in turnpike results is chosen in the same fashion in the turnpike results, i.e., e.g., , also the speed of exponential convergence to the turnpike is increased when decreasing $\alpha$, which confirms the findings of \cite[Table 3.1, p.\ 41]{Altmueller2014} where a lower stabilizing horizon for MPC could be chosen when decreasing the Tikhonov parameter $\alpha$. 
\begin{figure}[H]
	\centering
	\scalebox{.9}{% This file was created by matlab2tikz.
\begin{tikzpicture}

\begin{axis}[%
width=0.37\linewidth,
height=0.7in,
at={(0.4\linewidth,0in)},
scale only axis,
xmin=-0,
xmax=10,
yticklabel pos = right,
ylabel near ticks,
xlabel style={font=\color{white!15!black}},
xlabel={time $t$},
ymode=log,
ymin=1e-14,
ymax=1,
yminorticks=true,
ylabel style={font=\color{white!15!black}},
ylabel={$\sqrt{\alpha}\|q(t)\|_{L_2(\Omega)}$},
axis background/.style={fill=white},
%cycle list name = black white,
%mark repeat = 5
]
\addplot [dotted, line width = 1pt]
  table[row sep=crcr]{%
0	0\\
0.25	0.208574967261133\\
0.5	0.426078263869741\\
0.75	0.11510600394628\\
1	0.055049154163384\\
1.25	0.026365010514271\\
1.5	0.012641082747365\\
1.75	0.00606626875989471\\
2	0.00291316581035661\\
2.25	0.00139975991040048\\
2.5	0.000672889122313811\\
2.75	0.000323605824079745\\
3	0.000155689394717476\\
3.25	7.49213031442472e-05\\
3.5	3.60473680880635e-05\\
3.75	1.73303620741128e-05\\
4	8.32366674711618e-06\\
4.25	3.99744036753561e-06\\
4.5	1.9227755641445e-06\\
4.75	9.25049458602718e-07\\
5	4.39025335064529e-07\\
5.25	1.98165818441251e-07\\
5.5	8.31336273844678e-08\\
5.75	4.28207134164403e-08\\
6	3.29643507721274e-08\\
6.25	2.08740205727482e-08\\
6.5	6.95852466163625e-09\\
6.75	1.32299768603153e-08\\
7	2.47791116300353e-08\\
7.25	3.25813576300067e-08\\
7.5	3.51810519613976e-08\\
7.75	3.18587045177255e-08\\
8	2.44145424517742e-08\\
8.25	1.94694385973735e-08\\
8.5	2.36459366525118e-08\\
8.75	3.09014997925218e-08\\
9	3.51538615030609e-08\\
9.25	3.48880898009767e-08\\
9.5	3.04342944926861e-08\\
9.75	2.25348942775728e-08\\
10	1.20130012401216e-08\\
};
%\addlegendentry{$\alpha = 10^{-1}$}

\addplot [dashdotted, line width=1.5pt]
  table[row sep=crcr]{%
0	0\\
0.25	0.00186863210906469\\
0.5	0.0339564698488431\\
0.75	0.00622735191508289\\
1	0.000292674578651684\\
1.25	1.46846381610132e-05\\
1.5	7.69760658115784e-07\\
1.75	4.05042554015714e-08\\
2	3.48033503384648e-09\\
2.25	1.41119323833388e-09\\
2.5	1.14552597636131e-09\\
2.75	1.10601592860703e-09\\
3	1.15354348077008e-09\\
3.25	6.5837119967462e-10\\
3.5	6.61317363184396e-10\\
3.75	7.33265883543966e-10\\
4	6.76844548111823e-10\\
4.25	1.35323511900584e-09\\
4.5	1.00989244678471e-09\\
4.75	2.46361868453722e-10\\
5	9.53991322739584e-10\\
5.25	1.00798200220358e-09\\
5.5	5.16143921248455e-10\\
5.75	4.80142864607103e-10\\
6	5.9301843611728e-10\\
6.25	5.04861243630337e-10\\
6.5	2.5456050005994e-10\\
6.75	4.58926559996694e-10\\
7	3.64312318008637e-10\\
7.25	2.21911041993395e-10\\
7.5	5.31625538017889e-10\\
7.75	4.22792386388399e-10\\
8	5.75845136614758e-10\\
8.25	4.45220009611687e-10\\
8.5	3.87311363856083e-10\\
8.75	5.90103084490152e-10\\
9	1.52000732759054e-10\\
9.25	4.47213984284657e-10\\
9.5	2.79670976352199e-10\\
9.75	3.22184662218313e-10\\
10	3.33064618870789e-10\\
};
%\addlegendentry{$\alpha =10^{-3}$}

\addplot [solid, line width=1.5pt]
  table[row sep=crcr]{%
0	0\\
0.25	2.49938885008902e-06\\
0.5	0.000490756612683045\\
0.75	1.82664895101476e-05\\
1	4.97594013987674e-08\\
1.25	2.1483001429118e-10\\
1.5	2.31192181288709e-11\\
1.75	2.0033616073704e-11\\
2	1.77321779891605e-11\\
2.25	1.12916042461763e-11\\
2.5	1.305664245541e-11\\
2.75	5.97537501083781e-12\\
3	7.21417943703685e-12\\
3.25	8.8434056276122e-12\\
3.5	7.11681023209349e-12\\
3.75	5.5702135944266e-12\\
4	5.84769103852975e-12\\
4.25	5.20107834375086e-12\\
4.5	6.38557717695111e-12\\
4.75	3.95710460765047e-12\\
5	3.29511976639114e-12\\
5.25	3.79759450129896e-12\\
5.5	4.59467040289902e-12\\
5.75	4.09491272165669e-12\\
6	3.64968592475368e-12\\
6.25	3.7916773432622e-12\\
6.5	2.6491786265631e-12\\
6.75	3.0606245094297e-12\\
7	3.76992247122436e-12\\
7.25	2.48846501540757e-12\\
7.5	2.43991388785823e-12\\
7.75	2.78190212161413e-12\\
8	2.21847824657322e-12\\
8.25	2.15652146785672e-12\\
8.5	2.80054465158984e-12\\
8.75	3.68697443959595e-12\\
9	1.44850242101612e-12\\
9.25	3.37501675254039e-12\\
9.5	2.90429108055829e-12\\
9.75	2.1350053729273e-12\\
10	3.18432636663549e-12\\
};
%\addlegendentry{$\alpha = 10^{-5}$}

\addplot [color=black, line width=1.5pt, forget plot]
table[row sep=crcr]{%
0.5	10\\
0.5	1e-15\\
};
\end{axis}

\begin{axis}[%
width=0.37\linewidth,
height=0.7in,
at={(0,0in)},
scale only axis,
xmin=-0,
xmax=10,
xlabel style={font=\color{white!15!black}},
xlabel={time},
ymode=log,
ymin=1e-14,
ymax=1,
yminorticks=true,
ylabel style={font=\color{white!15!black}},
ylabel={$\|v(t)\|_{L_2(\Omega)}$},
axis background/.style={fill=white},
legend style={at={(1.045,1.05)}, anchor=south, legend cell align=left, align=left, draw=white!15!black,legend columns=-1},
cycle list name = black white
]
\addplot [dotted,line width=1.5pt]
table[row sep=crcr]{%
0	0\\
0.25	0.109876571076452\\
0.5	0.286637035728083\\
0.75	0.134272462722272\\
1	0.0633533549787621\\
1.25	0.0300262306561703\\
1.5	0.0142750597811561\\
1.75	0.0068021742862555\\
2	0.00324699188946012\\
2.25	0.00155210900218491\\
2.5	0.000742782353314028\\
2.75	0.000355808751827972\\
3	0.000170576953236308\\
3.25	8.18358972814823e-05\\
3.5	3.92985841741386e-05\\
3.75	1.89038063664205e-05\\
4	9.1221557339487e-06\\
4.25	4.42443182025027e-06\\
4.5	2.16120325605744e-06\\
4.75	1.06672694462566e-06\\
5	5.37619278881114e-07\\
5.25	2.84495913382307e-07\\
5.5	1.67120308366687e-07\\
5.75	1.16973143994686e-07\\
6	9.61347116887791e-08\\
6.25	8.24272338033877e-08\\
6.5	6.85221993758921e-08\\
6.75	5.69656029566675e-08\\
7	5.29281766264253e-08\\
7.25	5.6747727081332e-08\\
7.5	6.32550887647131e-08\\
7.75	6.73614483132469e-08\\
8	6.63867795527753e-08\\
8.25	6.04456993701162e-08\\
8.5	5.24719115498804e-08\\
8.75	4.69392999365754e-08\\
9	4.61589819003066e-08\\
9.25	4.77691633960838e-08\\
9.5	4.79777379152192e-08\\
9.75	4.45667015311506e-08\\
10	3.69237300613368e-08\\
};
\addlegendentry{$\alpha = 10^{-1}$}

\addplot [dashdotted,line width=1.5pt]
table[row sep=crcr]{%
0	0\\
0.25	0.00303326350133962\\
0.5	0.0617548003227151\\
0.75	0.00304154490969665\\
1	0.000158974762948836\\
1.25	8.64171885623015e-06\\
1.5	4.81881916774629e-07\\
1.75	2.90639765823592e-08\\
2	3.70366923331842e-09\\
2.25	3.36486599167022e-09\\
2.5	2.6026422579453e-09\\
2.75	2.39364608925062e-09\\
3	3.35482214391852e-09\\
3.25	2.63726817917971e-09\\
3.5	2.30558321349785e-09\\
3.75	2.70963813060365e-09\\
4	1.90943971735544e-09\\
4.25	3.67299149737542e-09\\
4.5	3.88417886170561e-09\\
4.75	1.8126318637614e-09\\
5	2.13894080789289e-09\\
5.25	3.60421802685551e-09\\
5.5	1.81040807124226e-09\\
5.75	2.85254823383044e-09\\
6	1.21292669618915e-09\\
6.25	2.23459674473446e-09\\
6.5	1.43438124919811e-09\\
6.75	1.57848118133095e-09\\
7	1.60666941902137e-09\\
7.25	7.66196017758807e-10\\
7.5	3.12386881782142e-09\\
7.75	1.3297097906264e-09\\
8	2.20040351753342e-09\\
8.25	2.23367487614967e-09\\
8.5	1.2315778052963e-09\\
8.75	3.00442019085552e-09\\
9	1.50026075639593e-09\\
9.25	1.93658093642398e-09\\
9.5	9.992483659602e-10\\
9.75	1.80436458448195e-09\\
10	2.00099863268709e-09\\
};
\addlegendentry{$\alpha =10^{-3}$}

\addplot [solid, line width=1.5pt]
table[row sep=crcr]{%
0	0\\
0.25	4.78865738897341e-06\\
0.5	0.00155278673715428\\
0.75	4.78874279791148e-06\\
1	2.23514860297312e-08\\
1.25	1.78758872350206e-10\\
1.5	1.15892899442714e-10\\
1.75	1.20463994615524e-10\\
2	1.22955921621858e-10\\
2.25	7.9233404787569e-11\\
2.5	9.87889133498295e-11\\
2.75	5.31671508993664e-11\\
3	6.08080472516808e-11\\
3.25	8.72428475567419e-11\\
3.5	6.45489426047558e-11\\
3.75	5.93104855603228e-11\\
4	6.04137877279779e-11\\
4.25	6.03188349072036e-11\\
4.5	7.37986130461609e-11\\
4.75	5.0409478035813e-11\\
5	5.21180505291105e-11\\
5.25	5.16310830635855e-11\\
5.5	5.29457280289426e-11\\
5.75	5.6309799212827e-11\\
6	5.63591968332172e-11\\
6.25	5.48978224510615e-11\\
6.5	4.80803431203954e-11\\
6.75	5.13577126880563e-11\\
7	5.6111283073628e-11\\
7.25	3.71220735343018e-11\\
7.5	4.73619118332523e-11\\
7.75	4.05646672173536e-11\\
8	4.22316463809683e-11\\
8.25	3.35802218499571e-11\\
8.5	4.17470869505279e-11\\
8.75	6.52187578894752e-11\\
9	3.37919848437783e-11\\
9.25	4.80764252668005e-11\\
9.5	4.95082859667744e-11\\
9.75	4.10705653560515e-11\\
10	5.23472576783698e-11\\
};
\addlegendentry{$\alpha = 10^{-5}$}

\addplot [color=black, line width=1.5pt, forget plot]
table[row sep=crcr]{%
	0.5	10\\
	0.5	1e-15\\
};
\end{axis}

\end{tikzpicture}%
}
	\caption{Norm of the secondary state $v$ and control $q$ of \eqref{eq:secondaryvars_timedisc} over time. The vertical black line indicates the implementation horizon $\tau=0.5$. The norm of the control space is scaled with $\sqrt{\alpha}$ for algorithmic purposes.}
	\label{fig:sec_long}
\end{figure}
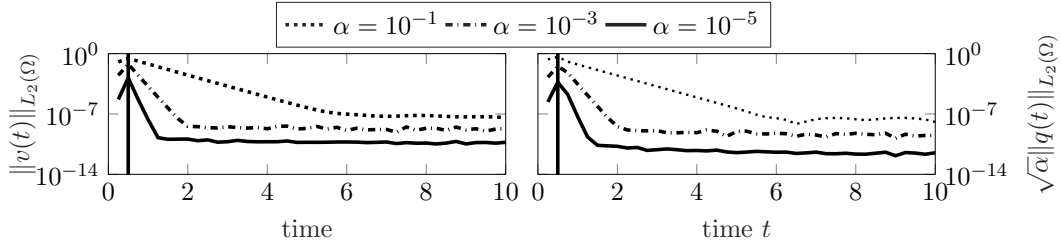
\section{Numerical results}
\label{sec:numerics:numres}
In this part we qualitatively and quantitatively examine the results of goal oriented error estimation with the specialized QOI defined in \eqref{eq:mpcqoi}, i.e., 
\begin{align*}
I^{\tau}(x,u) := \int\limits_0^\tau \bar{J}(t,x,u)\,dt
\end{align*} and compare it with classical error estimation using the full cost functional as QOI, i.e.,
\begin{align*}
J(x,u)=\int\limits_0^T \bar{J}(t,x,u)\,dt.
\end{align*}
We inspect the error indicators for time, space and space-time adaptivity, the resulting grids and the performance of a Model Predictive Controller evaluated via the cost functional value of the MPC trajectory, using goal oriented error estimation with either QOI, i.e., $I^\tau(x,u)$ or $J(x,u)$ used for adaptivity in every solution of an OCP. All numerical examples are performed with the C++-library for vector space algorithms \textit{Spacy}\footnote{https://spacy-dev.github.io/Spacy/}  using the finite element library \textit{Kaskade7} \cite{Goetschel2020}.

\subsubsection*{Problem setting}
In the following, we fix $\Omega = [0,3]\times [0,1]$ and the time horizon $T=10$. We utilize reference trajectories similar to the one defined in \cite[Section 6.7.2]{Meidner2008a}. That is, using
\begin{align*}
g(s) &:= \begin{cases}
10e^{1-\frac{1}{1-s^2}} \qquad &s< 1\\
0   &\text{else},
\end{cases}
\end{align*}
we define a static reference trajectory, depicted in \cref{fig:parabolic:auto_ref}, via
\begin{align}
\label{def:numerics:auto}
x_\text{d}^\text{stat}(\omega) &:= g\left(\frac{10}{3}\norm{\omega - \begin{pmatrix}
	1.5\\0.5\end{pmatrix}}\right).
\end{align}
\begin{figure}[H]
	\centering
	\scalebox{.7}{\input{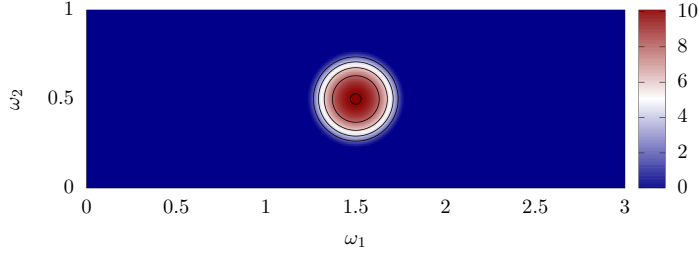}}
	\caption[Static reference trajectory $x_\text{d}^\text{stat}(\omega_1,\omega_2)$.]{Static reference trajectory $x_\text{d}^\text{stat}(\omega)=x_\text{d}^\text{stat}(\omega_1,\omega_2)$}
	\label{fig:parabolic:auto_ref}
\end{figure}
Further, we consider a dynamic reference trajectory given by
\begin{align}
\label{def:numerics:reference_nonauto}
x_\text{d}^\text{dyn}(t,\omega) := g\left(\frac{10}{3}\norm{\omega - \begin{pmatrix}
	\omega_{1,\text{peak}}(t)\\\omega_{2,\text{peak}}(t)\end{pmatrix}}\right),
\end{align}
where
\begin{align*}
\omega_{1,\text{peak}}(t):= 1.5 - \cos\left( \pi\left( \frac{t}{10} \right)\right),\qquad 
w_{2,\text{peak}}(t):= \left\vert \cos\left(\pi\left(\frac{t}{10}\right)\right)\right\vert.
\end{align*}
We will further consider an example with a reference concentrated at the boundary that grows exponentially in time, i.e., 
\begin{align}
\label{def:numerics:expincreasing}
x_\text{d}^\text{exp}(t,\omega) &:= e^{\frac{t}{2}}g\left(\frac{10}{3}\norm{\omega - \begin{pmatrix}
	1.5\\	1\end{pmatrix}}\right).
\end{align}
We consider the cost functional
\begin{align}
\label{def:numerics:cost}
\int_0^T \bar{J}(t,x,u)\,dt := \frac{1}{2}\int_0^T\|x(t)-x_\text{d}(t)&\|_{L_2(\Omega)}^2+ \alpha\|u(t)\|^2_{U} \,dt,
\end{align}
where $x_\text{d}$ is one of the reference trajectories defined above and $\alpha > 0$ is a Tikhonov parameter. Depending on the governing dynamics, we will set $U=L_2(\Omega)$ for the case of distributed control and $U=L_2(\partial\Omega)$ in the case of boundary control. 

%To guarantee reproducibility, we will always indicate the parameters and quantities in a table like \cref{tab:params} in order to characterize the OCP and the solution process in a unique way. The numbers $N_0$ and $M_0$ indicate the initial resolution of time and space grids, both obtained via uniform refinement.
%
%\begin{table}[!h]
%	\centering
%	\begin{tabular}{|| cc||} 
%		MPC&Initial grids \\
%		[0.5ex] 
%		\hline
%		$(\tau,K)$& $(k,h)$ 
%	\end{tabular}
%	\caption{Parameters for the OCP, the MPC-algorithm \cref{alg::mpcabstract} and the initial time and space mesh width.}
%	\label{tab:params}
%\end{table}
In the following we will consider different linear and nonlinear, unstable and stable dynamics with distributed and boundary control. We apply the MPC method of Algorithm~\ref{alg::mpcabstract} to these different model problems and perform goal oriented error estimation and grid refinement for either $I^\tau(x,u)$ or $J(x,u)$ as QOI after termination of the OCP solver. After refinement we use the interpolated solution on the refined grid as starting guess and solve the OCP again on the refined mesh. This procedure is repeated until the maximal number of time or space grid points is reached.

In all MPC simulations, we will perform four steps of Algorithm~\ref{alg::mpcabstract}. In case of time adaptivity, we compute all trajectories on a space grid three times uniformly refined, cf.\ \cref{tab:refinements}, and start the adaptive algorithm with three uniformly distributed time points, i.e., at time $0$, $T/2$ and $T$. The control is then fed back into a simulation performed on a grid with 51 uniformly distributed time grid points on $[0,\tau]$ and evaluated via interpolation. In case of space adaptivity, we fix the number of total time grid points to 41, perform adaptive space refinement starting with a grid with one uniform refinement, cf.\ \cref{tab:refinements} and perform the simulation with the same time step size, again evaluating the control via interpolation. The space grids for the simulation are five times uniformly refined, cf.\ \cref{tab:refinements}. In case of space-time adaptivity, we perform the simulation with 51 time grid points on $[0,\tau]$, where every space grid is five times uniformly refined and we start with five time grid points and one uniform refinement.

\begin{table}[H]
	\centering
	\begin{tabular}{c||c|c|c|c|c|c}
		\#uniform refs&0 & 1 &2 & 3 &4 &5  \\\hline\hline
		\#Triangles&12&48&192&768&3072&12288\\
		\#Vertices&11&33&113&417&1602&6273\\
	\end{tabular}
	\caption{Number of elements and degrees of freedom for different hierarchies of the spatial grid.}
	\label{tab:refinements}
\end{table}
We summarize this approach in the following pseudo-code.
	\begin{algorithm}[H]
		\caption{Adaptive MPC}
		\label{alg::MPC_ada}
		\begin{algorithmic}[1]
			\STATE{Given: Prediction horizon $0<T$, implementation horizon $0<\tau\leq T$, initial state $x_0$}
			\STATE{$k=0$}
			\WHILE{Controller active}
			\WHILE{\#DOFs $\leq$ maxDOFs }
			\STATE{Solve OCP on $[k\tau,T+k\tau]$ with initial datum $x_k$}
			\STATE{Compute error indicators, mark cells, refine grid}
			\ENDWHILE
			\STATE{Perform simulation on a fine space/time grid with control $u_{\big|[k\tau,(k+1)\tau]}$ via interpolation}
			\STATE{Update state $x_{k+1}$ via restriction}
			\STATE{$k = k+1$}
			\ENDWHILE
		\end{algorithmic}
	\end{algorithm}
	For the marking strategy in time we chose a D\"orfler strategy, cf.\ \cite[Section 4.2]{Doerfler1996} and aim to reduce the error by a factor of 50\% after refinement. The refinement itself is performed by standard bisection of the intervals. In space, we pursue a different marking strategy, namely we mark all cells that have error indicators higher than 30\%  of the maximal error. This alternative marking strategy is more aggressive if a lot of space grids need to be refined evenly, e.g., in case the variables are close to a turnpike for several time discretization points. With a D\"orfler criterion it can happen that the refinement procedure terminates after refining only some of the space grids, despite the error indicators being of similar size. For a more advanced marking strategy, where the number of refined cells is optimized, the interested reader is referred to \cite[Section 6.5]{Meidner2008a} and the references therein. After marking, the refinement in space is performed with standard red-green refinement, cf.\ \cite[Section 6.2.2]{AdaSolPDE} and \cite{Bank1983}.

In the following, we compare adaptivity with respect to the full cost functional with adaptivity with respect to a truncated cost functional. We do not draw any comparison to non-adaptive uniform discretizations. The issue with grid adaptivity in general is of course that it highly depends on the problem whether adaptive techniques are necessary at all. Moreover, there is a wide range of tuning parameters that parameterize the tradeoff between efficiency and accuracy and, on top of that, it depends on the particular implementation. 
	One example for such a parameter is the aggressiveness of the marking strategy. Depending on the choice of how many cells are marked in each ”solve, estimate, mark, refine”–cycle, the number of total cycles can change, up to the extreme that we only perform one of the above cycles. In order to avoid such dependencies, we chose to compare two adaptive approaches and only change the quantity of interest. This allows for a statement in terms of higher performance with comparable computational burden. If we compared an adaptive with a non-adaptive approach, we could have higher performance of the controller with higher computational cost, and the evaluation of what is desirable or what is the best compromise highly depends on the application.
	For these reasons, we leave the discussion of adaptive vs. non-adaptive for future research.
%\begin{itemize}
%\item Linear dynamics:
%\begin{itemize}
%\item time adaptivity: distributed control stable and unstable autonomous problem with $x_\text{d}^\text{stat}$, boundary control of stable non-autonomous problem with $x_\text{d}^\text{dyn}$.
%\item space adaptivity: distributed control of stable autonomous problem with $x_\text{d}^\text{stat}$ and non-autonomous problem $x_\text{d}^\text{exp}$.
%\end{itemize}
%\item Semilinear dynamics:
%\begin{itemize}
%	\item time adaptivity: distributed control of stable non-autonomous problem with $x_\text{d}^\text{dyn}$.
%	\item space adaptivity: distributed control of stable autonomous problem with $x_\text{d}\equiv 1$ on a rectangle with a crack.
%\end{itemize}
%\item Quasilinear dynamics:
%\begin{itemize}
%	\item time adaptivity: boundary control of stable non-autonomous problem with $x_\text{d}^\text{dyn}$.
%	\item space adaptivity: boundary control of stable non-autonomous problem with $x_\text{d}^\text{exp}$.
%\end{itemize}
%\end{itemize}
%\newpage
\subsection{Distributed linear quadratic control}
\label{subsec:numerics:linquad}
We will first consider linear quadratic problems and dynamics governed by a linear heat equation with distributed control, i.e.,
\begin{align}
\nonumber
&&\dot{x} &=   0.1\Delta x + sx + u  &&\text{in }(0,T) \times \Omega \\
\label{eq:numerics:linear_distributed}
&&x &= 0 &&\text{in }  (0,T)\times \partial\Omega \\
\nonumber
&&x(0) &= 0 &&\text{in } \Omega,
\end{align}
where $s\in \mathbb{R}$ is a stability (if $s < 0$) or instability (if $s>0$) parameter. Note that due to the negative eigenvalues of the Dirichlet Laplace operator, the uncontrolled equation ($u=0$) is stable when choosing $s$ to be smaller than the absolute value of the largest eigenvalue of the Dirichlet Laplacian.
We aim to minimize the standard tracking-type cost functional \eqref{def:numerics:ocp} subject to these dynamics.
\subsubsection*{Time adaptivity}
\label{subsec:timestatic}
In \cref{fig:numerics:linear_unstable}, we depict the spatial norm of state and control over time for an autonomous problem with reference trajectory $x_\text{d}^\text{stat}$, cf.\ \eqref{def:numerics:auto}, and Tikhonov parameter $\alpha = 10^{-1}$ governed by dynamics described by \eqref{eq:numerics:linear_distributed} with instability parameter $s=4$. The refinement algorithm for both QOIs is performed until the number of $41$ time grid points is reached. We observe that the refinement with respect to the truncated QOI $I^\tau(x,u)$ only takes place at the beginning of the time interval. Further, we see that the error indicators, computed as indicated at the end of \Cref{sec:intro} decay exponentially shortly after the implementation horizon $\tau = 0.5$ due to the exponential decay of the secondary variables proven in \cref{thm:numerics:scaling_cont,thm:scaling}. In contrast, choosing the entire cost functional $J(x,u)$ as a QOI, we see that the refined time grid is fine towards $t=0$ and $t=T$. This is due to the fact that the dynamics exhibit steady state turnpike behavior, i.e., the highly dynamic parts are located at the beginning and the end of the time horizon. Hence, in order to obtain an accurate solution on the whole horizon, these parts need to be refined. Further we observe that the solution obtained by refinement via $I^\tau(x,u)$ does not exhibit the leaving arc despite very clearly showing the approaching arc. The reason for this is that it would be not optimal to perform the leaving arc as due to the coarse discretization on the last part of the interval, the control/state have to leave the turnpike very early, leading to sub-optimal behavior. This reflects the exponential decay of sensitivity that is present in this optimal control problem, i.e., in order to have a small discretization error on the initial part, the behavior towards the end of the time horizon is (exponentially) irrelevant.
\begin{figure}[H]
	\centering
	\scalebox{0.9}{\input{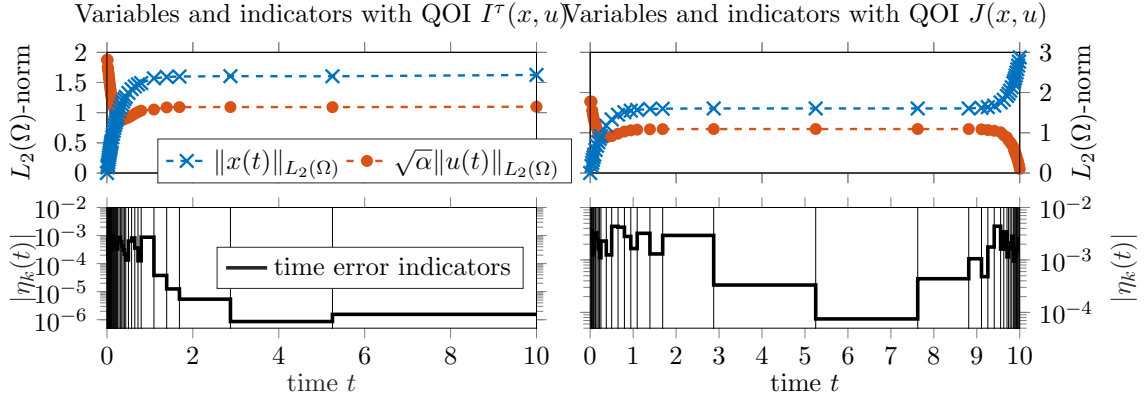}}
	%\scalebox{0.9}{\input{figs/goee/linear_dist_time.tex}}
	\caption[Open loop trajectories and time grids for a linear quadratic autonomous problem]{Open loop trajectories and error indicators in the first MPC step after adaptive refinement with 41 time grid points for an unstable problem with distributed control and static reference. The vertical lines illustrate the adaptively refined time grid.}
	\label{fig:numerics:linear_unstable}
\end{figure}
Having investigated the error indicators and the resulting time refinement in the context of one optimal control problem, we depict the performance gain in a Model Predictive Controller with three examples when using the truncated QOI in \cref{fig:numerics:linear_time_fvals} for adaptivity in every MPC step. We depict the closed-loop cost of the MPC trajectory obtained by applying four steps of the MPC method Algorithm~\ref{alg::mpcabstract} to the optimal control problem. The plot on the left depicts the closed-loop cost for a stable autonomous problem with $s=0$, $\alpha = 10^{-3}$, and reference $x_\text{d}^\text{stat}$. The plot in the middle refers to the to the unstable autonomous problem of \cref{fig:numerics:linear_unstable}. On the right, we evaluate the performance for an non-autonomous problem with boundary control, i.e., we replace in \cref{eq:numerics:linear_distributed} the distributed control by a Neumann boundary control, and use the reference $x_\text{d}^\text{dyn}$, $\alpha = 10^{-3}$, $s=0$, and MPC implementation horizon $\tau = 1$. In all three cases we observe that for a given number of maximal time steps, choosing the specialized QOI $I^\tau(x,u)$ as an objective for refinement leads to a significant reduction of the closed-loop cost, i.e., a better controller performance.
\begin{figure}[h]
	\centering
	\scalebox{.9}{% This file was created by matlab2tikz.
\definecolor{mycolor1}{rgb}{0.00000,0.44700,0.74100}%
\begin{tikzpicture}

\begin{axis}[%
width=0.28\linewidth,
height=0.8in,
at={(0in,0in)},
scale only axis,
xmin=4,
xmax=42,
xtick={ 5,  8, 11, 21, 31,41},
xlabel style={font=\color{white!15!black}},
xlabel={number of time grid points},
%ymode=log,
%ymin=1,
ymax=1,
%ytick={        1,     10,       100,      1000,     10000},
%yminorticks=true,
ylabel style={font=\color{white!15!black}},
ylabel={closed-loop cost},
axis background/.style={fill=white},
title style={align=center,at={(2.6in,1)}},
legend style={legend pos = north east,legend cell align=left, align=left, draw=white!15!black}
]
\addplot [color=black, dashdotted, line width=2.0pt, mark=*, mark options={solid, black}]
  table[row sep=crcr]{%
5 0.971723117738656\\
8 0.971723117831177\\
11 0.971723140161875\\
21 0.759592292670437\\
31 0.758025794412031\\
41 0.759833952585075\\
};

%\addlegendentry{uniform}

\addplot [color=black, dashdotted, line width=2.0pt, mark size=5.0pt, mark=x, mark options={solid, black}]
  table[row sep=crcr]{%
5 0.791424062456763\\
8 0.742257096474869\\
11 0.726280958666194\\
21 0.701739951631346\\
31 0.681182571137161\\
41 0.678307442969864\\
};

%\addlegendentry{exponential}

\end{axis}

\begin{axis}[%
width=0.28\linewidth,
height=0.8in,
at={(0.33\linewidth,0in)},
scale only axis,
xmin=4,
xmax=42,
xtick={ 5,  8, 11, 21, 31,41},
xlabel style={font=\color{white!15!black}},
xlabel={number of time grid points},
%ymode=log,
%yticklabel pos = right,
%ylabel near ticks,
%ymin=3,
%ymax=100,
%ytick={     1,       10,       100},
yminorticks=false,
ylabel style={font=\color{white!15!black}},
axis background/.style={fill=white},
legend style={at={(0.5,1.2)},legend cell align=center, align=center,anchor=center, draw=white!15!black,legend columns=-1}
]
\addplot [color=black, dashdotted, line width=2.0pt, mark=*, mark options={solid, black}]
table[row sep=crcr]{%
5 17.5885283708724\\
8 8.03928952377969\\
11 5.4918034273583\\
21 4.59969271601855\\
31 4.50887642822812\\
41 4.411738566776\\
};
\addlegendentry{ref.\ with QOI $J(x,u)$}

\addplot [color=black, dashdotted, line width=2.0pt, mark size=5.0pt, mark=x, mark options={solid, black}]
table[row sep=crcr]{%
5 8.60060939778925\\
8 5.44195647579295\\
11 4.49247972208364\\
21 4.31682790197201\\
31 4.27418294191909\\
41 4.25245220935333\\
};
\addlegendentry{ref.\ with QOI $I^\tau(x,u)$}

\end{axis}
\begin{axis}[%
width=0.28\linewidth,
height=0.8in,
at={(0.66\linewidth,0in)},
scale only axis,
xmin=4,
xmax=42,
xtick={ 5,  8, 11, 21, 31,41},
xlabel style={font=\color{white!15!black}},
xlabel={number of time grid points},
%ymode=log,
%yticklabel pos = left,
%ylabel near ticks,
%ymin=3,
%ymax=100,
%ytick={     1,       10,       100},
yminorticks=false,
ylabel style={font=\color{white!15!black}},
axis background/.style={fill=white},
legend style={legend pos = north east,legend cell align=left, align=left, draw=white!15!black}
]
\addplot [color=black, dashdotted, line width=2.0pt, mark=*, mark options={solid, black}]
table[row sep=crcr]{%
5 7.38564040754814\\
8 7.44823017542751\\
11 7.47685376197402\\
21 7.46012819804939\\
31 7.15560386616318\\
41 6.57907453503698\\
};
%\addlegendentry{$I^T$}

\addplot [color=black, dashdotted, line width=2.0pt, mark size=5.0pt, mark=x, mark options={solid, black}]
table[row sep=crcr]{%
5 7.41546152881424\\
8 7.06534926352963\\
11 6.65794758434338\\
21 6.4858623015927\\
31 6.48076679465141\\
41 6.47096545649196\\
};
%\addlegendentry{$I^\tau$}

\end{axis}
\end{tikzpicture}%
}
	\caption[Comparison of MPC closed-loop cost for time adaptivity with linear quadratic dynamics]{Comparison of cost functional values of the MPC closed-loop trajectory for different QOIs used for temporal refinement. Left: Stable autonomous problem. Middle: Unstable autonomous problem. Right: Boundary controlled non-autonomous problem.}
	\label{fig:numerics:linear_time_fvals}
\end{figure}
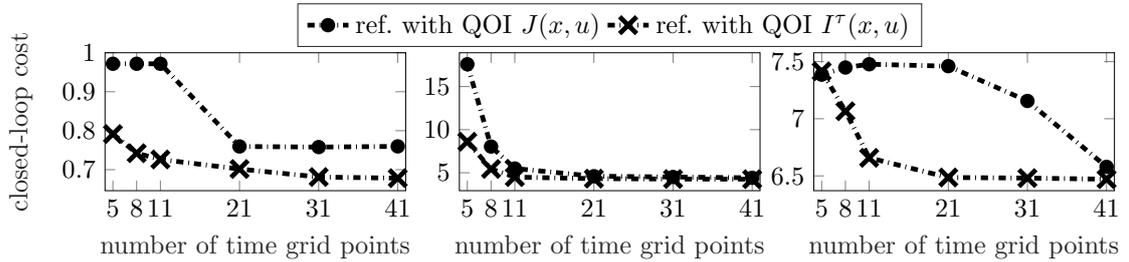

\FloatBarrier
\subsubsection*{Space adaptivity}
\label{subsec:space}
We now investigate the case of space refinement. To this end, we compare the error indicators, the resulting space grids and the closed-loop cost for refinement with objective $I^\tau(x,u)$ and $J(x,u)$, respectively. In the upper row of \cref{fig:numerics:linear_space_auto} the space error indicators for an autonomous optimal control problem governed by the linear dynamics with distributed control of \eqref{eq:numerics:linear_distributed} with $s=0$, reference $x_\text{d}^\text{stat}$ and Tikhonov parameter $\alpha=10^{-3}$ are depicted. Similar to the time error indicators in \cref{fig:numerics:linear_unstable}, the space error indicators for the objective $I^\tau(x,u)$ decay exponentially after the implementation horizon, whereas they stay almost constant over the whole time horizon in case of the QOI $J(x,u)$. The latter is again due to the turnpike property, i.e., the dynamic trajectories are close to the solution of the steady state problem for the majority of the time. The higher indicators at the beginning of the time interval are due to the high control action to approach the turnpike. Further, the indicators for the cost functional decay at the end of the horizon due to the terminal condition of the adjoint, which requires the control to approach zero, leading to a more regular state by diffusion and thus less need to refine. In the lower row of \cref{fig:numerics:linear_space_auto}, the resulting degrees of freedom (DOFs) of the space grids over time for three different numbers of maximal total spatial DOFs are depicted. It is clearly visible that for $I^\tau(x,u)$ a refinement only takes place at the beginning of the time horizon and the majority of the space grids are unrefined. In contrast to that, the spatial refinement for $J(x,u)$ takes place on the whole horizon. This numerical observation reflects the theoretical result of exponential decay of the secondary variables as shown in \Cref{sec:secondary} and \Cref{sec:secondary_discrete}. These secondary variables are also used for the computation of the space error indicators, cf.\ \eqref{eq:spaceerr}.
\begin{figure}[H]
	\centering
	\scalebox{.9}{% This file was created by matlab2tikz.
\definecolor{mycolor1}{rgb}{0.00000,0.44700,0.74100}%
\definecolor{mycolor2}{rgb}{0.85000,0.32500,0.09800}%
\definecolor{mycolor3}{rgb}{0.92900,0.69400,0.12500}%
\definecolor{mycolor4}{rgb}{0.49400,0.18400,0.55600}%

\begin{tikzpicture}

\begin{axis}[%
width=0.38\linewidth,
height=0.5in,
at={(0in,0.8in)},
scale only axis,
xmin=-0,
xmax=10,
xlabel style={font=\color{white!15!black}},
xlabel={},
xticklabels={\empty},
ymode=log,
ymin=2.77248e-12,
ymax=0.338236,
yminorticks=true,
ylabel style={font=\color{white!15!black}},
ylabel={$|\eta_h(t)|$},
axis background/.style={fill=white},
title={Refined with QOI $I^\tau(x,u)$}
]
\addplot [line width=1.5pt]
table[row sep=crcr]{%
0.0000000000 0.0000000000 \\ 
0.2500000000 0.0169118000 \\ 
0.5000000000 0.0022732900 \\ 
0.7500000000 0.0068881000 \\ 
1.0000000000 0.0004828205 \\ 
1.2500000000 0.0000332806 \\ 
1.5000000000 0.0000022562 \\ 
1.7500000000 0.0000001511 \\ 
2.0000000000 0.0000000123 \\ 
2.2500000000 0.0000000023 \\ 
2.5000000000 0.0000000004 \\ 
2.7500000000 0.0000000007 \\ 
3.0000000000 0.0000000006 \\ 
3.2500000000 0.0000000008 \\ 
3.5000000000 0.0000000011 \\ 
3.7500000000 0.0000000005 \\ 
4.0000000000 0.0000000002 \\ 
4.2500000000 0.0000000003 \\ 
4.5000000000 0.0000000001 \\ 
4.7500000000 0.0000000001 \\ 
5.0000000000 0.0000000003 \\ 
5.2500000000 0.0000000001 \\ 
5.5000000000 0.0000000002 \\ 
5.7500000000 0.0000000001 \\ 
6.0000000000 0.0000000000 \\ 
6.2500000000 0.0000000001 \\ 
6.5000000000 0.0000000000 \\ 
6.7500000000 0.0000000000 \\ 
7.0000000000 0.0000000001 \\ 
7.2500000000 0.0000000001 \\ 
7.5000000000 0.0000000002 \\ 
7.7500000000 0.0000000001 \\ 
8.0000000000 0.0000000001 \\ 
8.2500000000 0.0000000001 \\ 
8.5000000000 0.0000000002 \\ 
8.7500000000 0.0000000002 \\ 
9.0000000000 0.0000000001 \\ 
9.2500000000 0.0000000002 \\ 
9.5000000000 0.0000000000 \\ 
9.7500000000 0.0000000002 \\ 
10.0000000000 0.0000000001 \\ 
};
\addlegendentry{space error indicators};
\addplot [color=black, line width=1.5pt, forget plot]
table[row sep=crcr]{%
	0.5	2.77248e-12\\
	0.5	0.338236\\
};
\end{axis}

\begin{axis}[%
width=0.38\linewidth,
height=0.5in,
at={(0in,0in)},
scale only axis,
xmin=-0,
xmax=10,
xlabel style={font=\color{white!15!black}},
xlabel={time},
ymode=log,
ymin=28,
ymax=5966,
yminorticks=true,
ylabel style={font=\color{white!15!black}},
ylabel={spatial DOFs(t)},
axis background/.style={fill=white},
legend style={at={(0.97,0.7)}, anchor=east, legend cell align=left, align=left, draw=white!15!black}
]
\addplot [color=mycolor1,dotted,line width=1.5pt]
table[row sep=crcr]{%
	0	33\\
	0.25	169\\
	0.5	157\\
	0.75	41\\
	1	33\\
	1.25	33\\
	1.5	33\\
	1.75	33\\
	2	33\\
	2.25	33\\
	2.5	33\\
	2.75	33\\
	3	33\\
	3.25	33\\
	3.5	33\\
	3.75	33\\
	4	33\\
	4.25	33\\
	4.5	33\\
	4.75	33\\
	5	33\\
	5.25	33\\
	5.5	33\\
	5.75	33\\
	6	33\\
	6.25	33\\
	6.5	33\\
	6.75	33\\
	7	33\\
	7.25	33\\
	7.5	33\\
	7.75	33\\
	8	33\\
	8.25	33\\
	8.5	33\\
	8.75	33\\
	9	33\\
	9.25	33\\
	9.5	33\\
	9.75	33\\
	10	33\\
};
\addlegendentry{low \#DOFs}
%\addlegendentry{1621 total DOFs}
\addplot [color=mycolor2,dashdotted,line width=1.5pt]
table[row sep=crcr]{%
	0	33\\
	0.25	2261\\
	0.5	2137\\
	0.75	153\\
	1	53\\
	1.25	33\\
	1.5	33\\
	1.75	33\\
	2	33\\
	2.25	33\\
	2.5	33\\
	2.75	33\\
	3	33\\
	3.25	33\\
	3.5	33\\
	3.75	33\\
	4	33\\
	4.25	33\\
	4.5	33\\
	4.75	33\\
	5	33\\
	5.25	33\\
	5.5	33\\
	5.75	33\\
	6	33\\
	6.25	33\\
	6.5	33\\
	6.75	33\\
	7	33\\
	7.25	33\\
	7.5	33\\
	7.75	33\\
	8	33\\
	8.25	33\\
	8.5	33\\
	8.75	33\\
	9	33\\
	9.25	33\\
	9.5	33\\
	9.75	33\\
	10	33\\
};
\addlegendentry{med.\ \#DOFs}
%\addlegendentry{5825 total DOFs}

\addplot [color=mycolor4,solid,line width=1.5pt]
table[row sep=crcr]{%
	0	33\\
	0.25	5961\\
	0.5	4377\\
	0.75	377\\
	1	121\\
	1.25	53\\
	1.5	33\\
	1.75	33\\
	2	33\\
	2.25	33\\
	2.5	33\\
	2.75	33\\
	3	33\\
	3.25	33\\
	3.5	33\\
	3.75	33\\
	4	33\\
	4.25	33\\
	4.5	33\\
	4.75	33\\
	5	33\\
	5.25	33\\
	5.5	33\\
	5.75	33\\
	6	33\\
	6.25	33\\
	6.5	33\\
	6.75	33\\
	7	33\\
	7.25	33\\
	7.5	33\\
	7.75	33\\
	8	33\\
	8.25	33\\
	8.5	33\\
	8.75	33\\
	9	33\\
	9.25	33\\
	9.5	33\\
	9.75	33\\
	10	33\\
};
\addlegendentry{high \#DOFs}
%\addlegendentry{12077 total DOFs}

\addplot [color=black, line width=1.5pt, forget plot]
table[row sep=crcr]{%
	0.5	28\\
	0.5	5966\\
};
\end{axis}

\begin{axis}[%
width=0.38\linewidth,
height=0.5in,
at={(0.42\linewidth,0.8in)},
scale only axis,
xmin=-0,
xmax=10,
xlabel style={font=\color{white!15!black}},
xlabel={},
xticklabels={\empty},
ymode=log,
ymin=3.73227e-05,
ymax=0.337132,
yminorticks=true,
ylabel style={font=\color{white!15!black}},
ylabel={$|\eta_h(t)|$},
yticklabel pos = right,
ylabel near ticks,
axis background/.style={fill=white},
title={Refined with QOI $J(x,u)$}
]
\addplot [ line width=1.5pt]
table[row sep=crcr]{%
0.0000000000 0.0000000000 \\ 
0.2500000000 0.0168566000 \\ 
0.5000000000 0.0026735150 \\ 
0.7500000000 0.0026447300 \\ 
1.0000000000 0.0026675750 \\ 
1.2500000000 0.0026702000 \\ 
1.5000000000 0.0026704500 \\ 
1.7500000000 0.0026702800 \\ 
2.0000000000 0.0026700650 \\ 
2.2500000000 0.0026703650 \\ 
2.5000000000 0.0026702950 \\ 
2.7500000000 0.0026701850 \\ 
3.0000000000 0.0026704550 \\ 
3.2500000000 0.0026704000 \\ 
3.5000000000 0.0026702800 \\ 
3.7500000000 0.0026703900 \\ 
4.0000000000 0.0026703350 \\ 
4.2500000000 0.0026702650 \\ 
4.5000000000 0.0026703600 \\ 
4.7500000000 0.0026703400 \\ 
5.0000000000 0.0026702300 \\ 
5.2500000000 0.0026702850 \\ 
5.5000000000 0.0026703650 \\ 
5.7500000000 0.0026703300 \\ 
6.0000000000 0.0026703150 \\ 
6.2500000000 0.0026703400 \\ 
6.5000000000 0.0026703100 \\ 
6.7500000000 0.0026702850 \\ 
7.0000000000 0.0026703300 \\ 
7.2500000000 0.0026703450 \\ 
7.5000000000 0.0026703650 \\ 
7.7500000000 0.0026704200 \\ 
8.0000000000 0.0026703700 \\ 
8.2500000000 0.0026703050 \\ 
8.5000000000 0.0026702800 \\ 
8.7500000000 0.0026701450 \\ 
9.0000000000 0.0026701000 \\ 
9.2500000000 0.0026695800 \\ 
9.5000000000 0.0026585950 \\ 
9.7500000000 0.0024943450 \\ 
10.0000000000 0.0001866135 \\ 
};
%\addlegendentry{space error indicators};
\addplot [color=black, line width=1.5pt, forget plot]
table[row sep=crcr]{%
	0.5	3.73227e-05\\
	0.5	0.337132\\
};
\end{axis}

\begin{axis}[%
width=0.38\linewidth,
height=0.5in,
at={(0.42\linewidth,0in)},
scale only axis,
xmin=-0,
xmax=10,
xlabel style={font=\color{white!15!black}},
xlabel={time},
ymode=log,
ymin=28,
ymax=466,
yminorticks=true,
ylabel style={font=\color{white!15!black}},
ylabel={spatial DOFs(t)},
axis background/.style={fill=white},
yticklabel pos = right,
ylabel near ticks,
legend style={at={(0.97,0.5)}, anchor=east, legend cell align=left, align=left, draw=white!15!black}
]
\addplot [color=mycolor1, dotted,line width=1.5pt]
table[row sep=crcr]{%
	0	33\\
	0.25	57\\
	0.5	45\\
	0.75	45\\
	1	45\\
	1.25	39\\
	1.5	39\\
	1.75	39\\
	2	39\\
	2.25	39\\
	2.5	39\\
	2.75	39\\
	3	39\\
	3.25	39\\
	3.5	39\\
	3.75	39\\
	4	39\\
	4.25	39\\
	4.5	39\\
	4.75	45\\
	5	39\\
	5.25	39\\
	5.5	39\\
	5.75	39\\
	6	39\\
	6.25	39\\
	6.5	45\\
	6.75	39\\
	7	39\\
	7.25	39\\
	7.5	39\\
	7.75	39\\
	8	39\\
	8.25	39\\
	8.5	39\\
	8.75	39\\
	9	39\\
	9.25	39\\
	9.5	39\\
	9.75	39\\
	10	39\\
};
%\addlegendentry{1641 total DOFs}

\addplot [color=mycolor2,dashdotted,line width=1.5pt]
table[row sep=crcr]{%
	0	33\\
	0.25	149\\
	0.5	141\\
	0.75	139\\
	1	139\\
	1.25	139\\
	1.5	139\\
	1.75	139\\
	2	133\\
	2.25	133\\
	2.5	139\\
	2.75	139\\
	3	133\\
	3.25	139\\
	3.5	136\\
	3.75	133\\
	4	133\\
	4.25	136\\
	4.5	133\\
	4.75	133\\
	5	133\\
	5.25	136\\
	5.5	139\\
	5.75	133\\
	6	136\\
	6.25	139\\
	6.5	139\\
	6.75	133\\
	7	136\\
	7.25	133\\
	7.5	133\\
	7.75	133\\
	8	139\\
	8.25	139\\
	8.5	133\\
	8.75	133\\
	9	133\\
	9.25	133\\
	9.5	133\\
	9.75	133\\
	10	141\\
};
%\addlegendentry{5478 total DOFs}

\addplot [color=mycolor4,solid,line width=1.5pt]
table[row sep=crcr]{%
	0	33\\
	0.25	461\\
	0.5	397\\
	0.75	381\\
	1	375\\
	1.25	375\\
	1.5	375\\
	1.75	375\\
	2	375\\
	2.25	369\\
	2.5	369\\
	2.75	369\\
	3	369\\
	3.25	369\\
	3.5	369\\
	3.75	369\\
	4	369\\
	4.25	369\\
	4.5	369\\
	4.75	372\\
	5	369\\
	5.25	369\\
	5.5	372\\
	5.75	369\\
	6	369\\
	6.25	369\\
	6.5	369\\
	6.75	372\\
	7	369\\
	7.25	369\\
	7.5	369\\
	7.75	369\\
	8	369\\
	8.25	369\\
	8.5	369\\
	8.75	369\\
	9	369\\
	9.25	369\\
	9.5	369\\
	9.75	369\\
	10	299\\
};
%\addlegendentry{14894 total DOFs}

\addplot [color=black, line width=1.5pt, forget plot]
table[row sep=crcr]{%
	0.5	28\\
	0.5	466\\
};
\end{axis}
\end{tikzpicture}%
}
	\caption[Spatial error indicators and spatial DOFs for a linear quadratic autonomous problem]{Spatial error indicators before refinement and spatial degrees of freedom after last refinement for different maximal numbers of degrees of freedom for an autonomous optimal control problem. The vertical black line indicates the implementation horizon $\tau = 0.5$.}
	\label{fig:numerics:linear_space_auto}
\end{figure}
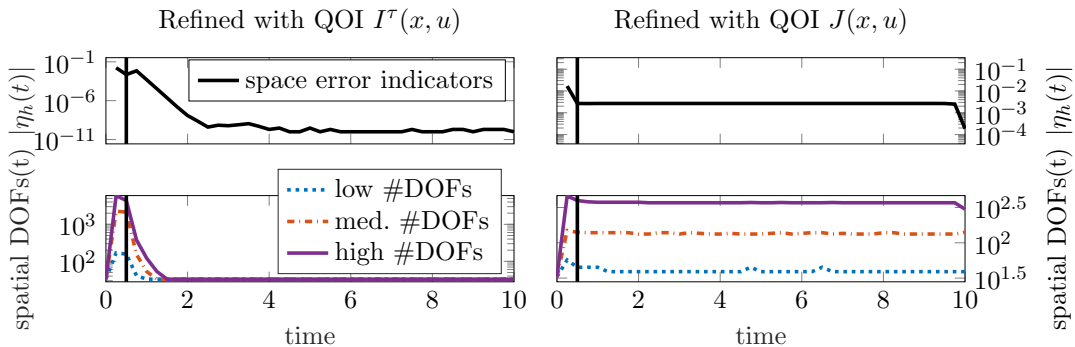 In \cref{fig:numerics:linear_spacegrids}, we depict the resulting space grids and the state over time for the intermediate case in \cref{fig:numerics:linear_space_auto}, i.e., the grids enjoy 5828 and 5478 total spatial DOFs, respectively. The slightly different numbers of total DOFs stem from the used space refinement methodology, where, in order to avoid hanging nodes, the refinement of one cell can lead to a refinement of neighboring cell.. In case of refinement for the full cost functional, the space grids have to capture the steady state turnpike on the majority of the interval. This is not the case for refinement with $I^\tau(x,u)$, where we observe unrefined space grids shortly after the implementation horizon $\tau=0.5$.
\begin{figure}%[H]
	\begin{tikzpicture}[scale=0.95]
	%\node at (7,5) {\large$\alpha = 10^{-3}$};
	\def\d{2.3}
	\def\t{0.5cm}
	\node (0,0){};
	\node[label={[label distance=0.2cm,text depth=-1ex]above:time $t$}] at (0.5\linewidth,1.4) {};
	\node[label={[label distance=0.2cm,text depth=-1ex]above:$I^\tau(x,u)$}] at (0.2\linewidth,1.4) {};
	\node[label={[label distance=0.2cm,text depth=-1ex]above:$J(x,u)$}] at (0.8\linewidth,1.4) {};
	\draw [very thick](0.5\linewidth,1.4) -> (0.5\linewidth,-5.5*\d) node [above right] {};
	\draw [arrow,very thick](0.5\linewidth,-6.5*\d) -> (0.5\linewidth,-7.5*\d) node [above right] {};
	\draw [very thick,dotted] (0.2\linewidth,-5.8*\d)-> (0.2\linewidth,-6.2*\d)[]{};
	\draw [very thick,dotted] (0.5\linewidth,-5.8*\d)-> (0.5\linewidth,-6.2*\d)[]{};
	\draw [very thick,dotted] (0.8\linewidth,-5.8*\d)-> (0.8\linewidth,-6.2*\d)[]{};
	
	\draw [very thick](0.5*\linewidth-\t,0) -> (0.5*\linewidth+\t,0);
	\draw [very thick](0.5*\linewidth-\t,0) -> (0.5*\linewidth+\t,0) node [ align=right,above right] {};
	\node[label={[anchor=south east]west:$t\!=\!0$}] at (0.51\linewidth,0.2cm) {};
	\draw [very thick](0.5*\linewidth-\t,-1*\d) -> (0.5*\linewidth+\t,-1*\d);;
	\node[label={[anchor=south east]west:$\tau\!=\!0.5$}] at (0.51\linewidth,-4.4) {};
	
	\draw [very thick](0.5*\linewidth-\t,-2*\d) -> (0.5*\linewidth+\t,-2*\d) node [ align=right,above right] {};
	\draw [very thick](0.5*\linewidth-\t,-3*\d) -> (0.5*\linewidth+\t,-3*\d) node [ align=right,above right] {};
	\draw [very thick](0.5*\linewidth-\t,-4*\d) -> (0.5*\linewidth+\t,-4*\d) node [ align=right,above right] {};
	\draw [very thick](0.5*\linewidth-\t,-5*\d) -> (0.5*\linewidth+\t,-5*\d) node [ align=right,above right] {};
	\draw [very thick](0.5*\linewidth-\t,-7*\d) -> (0.5*\linewidth+\t,-7*\d);
	\node[label={[anchor=south east]west:$T\!=\!10$}] at (0.51\linewidth,-15.9) {};
	%\draw [myblue,ultra thick](5,1)--(5,-0.5) node [below] {$\tau=0.5$};
	%\draw [thick](5,2)--(5,-0.2) node [below] {$\tau=0.5$};
	%\draw [](1,0.2) -- (1,-0.2) node [below=0.1cm] {$0$};
	%\draw [](3,0.2) -- (3,-0.2) node [above left=0.1cm] {};
	%\draw [](5,0.2) -- (5,-0.2) node [below right] {};
	%\draw [](7,0.2) -- (7,-0.2) node [below right] {};
	%\draw [thick](1,2) -- (1,-0.2);

	\node[inner sep=0pt,anchor=west] (whitehead) at (0,0)
	{\includegraphics[width=0.4\linewidth]{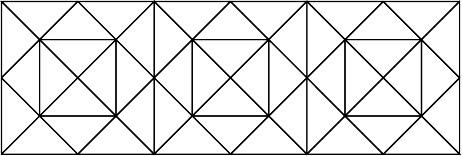}};
	\node[inner sep=0pt,anchor=east] (whitehead) at (0.2\linewidth,0)
	{\includegraphics[width=0.2\linewidth]{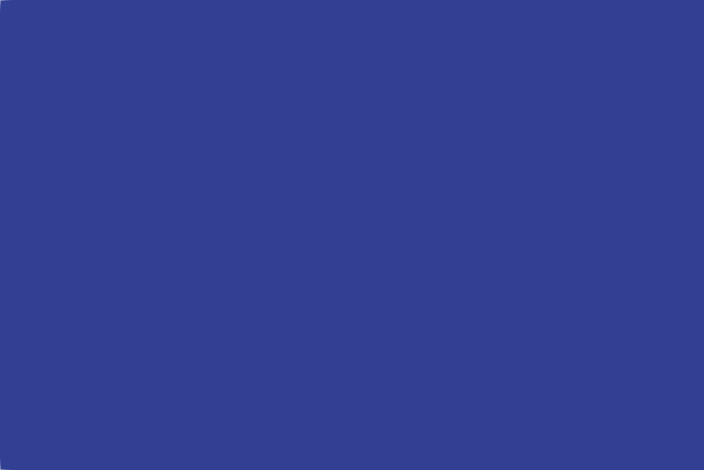}};
	
	\node[inner sep=0pt,anchor=west] (whitehead) at (0,-1*\d)
	{\includegraphics[width=0.4\linewidth]{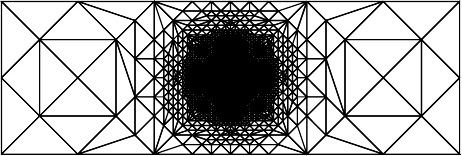}};
	\node[inner sep=0pt,anchor=east] (whitehead) at (0.2\linewidth,-1*\d)
	{\includegraphics[width=0.2\linewidth]{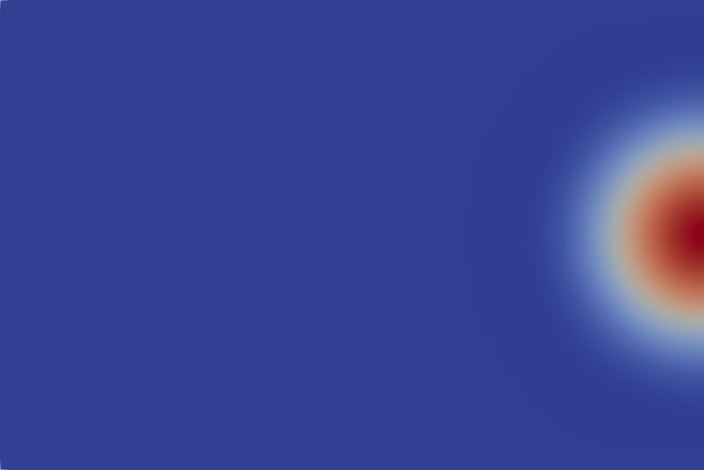}};
	\node[inner sep=0pt,anchor=west] (whitehead) at (0,-2*\d)
	{\includegraphics[width=0.4\linewidth]{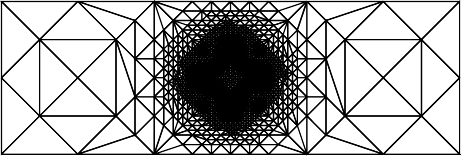}};
	\node[inner sep=0pt,anchor=east] (whitehead) at (0.2\linewidth,-2*\d)
	{\includegraphics[width=0.2\linewidth]{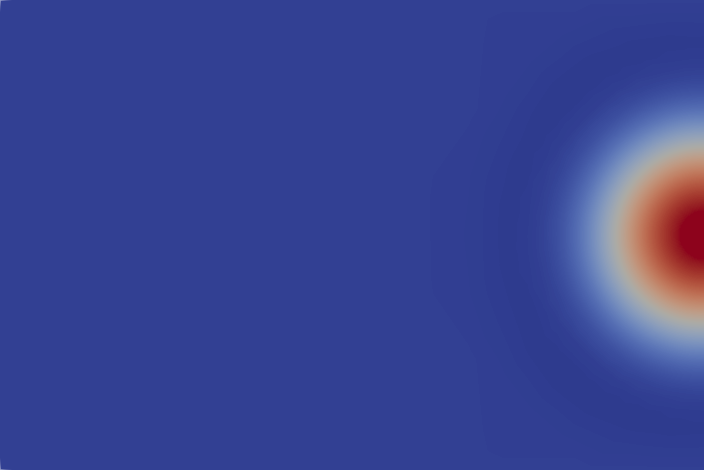}};
	\node[inner sep=0pt,anchor=west] (whitehead) at (0,-3*\d)
	{\includegraphics[width=0.4\linewidth]{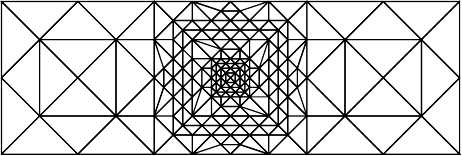}};
	\node[inner sep=0pt,anchor=east] (whitehead) at (0.2\linewidth,-3*\d)
	{\includegraphics[width=0.2\linewidth]{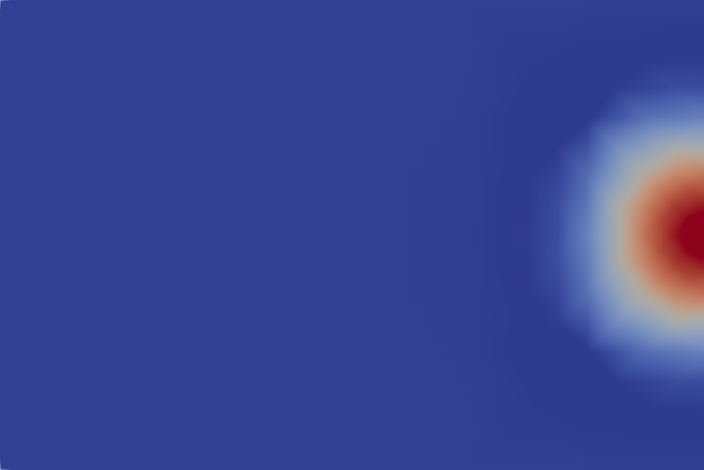}};
	\node[inner sep=0pt,anchor=west] (whitehead) at (0,-4*\d)
	{\includegraphics[width=0.4\linewidth]{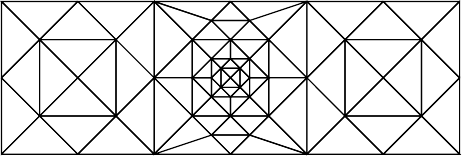}};
	\node[inner sep=0pt,anchor=east] (whitehead) at (0.2\linewidth,-4*\d)
	{\includegraphics[width=0.2\linewidth]{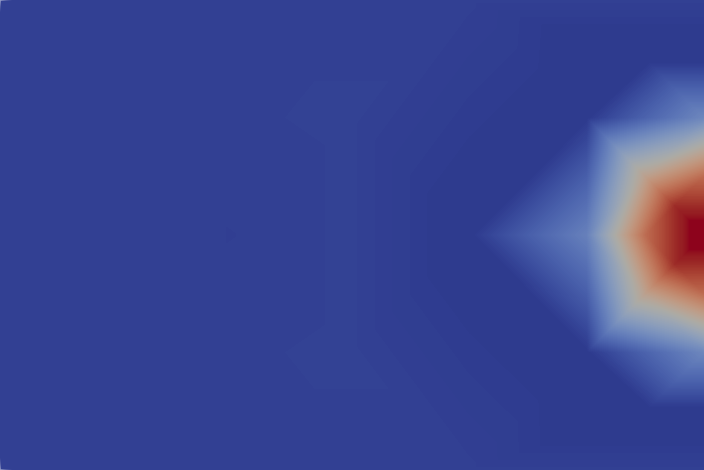}};
	\node[inner sep=0pt,anchor=west] (whitehead) at (0,-5*\d)
	{\includegraphics[width=0.4\linewidth]{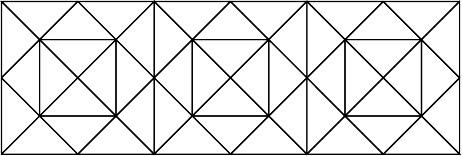}};
	\node[inner sep=0pt,anchor=east] (whitehead) at (0.2\linewidth,-5*\d)
	{\includegraphics[width=0.2\linewidth]{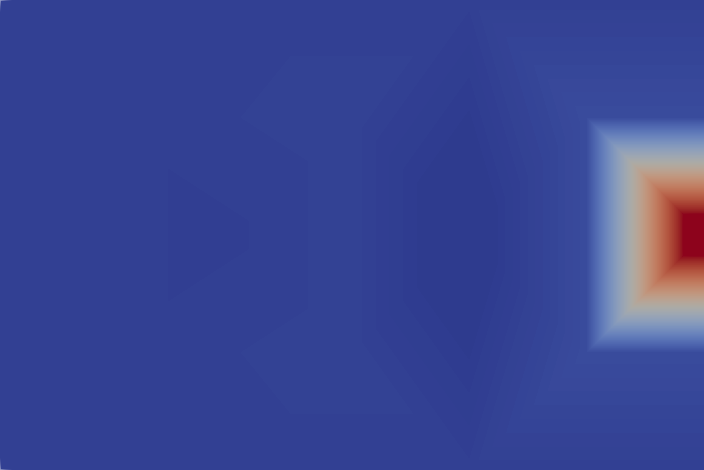}};
	\node[inner sep=0pt,anchor=west] (whitehead) at (0,-7*\d)
	{\includegraphics[width=0.4\linewidth]{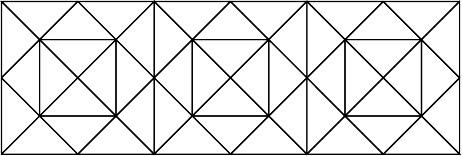}};
	\node[inner sep=0pt,anchor=east] (whitehead) at (0.2\linewidth,-7*\d)
	{\includegraphics[width=0.2\linewidth]{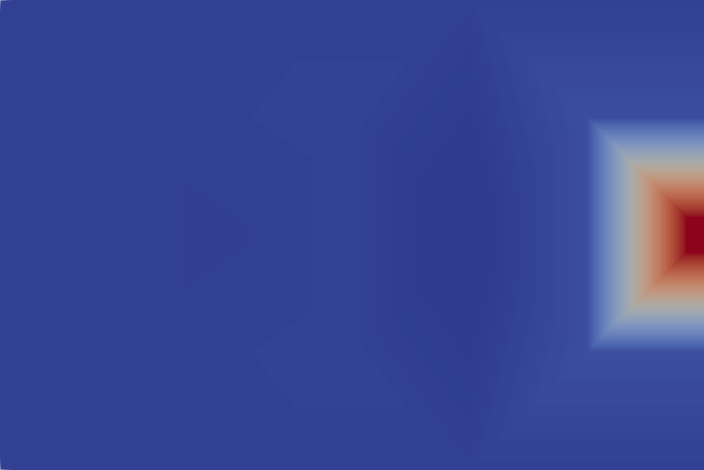}};

	\node[inner sep=0pt,anchor=east] (whitehead) at (\linewidth,0)
	{\includegraphics[width=0.4\linewidth]{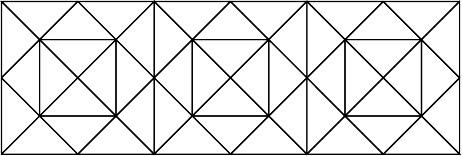}};
	\node[inner sep=0pt,anchor=west] (whitehead) at (0.8\linewidth,0)
	{\includegraphics[width=0.2\linewidth]{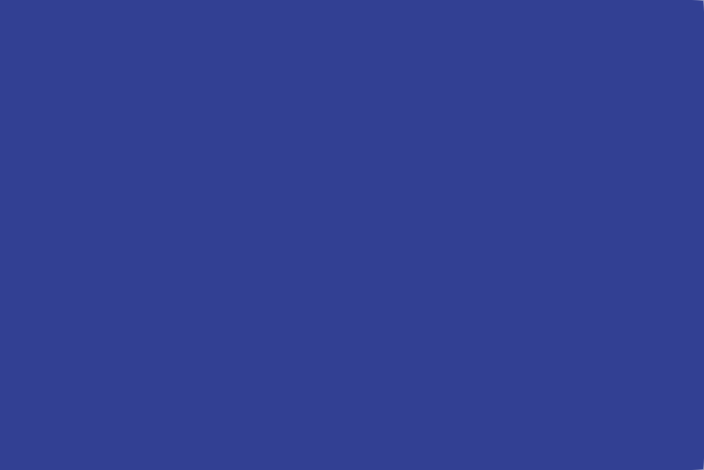}};
	\node[inner sep=0pt,anchor=east] (whitehead) at (\linewidth,-1*\d)
	{\includegraphics[width=0.4\linewidth]{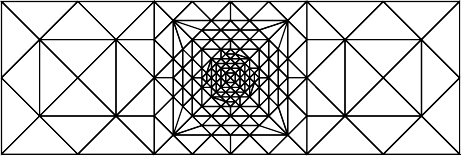}};
	\node[inner sep=0pt,anchor=west] (whitehead) at (0.8\linewidth,-1*\d)
	{\includegraphics[width=0.2\linewidth]{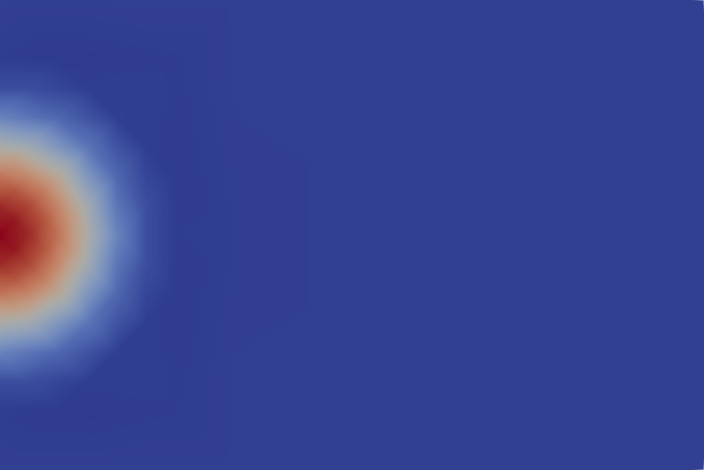}};
	
	\node[inner sep=0pt,anchor=east] (whitehead) at (\linewidth,-2*\d)
	{\includegraphics[width=0.4\linewidth]{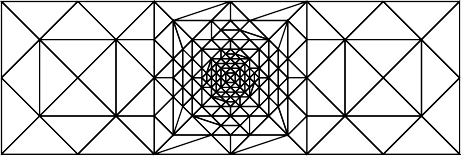}};
	\node[inner sep=0pt,anchor=west] (whitehead) at (0.8\linewidth,-2*\d)
	{\includegraphics[width=0.2\linewidth]{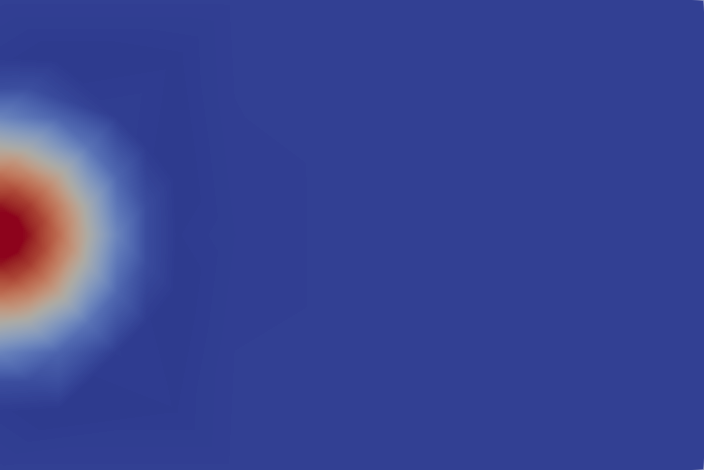}};
	\node[inner sep=0pt,anchor=east] (whitehead) at (\linewidth,-3*\d)
	{\includegraphics[width=0.4\linewidth]{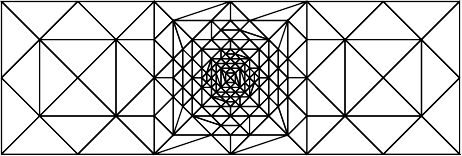}};
	\node[inner sep=0pt,anchor=west] (whitehead) at (0.8\linewidth,-3*\d)
	{\includegraphics[width=0.2\linewidth]{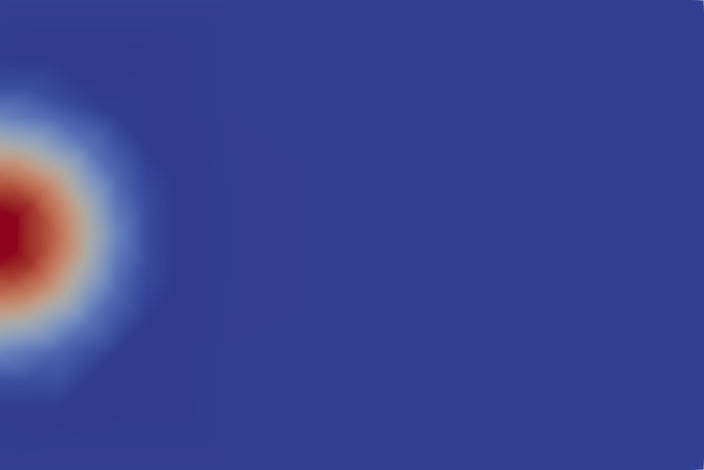}};
	\node[inner sep=0pt,anchor=east] (whitehead) at (\linewidth,-4*\d)
	{\includegraphics[width=0.4\linewidth]{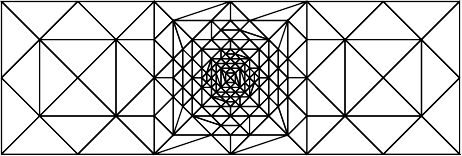}};
	\node[inner sep=0pt,anchor=west] (whitehead) at (0.8\linewidth,-4*\d)
	{\includegraphics[width=0.2\linewidth]{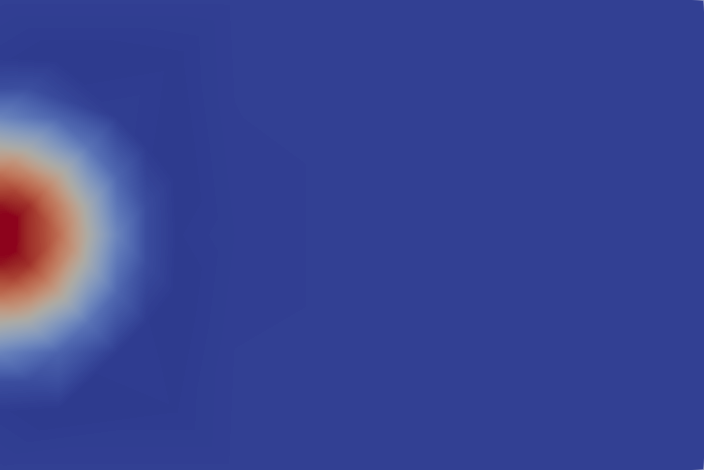}};
	\node[inner sep=0pt,anchor=east] (whitehead) at (\linewidth,-5*\d)
	{\includegraphics[width=0.4\linewidth]{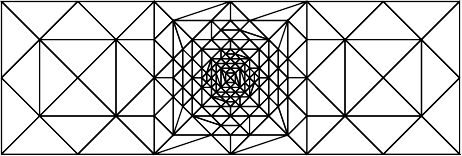}};
	\node[inner sep=0pt,anchor=west] (whitehead) at (0.8\linewidth,-5*\d)
	{\includegraphics[width=0.2\linewidth]{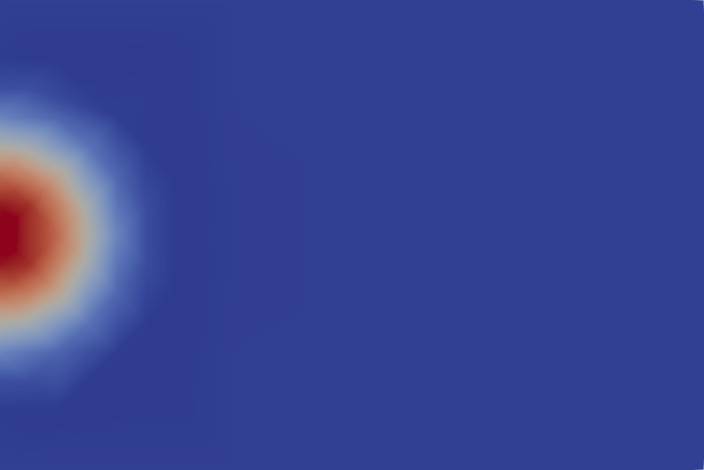}};
	\node[inner sep=0pt,anchor=east] (whitehead) at (\linewidth,-7*\d)
	{\includegraphics[width=0.4\linewidth]{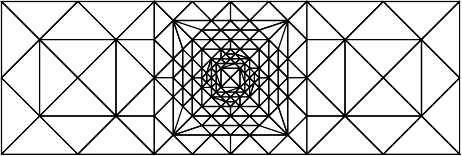}};
	\node[inner sep=0pt,anchor=west] (whitehead) at (0.8\linewidth,-7*\d)
	{\includegraphics[width=0.2\linewidth]{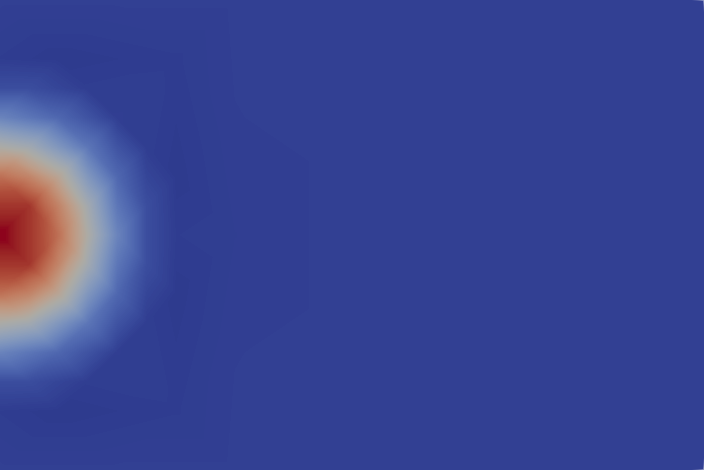}};
	\end{tikzpicture}
	\caption[Evolution of adaptively refined space grids for a linear quadratic problem]{Evolution of adaptively refined space grids for $I^\tau(x,u)$ (left) and $J(x,u)$ (right) with 5825 and 5478 total spatial DOFs, respectively.}
	\label{fig:numerics:linear_spacegrids}
\end{figure}
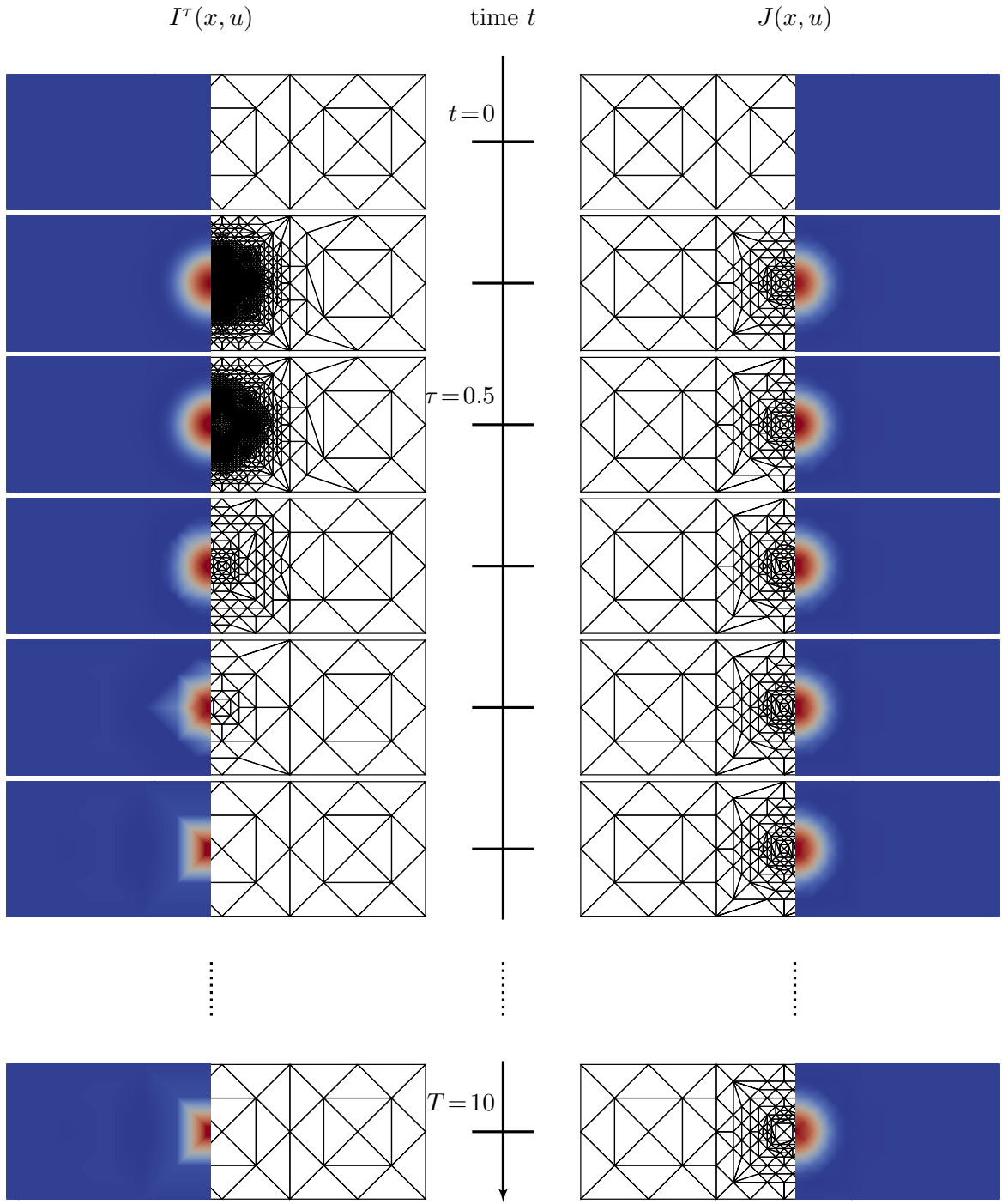
Finally we again inspect the performance gain from using $I^\tau(x,u)$ as a QOI in adaptive MPC. We examine the autonomous problem of \cref{fig:numerics:linear_space_auto} and further consider a non-autonomous problem with the exponentially increasing reference \cref{def:numerics:expincreasing}, $\alpha = 10^{-3}$ and $s=0$. In \cref{fig:fvals} we observe that for both examples the closed-loop cost is lower when using the specialized QOI $I^\tau(x,u)$ for refinement. In case of the exponentially increasing reference we further see that increasing the allowance for space refinement does not increase the performance when refining for $J(x,u)$, as all grid point are used towards $T$ and thus the MPC feedback is not refined at all. We will inspect this behavior in detail in \Cref{subsec:numerics:quasilin}.
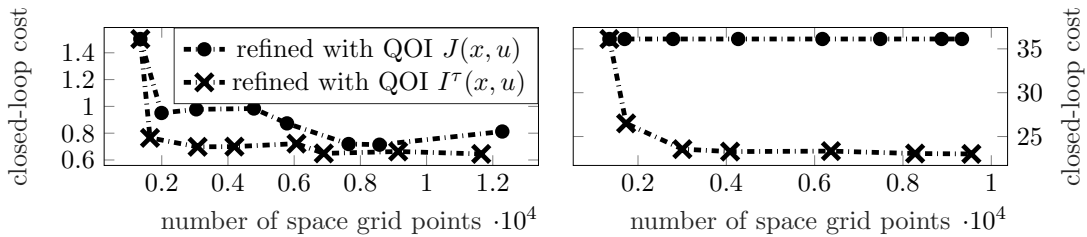
\begin{figure}%[H]
	\centering
	\scalebox{.9}{% This file was created by matlab2tikz.
\definecolor{mycolor1}{rgb}{0.00000,0.44700,0.74100}%
\begin{tikzpicture}

\begin{axis}[%
width=2.5in,
height=0.8in,
at={(0in,1.5in)},
scale only axis,
xlabel style={font=\color{white!15!black}},
xlabel={number of space grid points},
%ymode=log,
%ymin=1,
%ymax=1,
%ytick={        1,     10,       100,      1000,     10000},
%yminorticks=true,
ylabel style={font=\color{white!15!black}},
ylabel={closed-loop cost},
axis background/.style={fill=white},
title style={align=center,at={(2.6in,1)}},
legend style={at={(1,1)}, draw=white!15!black}
]
\addplot [color=black, dashdotted, line width=2.0pt, mark=*, mark options={solid, black}]
  table[row sep=crcr]{%
1353 1.50284434681286\\
1993 0.95080705665472\\
3043 0.977918486316675\\
4783 0.984025318472223\\
5783 0.873223215734578\\
7645 0.719113482958571\\
8577 0.714134749494935\\
12287 0.812515617149041\\
%15722 0.884290914035989\\
};

\addlegendentry{refined with QOI $J(x,u)$}

\addplot [color=black, dashdotted, line width=2.0pt, mark size=5.0pt, mark=x, mark options={solid, black}]
  table[row sep=crcr]{%
1353 1.50284434681286\\
1645 0.765080068299374\\
3083 0.698322870206426\\
4195 0.69979591362663\\
6070 0.722063689997339\\
6895 0.649161648818584\\
9141 0.663373225009628\\
%9513 0.641849811218674\\
11650 0.645129585746067\\
};

\addlegendentry{refined with QOI $I^\tau(x,u)$}

\end{axis}

\begin{axis}[%
width=2.5in,
height=0.8in,
at={(2.7in,1.5in)},
scale only axis,
xlabel style={font=\color{white!15!black}},
xlabel={number of space grid points},
%ymode=log,
yticklabel pos = right,
ylabel near ticks,
yminorticks=false,
ylabel style={font=\color{white!15!black}},
ylabel={closed-loop cost},
axis background/.style={fill=white}
]
\addplot [color=black, dashdotted, line width=2.0pt, mark=*, mark options={solid, black}]
table[row sep=crcr]{%
1353 36.1243007676739\\
1697 36.124300785368\\
2792 36.1243007561747\\
4269 36.1243007649599\\
6177 36.1243007632131\\
7485 36.1243007906212\\
8877 36.1243007777062\\
9335 36.1243007187975\\
%9335 36.1243007187975\\
};

\addplot [color=black, dashdotted, line width=2.0pt, mark size=5.0pt, mark=x, mark options={solid, black}]
table[row sep=crcr]{%
1353 36.1243007676739\\
1731 26.5006160977074\\
3005 23.5618765913998\\
4085 23.2827444582959\\
6369 23.337151620086\\
8275 23.0612891114376\\
9549 23.0148427334042\\
%9995 22.8045269797855\\
%12075 22.8098510651217\\
};

\end{axis}
%\begin{axis}[%
%width=2.5in,
%height=1in,
%at={(1.35in,0in)},
%scale only axis,
%%xmin=4,
%%xmax=42,
%%xtick={ 5,  8, 11, 21, 31,41},
%xlabel style={font=\color{white!15!black}},
%xlabel={number of time grid points},
%%ymode=log,
%yticklabel pos = left,
%ylabel near ticks,
%%ymin=3,
%%ymax=100,
%%ytick={     1,       10,       100},
%yminorticks=false,
%ylabel style={font=\color{white!15!black}},
%ylabel={Closed loop cost},
%axis background/.style={fill=white},
%legend style={legend pos = north east,legend cell align=left, align=left, draw=white!15!black}
%]
%\addplot [color=black, dashdotted, line width=2.0pt, mark=*, mark options={solid, black}]
%table[row sep=crcr]{%
%1353 36.1243007676739\\
%1697 36.124300785368\\
%2792 36.1243007561747\\
%4269 36.1243007649599\\
%6177 36.1243007632131\\
%7485 36.1243007906212\\
%8877 36.1243007777062\\
%9335 36.1243007187975\\
%9335 36.1243007187975\\
%};
%
%
%\addplot [color=black, dashdotted, line width=2.0pt, mark size=5.0pt, mark=x, mark options={solid, black}]
%table[row sep=crcr]{%
%1353 36.1243007676739\\
%1731 26.5006160977074\\
%3005 23.5618765913998\\
%4085 23.2827444582959\\
%6369 23.337151620086\\
%8275 23.0612891114376\\
%9549 23.0148427334042\\
%9995 22.8045269797855\\
%12075 22.8098510651217\\
%};

%\end{axis}
\end{tikzpicture}%
}
	\caption[Comparison of MPC closed-loop cost for space adaptivity with linear quadratic dynamics]{Comparison of cost functional values of the MPC closed-loop trajectory for different QOIs used for spatial refinement. Left: Autonomous problem. Right: Non-autonomous problem with exponentially increasing reference.}
	\label{fig:fvals}
\end{figure}
\FloatBarrier
\subsubsection*{Space-time adaptivity}
\label{subsec:spacetime}
We briefly address the subject of space and time adaptivity for the linear dynamics \eqref{eq:numerics:linear_distributed} with static reference $x_\text{d}^\text{stat}$, $\alpha = 10^{-3}$, $\tau=0.5$ and $s=0$. After time and space error estimation, we refine either space or time, depending on which is subject to a larger total error. This was chosen because of its clarity and simplicity and we note that there are more involved space-time refinement strategies, cf.\ \cite[Section 6.5]{Meidner2008a}. Using the policy described above, the adaptive refinement with $I^\tau(x,u)$ terminated with 12 time grid points, whereas the refinement with $J(x,u)$ terminated with 11 time grid points. As to be expected, the space and time grid refinement for $I^\tau(x,u)$ primarily takes place on the initial part of the horizon, cf.\ \cref{fig:numerics:spacetime_par}. The space and time error indicators for the QOI $I^\tau(x,u)$ decay again exponentially after the implementation horizon $\tau =0.5$, which is not the case when choosing $J(x,u)$ as QOI.  Further we observe clearly in \cref{fig:numerics:fvals_spacetime} that again refinement with the truncated cost functional leads to a better performance of the Model Predictive Controller. We note that for this example, employing only time adaptivity for the QOI $I^\tau(x,u)$ with three uniform refinements in space, cf.\ \cref{tab:refinements}, leads to a lower closed-loop cost for the same number of total DOFs. This is no longer the case for two uniform refinements. 
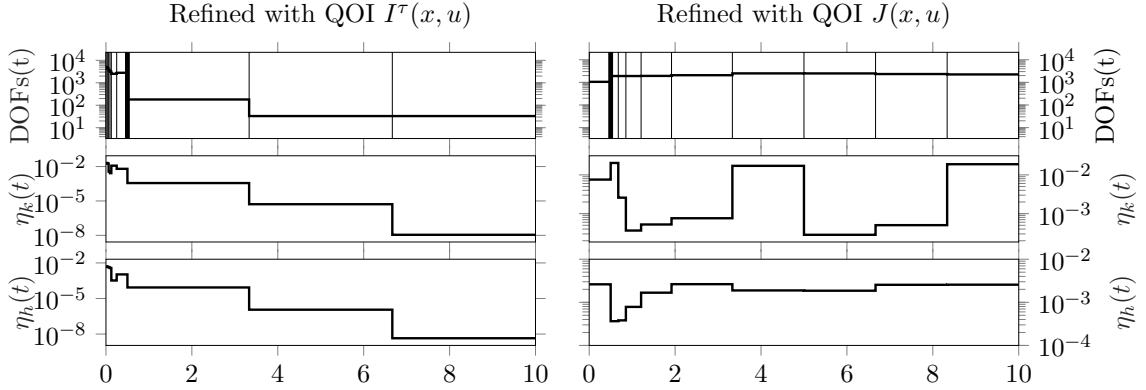
\begin{figure}[H]
	\centering
	\scalebox{.9}{% This file was created by matlab2tikz.
\begin{tikzpicture}

\begin{axis}[%
width=0.4\linewidth,
height=0.5in,
at={(0in,0in)},
scale only axis,
xmin=-0,
xmax=10,
xlabel style={font=\color{white!15!black}},
xlabel={},
xticklabels={\empty},
ymode=log,
ymin=3.3,
ymax=22730,
tick align=outside,
yminorticks=true,
ylabel style={font=\color{white!15!black}},
ylabel={DOFs(t)},
axis background/.style={fill=white},
title={Refined with QOI $I^\tau(x,u)$},
legend style={at={(0.99,0.85)}, anchor=east, legend cell align=left, align=left, draw=white!15!black}
]
\addplot[const plot,  line width=1.0pt] table[row sep=crcr] {%
	0	5039\\
	0.03125	4277\\
	0.0625	3555\\
	0.09375	3111\\
	0.125	2517\\
	0.25	2793\\
	0.5	181\\
	3.33333	33\\
	6.66667	33\\
	10	33\\
};
%\addlegendentry{Spatial DOFs}

\addplot [color=black, line width=0.1pt, forget plot]
table[row sep=crcr]{%
	0	10000000000\\
	0	1e-11\\
};
\addplot [color=black, line width=0.1pt, forget plot]
table[row sep=crcr]{%
	0.03125	10000000000\\
	0.03125	1e-11\\
};
\addplot [color=black, line width=0.1pt, forget plot]
table[row sep=crcr]{%
	0.0625	10000000000\\
	0.0625	1e-11\\
};
\addplot [color=black, line width=0.1pt, forget plot]
table[row sep=crcr]{%
	0.09375	10000000000\\
	0.09375	1e-11\\
};
\addplot [color=black, line width=0.1pt, forget plot]
table[row sep=crcr]{%
	0.125	10000000000\\
	0.125	1e-11\\
};
\addplot [color=black, line width=0.1pt, forget plot]
table[row sep=crcr]{%
	0.25	10000000000\\
	0.25	1e-11\\
};
\addplot [color=black, line width=0.1pt, forget plot]
table[row sep=crcr]{%
	0.5	10000000000\\
	0.5	1e-11\\
};
\addplot [color=black, line width=0.1pt, forget plot]
table[row sep=crcr]{%
	3.33333	10000000000\\
	3.33333	1e-11\\
};
\addplot [color=black, line width=0.1pt, forget plot]
table[row sep=crcr]{%
	6.66667	10000000000\\
	6.66667	1e-11\\
};
\addplot [color=black, line width=0.1pt, forget plot]
table[row sep=crcr]{%
	10	10000000000\\
	10	1e-11\\
};
\addplot [color=black, line width=2.0pt, forget plot]
table[row sep=crcr]{%
0.5	1.61595e-06\\
0.5	1e10\\
};
\end{axis}

\begin{axis}[%
width=0.4\linewidth,
height=0.5in,
at={(0.45\linewidth,0in)},
scale only axis,
xmin=-0,
xmax=10,
xlabel style={font=\color{white!15!black}},
xlabel={},
ymode=log,
ymin=3.3,
ymax=21270,
yminorticks=true,
xticklabels={\empty},
tick align=outside,
ylabel style={font=\color{white!15!black}},
ylabel={DOFs(t)},
xlabel={time $t$},
yticklabel pos = right,
ylabel near ticks,
axis background/.style={fill=white},
title={Refined with QOI $J(x,u)$},
legend style={at={(0.8,0.85)}, anchor=east, legend cell align=left, align=left, draw=white!15!black}
]
\addplot[const plot,line width=1.0pt, forget plot] table[row sep=crcr] {%
	0	1057\\
	0.5	1907\\
	0.677083 1907	\\
	0.854167 1907	\\
	1.20833	1931\\
	1.91667	2073\\
	3.33333	2497\\
	5	2485\\
	6.66667	2335\\
	8.33333	2245\\
	10	2245\\
};
\addplot [color=black, line width=0.1pt, forget plot]
table[row sep=crcr]{%
	0	10000000000\\
	0	1e-11\\
};
\addplot [color=black, line width=0.1pt, forget plot]
table[row sep=crcr]{%
	0.5	10000000000\\
	0.5	1e-11\\
};
\addplot [color=black, line width=0.1pt, forget plot]
table[row sep=crcr]{%
	0.677083	10000000000\\
	0.677083	1e-11\\
};
\addplot [color=black, line width=0.1pt, forget plot]
table[row sep=crcr]{%
	0.854167	10000000000\\
	0.854167	1e-11\\
};
\addplot [color=black, line width=0.1pt, forget plot]
table[row sep=crcr]{%
	1.20833	10000000000\\
	1.20833	1e-11\\
};
\addplot [color=black, line width=0.1pt, forget plot]
table[row sep=crcr]{%
	1.91667	10000000000\\
	1.91667	1e-11\\
};
\addplot [color=black, line width=0.1pt, forget plot]
table[row sep=crcr]{%
	3.33333	10000000000\\
	3.33333	1e-11\\
};
\addplot [color=black, line width=0.1pt, forget plot]
table[row sep=crcr]{%
	5	10000000000\\
	5	1e-11\\
};
\addplot [color=black, line width=0.1pt, forget plot]
table[row sep=crcr]{%
	6.66667	10000000000\\
	6.66667	1e-11\\
};
\addplot [color=black, line width=0.1pt, forget plot]
table[row sep=crcr]{%
	8.33333	10000000000\\
	8.33333	1e-11\\
};
\addplot [color=black, line width=0.1pt, forget plot]
table[row sep=crcr]{%
	10	10000000000\\
	10	1e-11\\
};

\addplot [color=black, line width=2.0pt, forget plot]
table[row sep=crcr]{%
0.5	0.0120526\\
0.5	1e10\\
};
\end{axis}

\begin{axis}[%
width=0.4\linewidth,
height=0.5in,
at={(0,-0.6in)},
scale only axis,
xmin=-0,
xmax=10,
xlabel style={font=\color{white!15!black}},
xlabel={},
xticklabels={\empty},
ymode=log,
tick align=outside,
yminorticks=true,
ylabel style={font=\color{white!15!black}},
ylabel={$\eta_k(t)$},
axis background/.style={fill=white},
legend style={at={(0.99,0.85)}, anchor=east, legend cell align=left, align=left, draw=white!15!black}
]
\addplot[const plot,line width=1.0pt] table[row sep=crcr] {%
0   0.0206007\\
0.03125   0.0182799\\
0.0625   0.00364931\\
0.09375   0.00265087\\
0.125    0.0119639\\
0.25    0.0063364\\
0.5  0.000363601\\
3.33333  5.18284e-06\\
6.66667 1.11396e-08\\
10 1.11396e-08\\
};
\end{axis}

\begin{axis}[%
width=0.4\linewidth,
height=0.5in,
at={(0.45\linewidth,-0.6in)},
scale only axis,
xmin=-0,
xmax=10,
xlabel style={font=\color{white!15!black}},
xlabel={},
xticklabels={\empty},
ymode=log,
tick align=outside,
yminorticks=true,
yticklabel pos = right,
ylabel near ticks,
ylabel style={font=\color{white!15!black}},
ylabel={$\eta_k(t)$},
axis background/.style={fill=white},
legend style={at={(0.99,0.85)}, anchor=east, legend cell align=left, align=left, draw=white!15!black}
]
\addplot[const plot,line width=1.0pt] table[row sep=crcr] {%
0  0.00756977\\
0.5    0.0203205\\
0.677083   0.00258262\\
0.854167  0.000366402\\
1.20833 0.000525111\\
1.91667  0.000761037\\
3.33333    0.0171188\\
5  0.000284531\\
6.66667  0.000504767\\
8.33333   0.0188518\\
10 0.0188518\\
};
\end{axis}

\begin{axis}[%
width=0.4\linewidth,
height=0.5in,
at={(0,-1.2in)},
scale only axis,
xmin=-0,
xmax=10,
xlabel style={font=\color{white!15!black}},
%label={},
%xticklabels={\empty},
ymode=log,
tick align=outside,
yminorticks=true,
ylabel style={font=\color{white!15!black}},
ylabel={$\eta_h(t)$},
axis background/.style={fill=white},
legend style={at={(0.99,0.85)}, anchor=east, legend cell align=left, align=left, draw=white!15!black}
]
\addplot[const plot,line width=1.0pt] table[row sep=crcr] {%
0 0.00508267\\
0.03125 0.00457754\\
0.0625 0.00420686\\
0.09375 0.00367643\\
0.125 0.000328024\\
0.25 0.00107521\\
0.5 8.51704e-05\\
3.33333 1.13306e-06\\
6.66667 4.42962e-09\\
10 4.42962e-09\\
};
\end{axis}

\begin{axis}[%
width=0.4\linewidth,
height=0.5in,
at={(0.45\linewidth,-1.2in)},
scale only axis,
xmin=-0,
xmax=10,
ymin=1e-4,
ymax = 1e-2,
xlabel style={font=\color{white!15!black}},
ymode=log,
tick align=outside,
yminorticks=true,
yticklabel pos = right,
ylabel near ticks,
ylabel style={font=\color{white!15!black}},
ylabel={$\eta_h(t)$},
axis background/.style={fill=white},
legend style={at={(0.99,0.85)}, anchor=east, legend cell align=left, align=left, draw=white!15!black}
]
\addplot[const plot, line width=1.0pt] table[row sep=crcr] {%
0 0.00261728\\
0.5 0.00036518\\
0.677083 0.000381479\\
0.854167 0.000783268\\
1.20833 0.00167348\\
1.91667  0.00263691\\
3.33333 0.0018886\\
5 0.00185918\\
6.66667 0.00254914\\
8.33333 0.0025655\\
10 0.0025655\\
};
\end{axis}
\end{tikzpicture}%
}
	\caption{Spatial DOFs over time for a total allowance for 20000 degrees of freedom for a fully adaptive space-time refinement.}
	\label{fig:numerics:spacetime_par}
\end{figure} 
\begin{figure}[H]
	\centering
	\scalebox{.9}{% This file was created by matlab2tikz.
\definecolor{mycolor1}{rgb}{0.00000,0.44700,0.74100}%
\begin{tikzpicture}

\begin{axis}[%
width=2.5in,
height=0.8in,
at={(0in,1.5in)},
scale only axis,
xlabel style={font=\color{white!15!black}},
xlabel={number of space grid points},
%ymode=log,
%ymin=1,
%ymax=1,
ymin = 0.8,
ymax = 1.1,
%ytick={        1,     10,       100,      1000,     10000},
%yminorticks=true,
ylabel style={font=\color{white!15!black}},
ylabel={closed-loop cost},
axis background/.style={fill=white},
title style={align=center,at={(2.6in,1)}},
legend style={anchor=west, draw=white!15!black,at={(1,0.5)}}
]
\addplot [color=black, dashdotted, line width=2.0pt, mark=*, mark options={solid, black}]
  table[row sep=crcr]{%
2165 1.07381992197631\\
4166 1.06091952710597\\
7604 1.01785399255639\\
10859 1.01680099074648\\
15509 1.01643632504998\\
};

\addlegendentry{refined with QOI $J(x,u)$}

\addplot [color=black, dashdotted, line width=2.0pt, mark size=5.0pt, mark=x, mark options={solid, black}]
  table[row sep=crcr]{%
2389 0.91682016611256\\
5318 0.944416291212479\\
7873 0.870491893864859\\
10537 0.82933306700543\\
15932 0.860937615977943\\
};

\addlegendentry{refined with QOI $I^\tau(x,u)$}

\end{axis}

\end{tikzpicture}%
}
	\caption[Comparison of MPC closed-loop cost for space-time adaptivity with linear quadratic dynamics]{Comparison of cost functional values of the MPC closed-loop trajectory for different QOIs used for space-time refinement.}
	\label{fig:numerics:fvals_spacetime}
\end{figure}
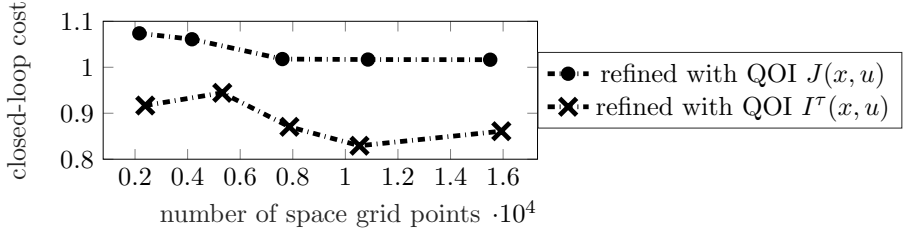

\FloatBarrier
\subsection{Boundary control of a quasilinear equation}
\label{subsec:numerics:quasilin}
As a second model problem we consider optimal control with the boundary controlled quasilinear problem. We define the nonlinear heat conduction tensor $\kappa(x)(t,\omega):=\left(c|x(t,\omega)|^2+d\right)$, where $c,d >0$ and consider the dynamics
\begin{align*}
&&\dot{x} - \nabla \cdot(\kappa(x)\nabla x) &= 0  &&\text{in } (0,T)\times \Omega ,\\
&&\kappa(x)\frac{\partial x}{\partial\nu} &= u &&\text{in } (0,T)\times \partial\Omega,\\
&&x(0) &= 0 &&\text{in } \Omega.
\end{align*}
\FloatBarrier
\subsubsection*{Time adaptivity}
In \cref{fig:numerics_quasilin_space_indicators} we depict the time error indicators and corresponding state and control norm over time for a non-autonomous problem with $x_\text{d}^\text{dyn}$ as reference, $\alpha = 10^{-2}$, $c=d=0.1$, and implementation horizon $\tau=1$. When using the full cost functional as QOI, the implementation horizon $[0,\tau]$ remains unrefined. The refinement for the truncated cost functional $I^\tau(x,u)$ is concentrated on the initial part.  The norm for the state is a scaled $H^1(\Omega)$-norm corresponding to the second derivative of the Lagrange function, i.e., $\|v\|_{\alpha d,H^1(\Omega)}:=\|v\|_{L_2(\Omega)}+\sqrt{\alpha d}\|\nabla v\|_{L_2(\Omega)}$.
\begin{figure}[H]
	\centering
	\scalebox{.9}{\input{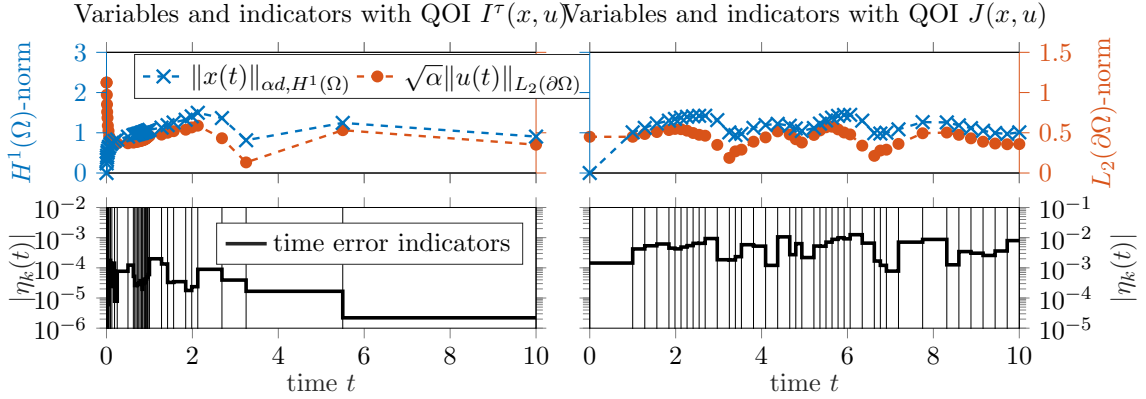}}
	\caption[Open loop trajectories and time grids for a boundary controlled quasilinear problem]{Open loop trajectories and error indicators in the first MPC step after adaptive refinement with 41 time grid points for an autonomous problem with boundary controlled quasilinear dynamics. The vertical lines illustrate the adaptively refined time grid.}
	\label{fig:numerics_quasilin_space_indicators}
\end{figure}

We briefly discuss the different choice of the implementation horizon $\tau$. In case of an autonomous system with distributed control it can happen that the the turnpike, i.e., an optimal steady state, is approached very quickly by the closed-loop trajectory. In order to be able to perform several MPC-steps before this convergence and to evaluate the performance, we chose $\tau = 0.5$ in \Cref{subsec:numerics:linquad}. However, when considering boundary controlled or non-autonomous systems, the behavior of the closed-loop trajectory is more dynamic, even after some action of the MPC controller. This is the reason why we are able to choose a higher horizon $\tau = 1$ and still perform four MPC steps without spending most of the time at an optimal steady state.

The depiction of the closed-loop cost of the MPC trajectory in \cref{fig:numerics:cost_quasilin} shows that the cost is again consistently lower when using $I^\tau(x,u)$ as a QOI for adaptive time refinement.
\begin{figure}[H]
	\centering
	\scalebox{.9}{% This file was created by matlab2tikz.
\definecolor{mycolor1}{rgb}{0.00000,0.44700,0.74100}%
\begin{tikzpicture}

\begin{axis}[%
width=2.5in,
height=0.8in,
at={(0in,0in)},
scale only axis,
xmin=4,
xmax=42,
xtick={ 5,  8, 11, 21, 31,41},
xlabel style={font=\color{white!15!black}},
xlabel={number of time grid points},
%ymode=log,
%ymin=1,
%ymax=,
%ytick={        1,     10,       100,      1000,     10000},
%yminorticks=true,
ylabel style={font=\color{white!15!black}},
ylabel={closed-loop cost},
axis background/.style={fill=white},
title style={align=center,at={(2.6in,1)}},
legend style={anchor=west, draw=white!15!black,at={(1,0.5)}}
]
\addplot [color=black, dashdotted, line width=2.0pt, mark=*, mark options={solid, black}]
  table[row sep=crcr]{%
5 9.2387769491819\\
8 9.26520448223292\\
11 9.26685645361639\\
21 9.22557569348871\\
31 9.22150697819468\\
41 9.13172109158479\\
};
\addlegendentry{refined with QOI $I^\tau(x,u)$}

\addplot [color=black, dashdotted, line width=2.0pt, mark size=5.0pt, mark=x, mark options={solid, black}]
  table[row sep=crcr]{%
5 9.13886382596811\\
8 9.14378885260488\\
11 8.99983838408757\\
21 8.88502511116267\\
31 8.86040609712719\\
41 8.82983038821291\\
};
\addlegendentry{refined with QOI $J(x,u)$}

\end{axis}
\end{tikzpicture}%
}
	\caption[Comparison of MPC closed-loop cost for time adaptivity with boundary controlled quasilinear dynamics]{Comparison of cost functional values of the MPC closed-loop trajectory for different QOIs used for temporal refinement with quasilinear dynamics.}
	\label{fig:numerics:cost_quasilin}
\end{figure}
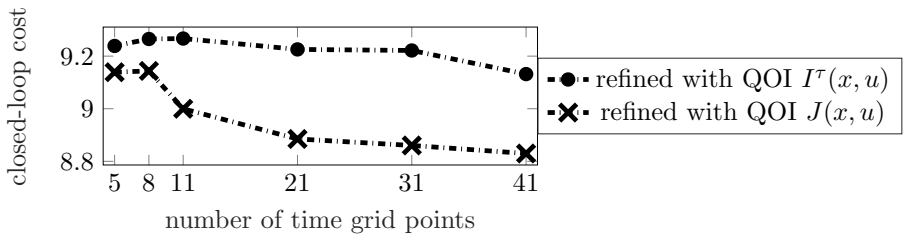
\FloatBarrier
\subsubsection*{Space adaptivity}
We consider the exponentially increasing reference trajectory $x_\text{d}^\text{exp}$, implementation horizon $\tau = 1$, Tikhonov parameter $\alpha=10^{-3}$ and parameters $d=10^{-1}$ and $c=10^{-2}$ for the heat conduction tensor.
In \cref{fig:spacedofs_exp_quasilin} we see that despite the exponentially increasing trajectory, the error indicators for $I^\tau(x,u)$ still decrease exponentially over time. The error indicators and the corresponding spatial DOFs after refinement for the full cost functional $J(x,u)$ are exponentially increasing. Note that due to the time-dependent reference trajectory, this problem does not exhibit steady-state turnpike behavior.
\begin{figure}[H]
	\centering
	\scalebox{.9}{% This file was created by matlab2tikz.
\definecolor{mycolor1}{rgb}{0.00000,0.44700,0.74100}%
\definecolor{mycolor2}{rgb}{0.85000,0.32500,0.09800}%
\definecolor{mycolor3}{rgb}{0.92900,0.69400,0.12500}%
\definecolor{mycolor4}{rgb}{0.49400,0.18400,0.55600}%
\begin{tikzpicture}

\begin{axis}[%
width=0.38\linewidth,
height=0.5in,
at={(0in,0.8in)},
scale only axis,
xmin=-0,
xmax=10,
xlabel style={font=\color{white!15!black}},
xlabel={},
xticklabels={\empty},
ymode=log,
ymin=1.61595e-06,
ymax=3.99486,
yminorticks=true,
ylabel style={font=\color{white!15!black}},
ylabel={$|\eta_h(t)|$},
axis background/.style={fill=white},
title={Refined with QOI $I^\tau(x,u)$},
legend style={at={(0.99,0.85)}, anchor=east, legend cell align=left, align=left, draw=white!15!black}
]
\addplot [ line width=1.5pt]
  table[row sep=crcr]{%
0.0000000000 0.0000000000 \\ 
0.2500000000 0.0663975000 \\ 
0.5000000000 0.1122705000 \\ 
0.7500000000 0.1818965000 \\ 
1.0000000000 0.1997430000 \\ 
1.2500000000 0.0501645000 \\ 
1.5000000000 0.0268392000 \\ 
1.7500000000 0.0121099500 \\ 
2.0000000000 0.0055271000 \\ 
2.2500000000 0.0025387650 \\ 
2.5000000000 0.0011751550 \\ 
2.7500000000 0.0005505450 \\ 
3.0000000000 0.0002578195 \\ 
3.2500000000 0.0000846285 \\ 
3.5000000000 0.0000434736 \\ 
3.7500000000 0.0000435253 \\ 
4.0000000000 0.0000244264 \\ 
4.2500000000 0.0000489252 \\ 
4.5000000000 0.0000647980 \\ 
4.7500000000 0.0000097117 \\ 
5.0000000000 0.0000080798 \\ 
5.2500000000 0.0000396541 \\ 
5.5000000000 0.0000386741 \\ 
5.7500000000 0.0000240972 \\ 
6.0000000000 0.0000353426 \\ 
6.2500000000 0.0000346831 \\ 
6.5000000000 0.0000686965 \\ 
6.7500000000 0.0000578610 \\ 
7.0000000000 0.0000313704 \\ 
7.2500000000 0.0000275549 \\ 
7.5000000000 0.0000522665 \\ 
7.7500000000 0.0000484878 \\ 
8.0000000000 0.0000286541 \\ 
8.2500000000 0.0000099293 \\ 
8.5000000000 0.0000435116 \\ 
8.7500000000 0.0000396604 \\ 
9.0000000000 0.0002181560 \\ 
9.2500000000 0.0003681340 \\ 
9.5000000000 0.0002544545 \\ 
9.7500000000 0.0002315595 \\ 
10.0000000000 0.0000674775 \\ 
};
\addlegendentry{space error indicators}
\addplot [color=black, line width=1.5pt, forget plot]
table[row sep=crcr]{%
1	1.61595e-06\\
1	3.99486\\
};
\end{axis}

\begin{axis}[%
width=0.38\linewidth,
height=0.5in,
at={(0in,0in)},
scale only axis,
xmin=-0,
xmax=10,
xlabel style={font=\color{white!15!black}},
xlabel={time},
ymode=log,
ymin=28,
ymax=1168,
yminorticks=true,
ylabel style={font=\color{white!15!black}},
ylabel={spatial DOFs(t)},
axis background/.style={fill=white},
legend style={at={(0.97,0.7)}, anchor=east, legend cell align=left, align=left, draw=white!15!black}
]
\addplot [color=mycolor1,dotted, line width=1.5pt]
  table[row sep=crcr]{%
0	33\\
0.25	88\\
0.5	87\\
0.75	105\\
1	111\\
1.25	49\\
1.5	52\\
1.75	46\\
2	33\\
2.25	33\\
2.5	33\\
2.75	33\\
3	33\\
3.25	33\\
3.5	33\\
3.75	33\\
4	33\\
4.25	33\\
4.5	33\\
4.75	33\\
5	33\\
5.25	33\\
5.5	33\\
5.75	33\\
6	33\\
6.25	33\\
6.5	33\\
6.75	33\\
7	33\\
7.25	33\\
7.5	33\\
7.75	33\\
8	33\\
8.25	33\\
8.5	33\\
8.75	33\\
9	33\\
9.25	33\\
9.5	33\\
9.75	33\\
10	33\\
};
%\addlegendentry{1660 total DOFs}
\addlegendentry{low \#DOFs}

\addplot [color=mycolor2,dashdotted, line width=1.5pt]
  table[row sep=crcr]{%
0	33\\
0.25	299\\
0.5	331\\
0.75	384\\
1	411\\
1.25	153\\
1.5	124\\
1.75	88\\
2	62\\
2.25	56\\
2.5	55\\
2.75	39\\
3	33\\
3.25	33\\
3.5	33\\
3.75	33\\
4	33\\
4.25	33\\
4.5	33\\
4.75	33\\
5	33\\
5.25	33\\
5.5	33\\
5.75	33\\
6	33\\
6.25	33\\
6.5	33\\
6.75	33\\
7	33\\
7.25	33\\
7.5	33\\
7.75	33\\
8	33\\
8.25	33\\
8.5	33\\
8.75	33\\
9	33\\
9.25	33\\
9.5	33\\
9.75	33\\
10	33\\
};
%\addlegendentry{2992 total DOFs}
\addlegendentry{med.\ \#DOFs}

\addplot [color=mycolor4,solid, line width=1.5pt]
  table[row sep=crcr]{%
0	33\\
0.25	520\\
0.5	788\\
0.75	1037\\
1	1163\\
1.25	417\\
1.5	231\\
1.75	205\\
2	153\\
2.25	101\\
2.5	83\\
2.75	66\\
3	55\\
3.25	45\\
3.5	39\\
3.75	33\\
4	33\\
4.25	33\\
4.5	33\\
4.75	33\\
5	33\\
5.25	33\\
5.5	36\\
5.75	36\\
6	36\\
6.25	36\\
6.5	33\\
6.75	33\\
7	41\\
7.25	36\\
7.5	36\\
7.75	43\\
8	33\\
8.25	36\\
8.5	39\\
8.75	42\\
9	42\\
9.25	36\\
9.5	42\\
9.75	68\\
10	63\\
};
%\addlegendentry{5934 total DOFs}
\addlegendentry{high \#DOFs}

\addplot [color=black, line width=1.5pt, forget plot]
table[row sep=crcr]{%
1	28\\
1	1168\\
};
\end{axis}
\begin{axis}[%
width=0.38\linewidth,
height=0.5in,
at={(0.42\linewidth,0.8in)},
scale only axis,
xmin=-0,
xmax=10,
xlabel style={font=\color{white!15!black}},
xlabel={},
xticklabels={\empty},
ymode=log,
ymin=0.0120526,
ymax=37784.7,
yminorticks=true,
yticklabel pos = right,
ylabel near ticks,
ylabel style={font=\color{white!15!black}},
ylabel={$|\eta_h(t)|$},
axis background/.style={fill=white},
title={Refined with QOI $J(x,u)$},
legend style={at={(0.8,0.85)}, anchor=east, legend cell align=left, align=left, draw=white!15!black}
]
\addplot [ line width=1.5pt]
  table[row sep=crcr]{%
0.0000000000 0.0000000000 \\ 
0.2500000000 0.0602630000 \\ 
0.5000000000 0.0982705000 \\ 
0.7500000000 0.1571080000 \\ 
1.0000000000 0.2256265000 \\ 
1.2500000000 0.3099850000 \\ 
1.5000000000 0.4133145000 \\ 
1.7500000000 0.5350850000 \\ 
2.0000000000 0.6715100000 \\ 
2.2500000000 0.8156100000 \\ 
2.5000000000 0.9568500000 \\ 
2.7500000000 1.0804750000 \\ 
3.0000000000 1.1664650000 \\ 
3.2500000000 1.1882050000 \\ 
3.5000000000 1.1109750000 \\ 
3.7500000000 0.8907600000 \\ 
4.0000000000 0.4741110000 \\ 
4.2500000000 0.2003195000 \\ 
4.5000000000 1.2007850000 \\ 
4.7500000000 2.5970600000 \\ 
5.0000000000 3.8788450000 \\ 
5.2500000000 3.5903350000 \\ 
5.5000000000 2.3878700000 \\ 
5.7500000000 0.8618600000 \\ 
6.0000000000 0.4965765000 \\ 
6.2500000000 1.0959700000 \\ 
6.5000000000 0.3476000000 \\ 
6.7500000000 2.4326500000 \\ 
7.0000000000 8.1340500000 \\ 
7.2500000000 17.9593500000 \\ 
7.5000000000 33.5454500000 \\ 
7.7500000000 57.0405000000 \\ 
8.0000000000 94.5305000000 \\ 
8.2500000000 337.2605000000 \\ 
8.5000000000 493.8325000000 \\ 
8.7500000000 657.0600000000 \\ 
9.0000000000 839.3300000000 \\ 
9.2500000000 1062.1400000000 \\ 
9.5000000000 1346.1250000000 \\ 
9.7500000000 1681.7000000000 \\ 
10.0000000000 1889.2350000000 \\ 
};
%\addlegendentry{space error indicators}
\addplot [color=black, line width=1.5pt, forget plot]
table[row sep=crcr]{%
1	0.0120526\\
1	37784.7\\
};
\end{axis}

\begin{axis}[%
width=0.38\linewidth,
height=0.5in,
at={(0.42\linewidth,0in)},
scale only axis,
xmin=-0,
xmax=10,
xlabel style={font=\color{white!15!black}},
xlabel={time},
ymode=log,
ymin=28,
ymax=693,
yminorticks=true,
yticklabel pos = right,
ylabel near ticks,
ylabel style={font=\color{white!15!black}},
ylabel={spatial DOFs(t)},
axis background/.style={fill=white},
legend style={at={(0,0.7)}, anchor=west, legend cell align=left, align=left, draw=white!15!black}
]
\addplot [color=mycolor1,dotted,line width=1.5pt]
  table[row sep=crcr]{%
0	33\\
0.25	33\\
0.5	33\\
0.75	33\\
1	33\\
1.25	33\\
1.5	33\\
1.75	33\\
2	33\\
2.25	33\\
2.5	33\\
2.75	33\\
3	33\\
3.25	33\\
3.5	33\\
3.75	33\\
4	33\\
4.25	33\\
4.5	33\\
4.75	33\\
5	33\\
5.25	33\\
5.5	33\\
5.75	33\\
6	33\\
6.25	36\\
6.5	36\\
6.75	43\\
7	43\\
7.25	43\\
7.5	43\\
7.75	45\\
8	50\\
8.25	56\\
8.5	56\\
8.75	56\\
9	59\\
9.25	67\\
9.5	65\\
9.75	89\\
10	99\\
};
%\addlegendentry{1711 total DOFs}

\addplot [color=mycolor2,dashdotted, line width=1.5pt]
  table[row sep=crcr]{%
0	33\\
0.25	33\\
0.5	33\\
0.75	33\\
1	33\\
1.25	33\\
1.5	33\\
1.75	33\\
2	33\\
2.25	33\\
2.5	33\\
2.75	33\\
3	33\\
3.25	33\\
3.5	33\\
3.75	33\\
4	33\\
4.25	33\\
4.5	33\\
4.75	36\\
5	43\\
5.25	43\\
5.5	43\\
5.75	43\\
6	43\\
6.25	50\\
6.5	56\\
6.75	56\\
7	65\\
7.25	65\\
7.5	78\\
7.75	90\\
8	109\\
8.25	109\\
8.5	116\\
8.75	144\\
9	153\\
9.25	187\\
9.5	219\\
9.75	252\\
10	297\\
};
%\addlegendentry{2924 total DOFs}

\addplot [color=mycolor4,solid,line width=1.5pt]
  table[row sep=crcr]{%
0	33\\
0.25	33\\
0.5	33\\
0.75	33\\
1	33\\
1.25	33\\
1.5	33\\
1.75	33\\
2	33\\
2.25	33\\
2.5	33\\
2.75	33\\
3	33\\
3.25	33\\
3.5	33\\
3.75	38\\
4	43\\
4.25	43\\
4.5	43\\
4.75	45\\
5	50\\
5.25	56\\
5.5	56\\
5.75	62\\
6	71\\
6.25	86\\
6.5	101\\
6.75	112\\
7	115\\
7.25	127\\
7.5	142\\
7.75	176\\
8	234\\
8.25	273\\
8.5	326\\
8.75	386\\
9	446\\
9.25	502\\
9.5	576\\
9.75	632\\
10	688\\
};
%\addlegendentry{5924 total DOFs}

\addplot [color=black, line width=1.5pt, forget plot]
table[row sep=crcr]{%
	1	28\\
	1	693\\
};
\end{axis}
\end{tikzpicture}%
}
	\caption[Spatial error indicators and spatial DOFs for a quasilinear problem]{Spatial error indicators before refinement and spatial degrees of freedom after last refinement for different maximal numbers of degrees of freedom for a boundary controlled quasilinear problem. The vertical black line indicates the implementation horizon $\tau = 1$.}
	\label{fig:spacedofs_exp_quasilin}
\end{figure}
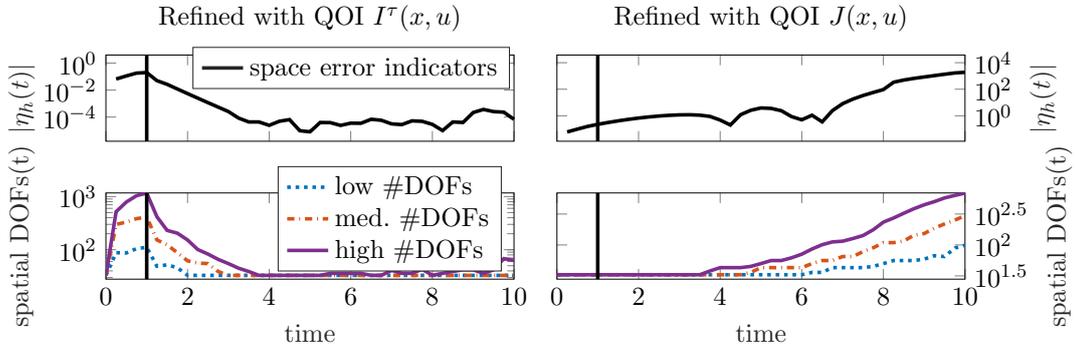 
The state over time and the corresponding space grids are shown in \cref{fig:spacegrids_quasilin}. Although state and control are relatively small on the initial part, the spatial refinement is most active there when refining for $I^\tau(x,u)$. On the other hand, the spatial grids refined for the full cost functional show no refinement on the whole implementation horizon $[0,1]$, as they are primarily refined towards the end of the horizon, due to the exponentially increasing reference trajectory.
\begin{figure}%[H]
	\begin{tikzpicture}[scale=0.95]
	%\node at (7,5) {\large$\alpha = 10^{-3}$};
	\def\d{2.3}
	\def\t{0.5cm}
	\node (0,0){};
	\node[label={[label distance=0.2cm,text depth=-1ex]above:time $t$}] at (0.5\linewidth,1.4) {};
	\node[label={[label distance=0.2cm,text depth=-1ex]above:$I^\tau(x,u)$}] at (0.2\linewidth,1.4) {};
	\node[label={[label distance=0.2cm,text depth=-1ex]above:$J(x,u)$}] at (0.8\linewidth,1.4) {};
	\draw [very thick](0.5\linewidth,1.4) -> (0.5\linewidth,-4.5*\d) node [above right] {};
	\draw [arrow,very thick](0.5\linewidth,-5.5*\d) -> (0.5\linewidth,-7.5*\d) node [above right] {};
	\draw [very thick,dotted] (0.2\linewidth,-4.8*\d)-> (0.2\linewidth,-5.2*\d)[]{};
	\draw [very thick,dotted] (0.5\linewidth,-4.8*\d)-> (0.5\linewidth,-5.2*\d)[]{};
	\draw [very thick,dotted] (0.8\linewidth,-4.8*\d)-> (0.8\linewidth,-5.2*\d)[]{};
	
	\draw [very thick](0.5*\linewidth-\t,0) -> (0.5*\linewidth+\t,0);
	\draw [very thick](0.5*\linewidth-\t,0) -> (0.5*\linewidth+\t,0) node [ align=right,above right] {};
	\node[label={[anchor=south east]west:$t\!=\!0$}] at (0.51\linewidth,0.2cm) {};
	\draw [very thick](0.5*\linewidth-\t,-1*\d) -> (0.5*\linewidth+\t,-1*\d);;
	\node[label={[anchor=south east]west:$\tau\!=\!1$}] at (0.51\linewidth,-9) {};
	
	\draw [very thick](0.5*\linewidth-\t,-2*\d) -> (0.5*\linewidth+\t,-2*\d) node [ align=right,above right] {};
	\draw [very thick](0.5*\linewidth-\t,-3*\d) -> (0.5*\linewidth+\t,-3*\d) node [ align=right,above right] {};
	\draw [very thick](0.5*\linewidth-\t,-4*\d) -> (0.5*\linewidth+\t,-4*\d) node [ align=right,above right] {};
	\draw [very thick](0.5*\linewidth-\t,-6*\d) -> (0.5*\linewidth+\t,-6*\d) node [ align=right,above right] {};
	\draw [very thick](0.5*\linewidth-\t,-7*\d) -> (0.5*\linewidth+\t,-7*\d);
	\node[label={[anchor=south east]west:$T\!=\!10$}] at (0.51\linewidth,-15.9) {};
	%\draw [myblue,ultra thick](5,1)--(5,-0.5) node [below] {$\tau=0.5$};
	%\draw [thick](5,2)--(5,-0.2) node [below] {$\tau=0.5$};
	%\draw [](1,0.2) -- (1,-0.2) node [below=0.1cm] {$0$};
	%\draw [](3,0.2) -- (3,-0.2) node [above left=0.1cm] {};
	%\draw [](5,0.2) -- (5,-0.2) node [below right] {};
	%\draw [](7,0.2) -- (7,-0.2) node [below right] {};
	%\draw [thick](1,2) -- (1,-0.2);

	\node[inner sep=0pt,anchor=west] (whitehead) at (0,0)
	{\includegraphics[width=0.4\linewidth]{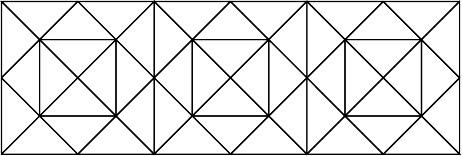}};
	\node[inner sep=0pt,anchor=east] (whitehead) at (0.2\linewidth,0)
	{\includegraphics[width=0.2\linewidth]{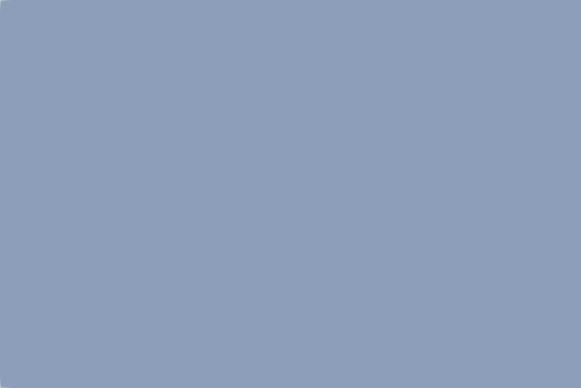}};
	
	\node[inner sep=0pt,anchor=west] (whitehead) at (0,-1*\d)
	{\includegraphics[width=0.4\linewidth]{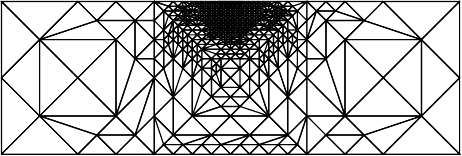}};
	\node[inner sep=0pt,anchor=east] (whitehead) at (0.2\linewidth,-1*\d)
	{\includegraphics[width=0.2\linewidth]{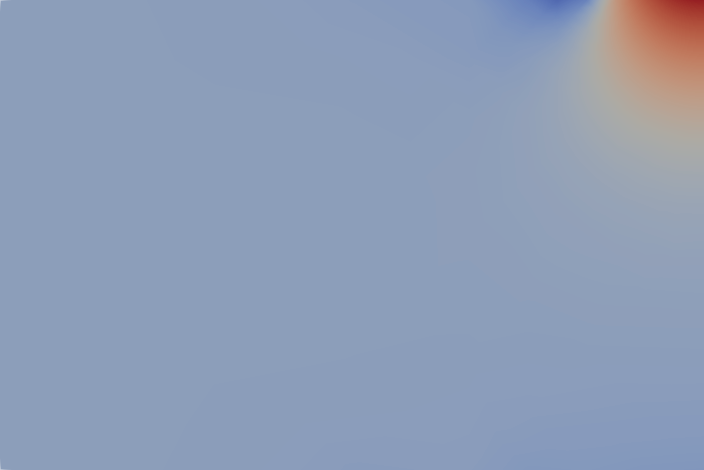}};
	\node[inner sep=0pt,anchor=west] (whitehead) at (0,-2*\d)
	{\includegraphics[width=0.4\linewidth]{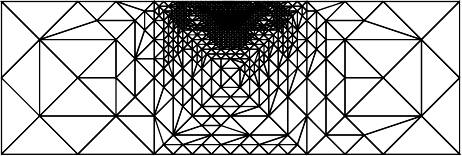}};
	\node[inner sep=0pt,anchor=east] (whitehead) at (0.2\linewidth,-2*\d)
	{\includegraphics[width=0.2\linewidth]{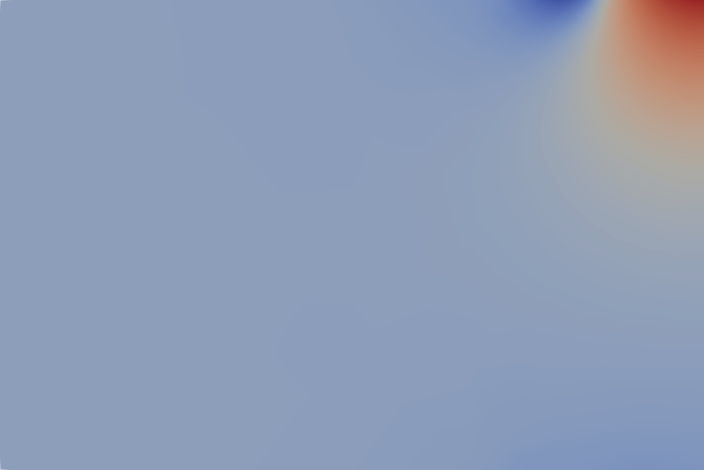}};
	\node[inner sep=0pt,anchor=west] (whitehead) at (0,-3*\d)
	{\includegraphics[width=0.4\linewidth]{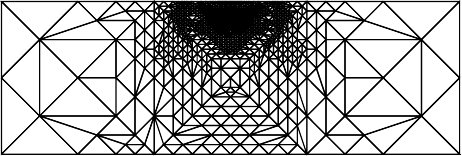}};
	\node[inner sep=0pt,anchor=east] (whitehead) at (0.2\linewidth,-3*\d)
	{\includegraphics[width=0.2\linewidth]{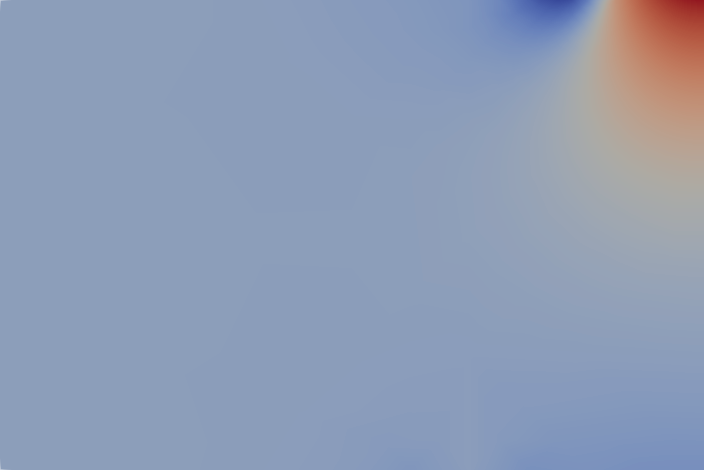}};
	\node[inner sep=0pt,anchor=west] (whitehead) at (0,-4*\d)
	{\includegraphics[width=0.4\linewidth]{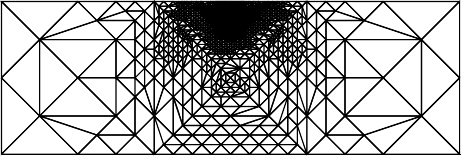}};
	\node[inner sep=0pt,anchor=east] (whitehead) at (0.2\linewidth,-4*\d)
	{\includegraphics[width=0.2\linewidth]{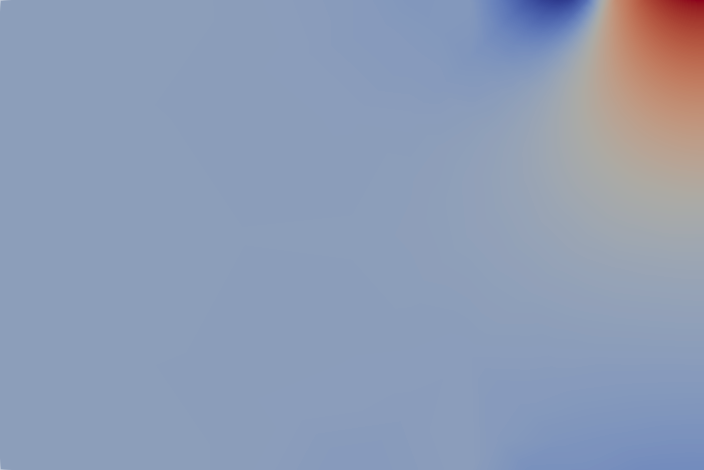}};
	\node[inner sep=0pt,anchor=west] (whitehead) at (0,-6*\d)
	{\includegraphics[width=0.4\linewidth]{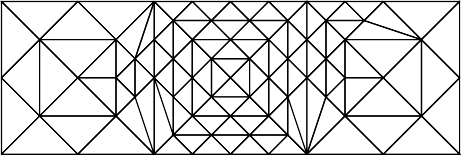}};
	\node[inner sep=0pt,anchor=east] (whitehead) at (0.2\linewidth,-6*\d)
	{\includegraphics[width=0.2\linewidth]{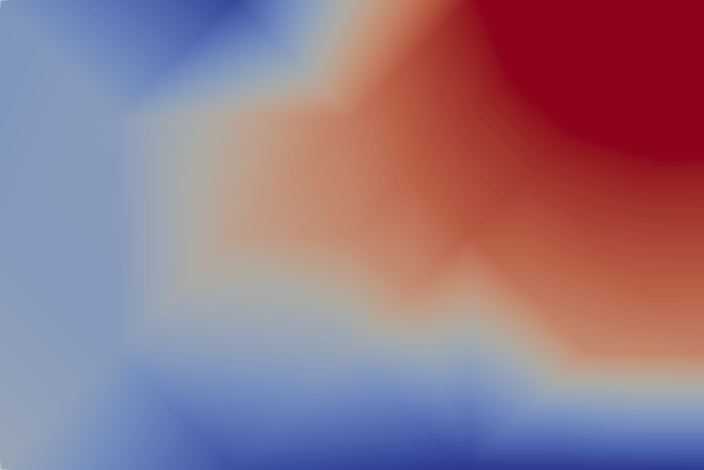}};
	\node[inner sep=0pt,anchor=west] (whitehead) at (0,-7*\d)
	{\includegraphics[width=0.4\linewidth]{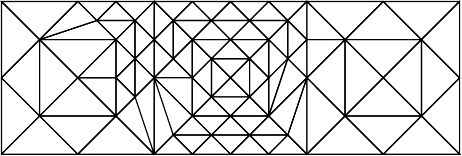}};
	\node[inner sep=0pt,anchor=east] (whitehead) at (0.2\linewidth,-7*\d)
	{\includegraphics[width=0.2\linewidth]{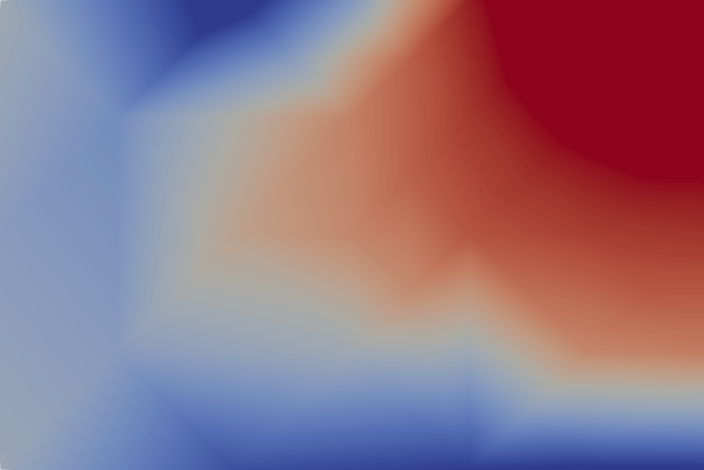}};

	\node[inner sep=0pt,anchor=east] (whitehead) at (\linewidth,0)
	{\includegraphics[width=0.4\linewidth]{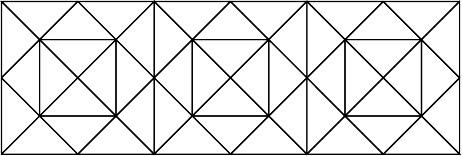}};
	\node[inner sep=0pt,anchor=west] (whitehead) at (0.8\linewidth,0)
	{\includegraphics[width=0.2\linewidth]{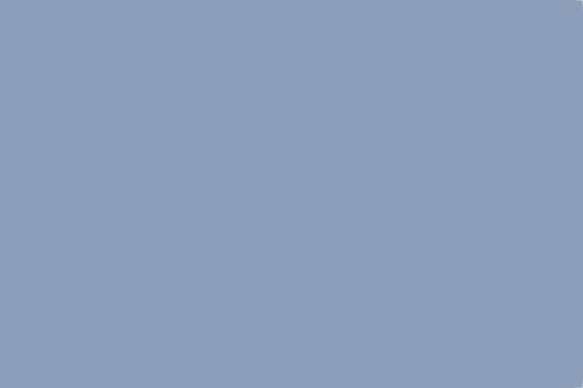}};
	\node[inner sep=0pt,anchor=east] (whitehead) at (\linewidth,-1*\d)
	{\includegraphics[width=0.4\linewidth]{figures/quasilin_bound/T0-4}};
	\node[inner sep=0pt,anchor=west] (whitehead) at (0.8\linewidth,-1*\d)
	{\includegraphics[width=0.2\linewidth]{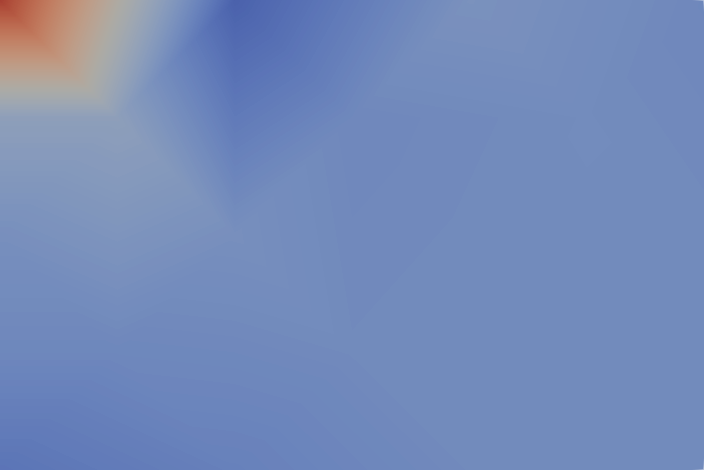}};
	
	\node[inner sep=0pt,anchor=east] (whitehead) at (\linewidth,-2*\d)
	{\includegraphics[width=0.4\linewidth]{figures/quasilin_bound/T0-4}};
	\node[inner sep=0pt,anchor=west] (whitehead) at (0.8\linewidth,-2*\d)
	{\includegraphics[width=0.2\linewidth]{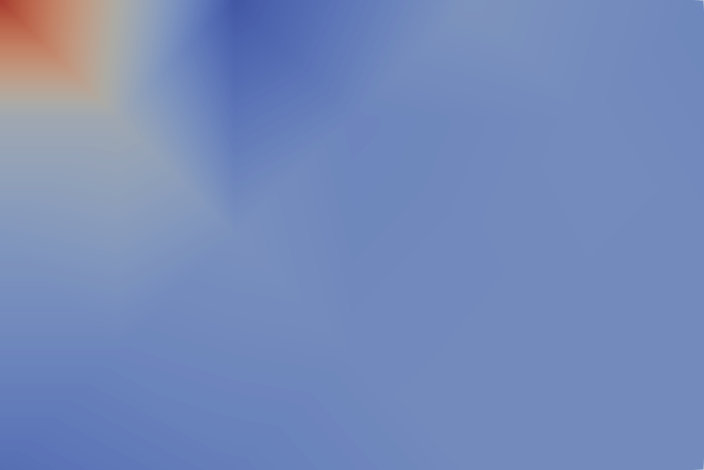}};
	\node[inner sep=0pt,anchor=east] (whitehead) at (\linewidth,-3*\d)
	{\includegraphics[width=0.4\linewidth]{figures/quasilin_bound/T0-4}};
	\node[inner sep=0pt,anchor=west] (whitehead) at (0.8\linewidth,-3*\d)
	{\includegraphics[width=0.2\linewidth]{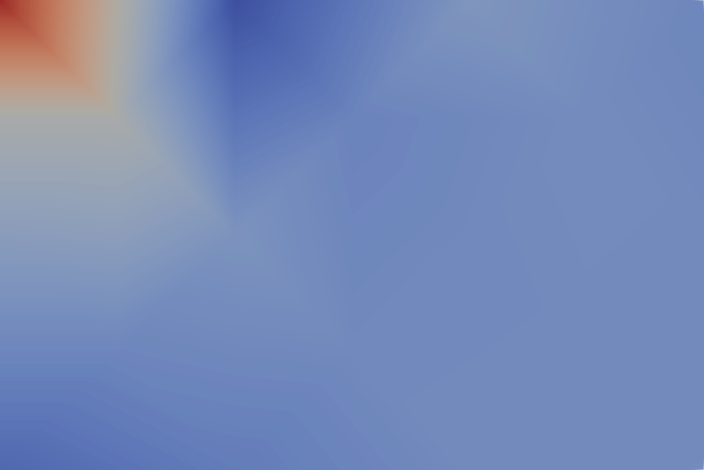}};
	\node[inner sep=0pt,anchor=east] (whitehead) at (\linewidth,-4*\d)
	{\includegraphics[width=0.4\linewidth]{figures/quasilin_bound/T0-4}};
	\node[inner sep=0pt,anchor=west] (whitehead) at (0.8\linewidth,-4*\d)
	{\includegraphics[width=0.2\linewidth]{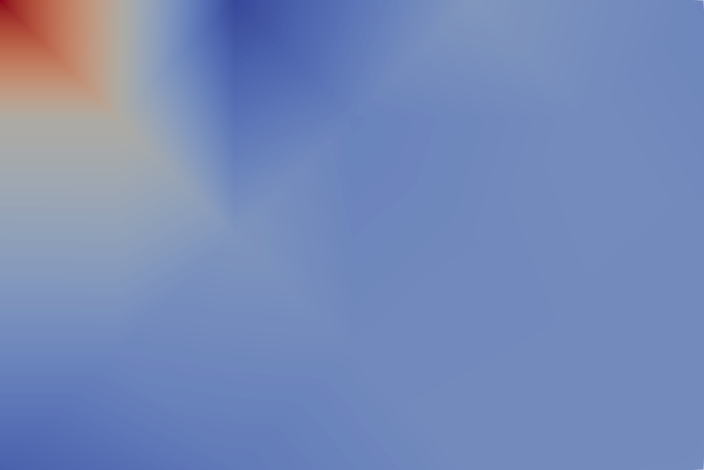}};
	\node[inner sep=0pt,anchor=east] (whitehead) at (\linewidth,-6*\d)
	{\includegraphics[width=0.4\linewidth]{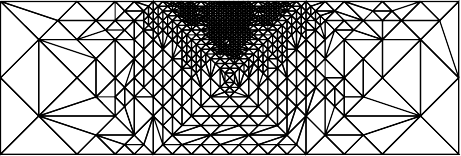}};
	\node[inner sep=0pt,anchor=west] (whitehead) at (0.8\linewidth,-6*\d)
	{\includegraphics[width=0.2\linewidth]{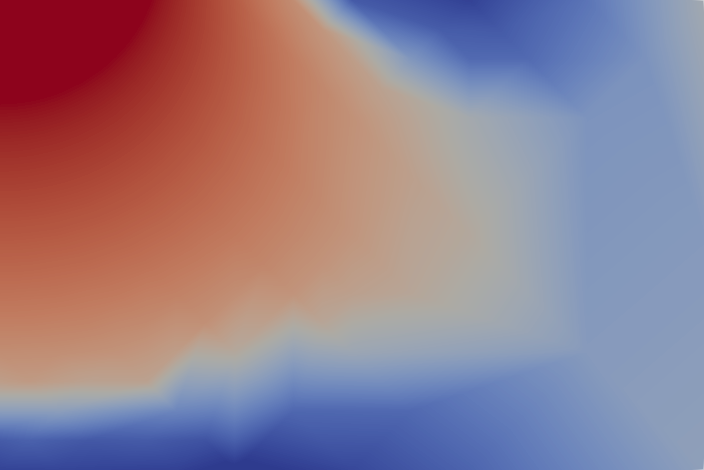}};
	\node[inner sep=0pt,anchor=east] (whitehead) at (\linewidth,-7*\d)
	{\includegraphics[width=0.4\linewidth]{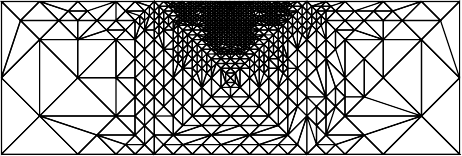}};
	\node[inner sep=0pt,anchor=west] (whitehead) at (0.8\linewidth,-7*\d)
	{\includegraphics[width=0.2\linewidth]{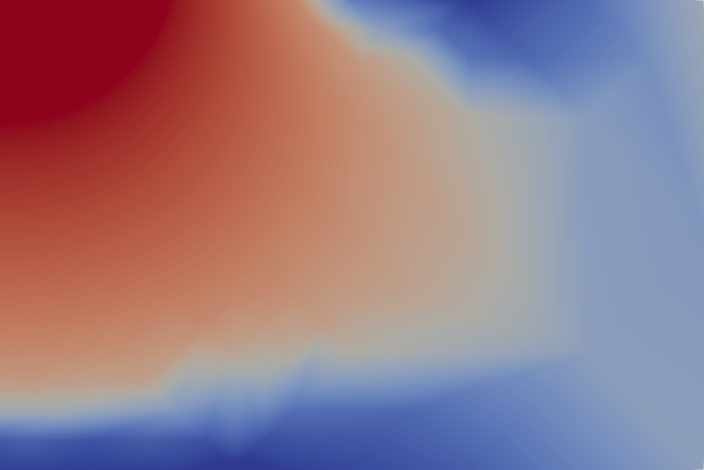}};
	\end{tikzpicture}
	\caption[Evolution of adaptively refined space grids for a quasilinear problem]{Evolution of adaptively refined space grids for the boundary control of a quasilinear equation refined for QOI $I^\tau(x,u)$ (left) and $J(x,u)$ (right) with 5934 and 5924 total spatial DOFs, respectively.}
	\label{fig:spacegrids_quasilin}
\end{figure}
Finally we show in \cref{fig:fvals_quasilin} the corresponding closed-loop cost of the MPC trajectory. Similar to the linear quadratic example with exponentially increasing reference in \cref{fig:fvals}, an increasing number of space grid points does not improve the MPC performance when refining with $J(x,u)$. This is due to the fact that the error indicators and thus also the refinements are predominant towards $T$ and not on the MPC implementation horizon. Thus, a refinement with the QOI $I^\tau(x,u)$ yields a significantly better controller performance.
%\FloatBarrier
\begin{figure}%[H]
	\centering
	\scalebox{.9}{% This file was created by matlab2tikz.
\definecolor{mycolor1}{rgb}{0.00000,0.44700,0.74100}%
\begin{tikzpicture}

\begin{axis}[%
width=2.5in,
height=0.8in,
at={(0in,1.5in)},
scale only axis,
xlabel style={font=\color{white!15!black}},
xlabel={number of space grid points},
%ymode=log,
%ymin=1,
%ymax=1,
%ytick={        1,     10,       100,      1000,     10000},
%yminorticks=true,
ylabel style={font=\color{white!15!black}},
ylabel={closed-loop cost},
axis background/.style={fill=white},
title style={align=center,at={(2.6in,1)}},
legend style={at={(1,0.8)}, draw=white!15!black}
]
\addplot [color=black, dashdotted, line width=2.0pt, mark=*, mark options={solid, black}]
  table[row sep=crcr]{%
1353 89.17659939568\\
1625 89.1765989186191\\
2710 89.1765997259963\\
4929 89.176599920865\\
6250 89.1765903640215\\
};

\addlegendentry{refined with QOI $J(x,u)$}

\addplot [color=black, dashdotted, line width=2.0pt, mark size=5.0pt, mark=x, mark options={solid, black}]
  table[row sep=crcr]{%
1353 89.17659939568\\
1675 75.558674783273\\
2845 75.489945214135\\
4791 74.6925986000899\\
5932 74.6168870678624\\
};

\addlegendentry{refined with QOI $I^\tau(x,u)$}

\end{axis}

\end{tikzpicture}%
}
	\caption[Comparison of MPC closed-loop cost for space adaptivity with quasilinear dynamics]{Comparison of cost functional values of the MPC closed-loop trajectory for different QOIs used for spatial refinement with quasilinear problem.}
	\label{fig:fvals_quasilin}
\end{figure}
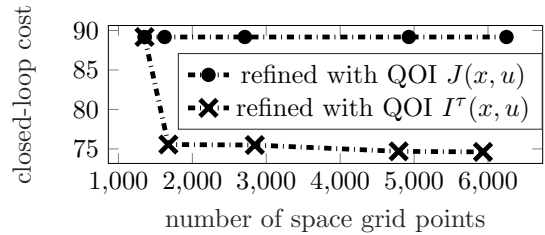
%\FloatBarrier
\section{Conclusion and outlook}
\label{sec:numerics:outlook}
We have shown in this work how goal oriented error estimation can be used to efficiently solve the optimal control problems occurring in a Model Predictive Controller. To this end, we proved estimates on the continuous-time and discrete-time error indicators showing that they decay exponentially outside the support of the QOI used for refinement. Moreover we presented two particular examples and illustrated the performance of a truncated cost functional as QOI for adaptive MPC in terms of closed-loop cost. We showed in various examples the efficiency of this approach, i.e., for a fixed number of total degrees of freedom, the distribution of the grid points induced by a localized cost functional as QOI leads to a significant reduction in the closed-loop cost in comparison to using the full cost function as QOI.

We conclude with several research perspectives. A straightforward adaption of the approach presented in this paper to hyperbolic problems can be considered, cf.\ \cite{Kroener2011}. Further, one could utilize model order reduction combined with grid adaptivity to obtain fast MPC methods. To this end, on the one hand, we refer to recent works combining grid adaptivity and proper orthogonal decomposition \cite{Graessle2019,Graessle2018}. On the other hand, there are several recent works employing proper orthogonal decomposition in an MPC context, cf.\ \cite{Gruene2019b,Mechelli2017}. In that context, the turnpike property could turn out useful as it reveals a lot of structure of the dynamic problem. In case of a steady state turnpike, a reduced basis for the corresponding elliptic steady-state OCP might help in constructing a basis for the time-dependent problem.

\bibliographystyle{abbrv}
\bibliography{references.bib}
\end{document}